\documentclass[envcountsame,envcountchap]{svmono}

\usepackage{makeidx}         
\usepackage{graphicx}        
\usepackage{multicol}        
\usepackage[bottom]{footmisc}
\usepackage{amssymb,amsfonts,euscript}         

\makeindex             

\begin{document}

\author{Anna Karczewska}
\title{Convolution type stochastic Volterra equations  
}
\subtitle{-- Monograph --}
\maketitle

\frontmatter

\preface

This volume is the habilitation dissertation of the author written at the 
Faculty of Mathematics, Computer Science and Econometrics of the 
University of Zielona G\'ora. 

The aim of this work is to present, in self-contained form, results 
concerning fundamental and the most important questions related to 
linear stochastic Volterra equations of convolution type. 
The paper is devoted to study the existence and some kind of regularity of
solutions to stochastic Volterra equations in Hilbert space and the space of
tempered distributions, as well.

In recent years the theory of Volterra equations, particularly fractional ones,
has undergone a big development. 
This is an emerging area of research with
interesting mathematical questions and various important applications.
The increasing interest in these equations 
comes from their applications to problems from physics and engeenering,
particularly from viscoelasticity, heat conduction in materials with memory 
or electrodynamics with memory.

The paper is divided into four chapters. The first two of them have 
an introductory 
character and provide deterministic and stochastic tools needed to study 
existence of solutions to the equations considered and
their regularity. Chapter \ref{SVEHSch:1} is devoted to stochastic 
Volterra equations in a separable Hilbert space. 
In chapter \ref{SEEMDch:4} stochastic linear evolution equations in the space of
distributions are studied.

The work is based on some earlier papers of the author, however part of results 
is not yet published.

I wish to express my sincere gratitude to prof. Jerzy Zabczyk for 
introducing me into the world of stochastic Volterra equations and  
inspiring mathematical discussions. I am particularly grateful  
to prof. Carlos Lizama for fruitful joint research, exclusively through
internet connection, which results is a big part of this monograph.

\vspace{5mm}
\begin{flushright}\noindent
Zielona G\'ora, June 2007\hfill {\it Anna Karczewska}\\
\end{flushright}

\tableofcontents

\mainmatter

\chapter*{Introduction}\label{INTROch:0} 
\addcontentsline{toc}{chapter}{Introduction}

\section*{The main results in brief}\label{sIN1}

In the paper, two general problems concerning linear stochastic
evolution equations of convolution type are studied: existence of strong 
solutions to such stochastic Volterra equations in a Hilbert space
and regularity of solutions to two classes of stochastic Volterra equations 
in spaces of distributions.

First, we consider
Volterra equations in a separable Hilbert space $H$ of the form
\begin{equation}\label{e0.1}
 X(t) = X(0) + \int_0^t a(t-\tau)\, A\, X(\tau)\,d\tau 
	  + \int_0^t \Psi(\tau)\,dW(\tau), \quad t\geq 0,
\end{equation}
where  $X(0)\in H$, $a\in L_\mathrm{loc}^1(\mathbb{R}_+)$ is a scalar
kernel function and $A$ is a
closed linear unbounded operator with the dense domain $D(A)$ equipped with the
graph norm. In (\ref{e0.1}), $\Psi(t)$, $t\ge 0$ is an appropriate 
stochastic process and $W(t)$, $t\ge 0$ is a cylindrical Wiener process;
both processes are  defined on a stochastic basis $(\Omega, \mathcal{F}, 
(\mathcal{F}_t)_{t\ge 0},P)$.

Equation (\ref{e0.1}) arises, in the deterministic case, in a
variety of applications as model problems, see e.g.\ \cite{Pr93}, 
\cite{Ba01} and references therein. Well-known techniques
like localization, perturbation and coordinate transformation
allow to transfer results for such problems to 
integro-differential equations.  In these applications, the
operator $A$ typically is a differential operator acting in
spatial variables, like the Laplacian, the Stokes operator, or the
elasticity operator. The kernel function $a(t)$ should be thought as a
kernel like $ a(t) = e^{-\eta t} t^{\beta -1}/\Gamma (\beta); \eta
\geq 0, \beta \in (0,2)$. 
The stochastic approach to integral equations has been recently
used due to the fact that the level of accuracy for a given model
not always seems to be significantly changed with increasing model
complexity. 

Our main results concerning (\ref{e0.1}), rely essentially on
techniques using a strongly continuous family of operators $S(t),
t\geq 0$, defined on the space $H$ and called the {\tt
resolvent}.  
Hence, in what follows, we assume that the deterministic version of equation
(\ref{e0.1}) is {\tt well-posed}, that is, admits a resolvent
$S(t), t\geq 0$. 

The stochastic Volterra equations of the form (\ref{e0.1})
have been treated by many
authors, see e.g. \cite{CD96}, \cite{CD97}, \cite{CDP97},
\cite{CD00} or \cite{RS00}, \cite{RS01}. In the first
three papers stochastic Volterra equations are studied in connection with
viscoelasticity and heat conduction in materials with memory.
The paper due to Cl\'ement and Da Prato \cite{CD96} is 
particularly significant
because the authors have extended the well-known semigroup approach, applied
to stochastic differential equations, to a subclass of the
equation (\ref{e0.1}). 
In the next papers, weak and mild solutions to the equation (\ref{e0.1})
have been studied and some results like regularity of solutions or large
deviations of equations have been given. The resolvent approach to
stochastic Volterra equations, introduced in \cite{CD96}, enables us to obtain
new results in an elegant way, analogously like in semigroup case.
In resolvent case, new difficulties appear because the family $S(t)$,
$t\ge 0$, in general do not create a semigroup.

Our main results concerning the equations (\ref{e0.1}) in the space $H$ 
are the existence theorems of strong
solutions to some classes of such equations. 
We provide existence of strong solutions to (\ref{e0.1}) under different
conditions on the kernel function. In some cases, we arrive at stochastic
versions of fractional Volterra equations with corresponding $\alpha$-times
resolvent families $S_\alpha(t)$, $t\ge 0$.
The key role in our proofs is played by convergence of resolvent or 
$\alpha$-times resolvent families corresponding to deterministic
versions of Volterra equations. These convergence theorems are resolvent
analogies of the well-known Hille--Yosida theorem in semigroup case. 
Our convergence
results generalize theorems due to Cl\'ement and Nohel \cite{CN79} obtained for
contraction semigroups.
Having such resolvent analogies of the
Hille--Yosida theorem, we proved that the stochastic convolutions arising
in Volterra equations (\ref{e0.1}) are strong solutions to (\ref{e0.1}). 
\\

In the remaining part of the paper we study two classes of equations of
convolution type: the equation 
\begin{equation}\label{e0.2}
X(t,\theta)=X_0(\theta)+\int_0^tb(t-\tau)AX(\tau,\theta)d\tau
+ W_\Gamma (t,\theta),\quad t\ge0,\quad \theta\in \mathbb{R}^d,
\end{equation}
and the integro-differential stochastic equation with infinite delay
\begin{equation}{\label{e0.3}}
X(t, \theta) = \int_{-\infty}^t b(t-s)[ \Delta X(s, \theta) +
\dot{W}_{\Gamma} (s,\theta)] ds, \quad t
\geq 0, \quad \theta \in T^d,
\end{equation}
where $T^d$ is a $d$-dimensional torus. 
\index{$b(t)$}
The kernel function $b$ is integrable on $\mathbb{R}_+$ and the class of
operators $A$ contains the Laplace operator and its fractional powers.
In both equations (\ref{e0.2}) and 
(\ref{e0.3}), $W_\Gamma$ denotes a spatially homogeneous Wiener process, which
takes values in the space of tempered distributions $S'(\mathbb{R}^d)$, 
$d\ge 1$.

Equations (\ref{e0.2}) and (\ref{e0.3}) are generalizations of stochastic heat
and wave equations studied by many authors. Particularly, regularity problems of
these equations have attracted many authors. 
For an exhaustive bibliography we refer to \cite{Le01}.

In the paper, we consider existence of the solutions 
to (\ref{e0.2}) and (\ref{e0.3})
in the space $S'(\mathbb{R}^d)$ and next we derive conditions under which the
solutions to (\ref{e0.2}) and (\ref{e0.3}) take values in function spaces.

In the case of equation (\ref{e0.2}), the results have been obtained by using
the resolvent operators corresponding to Volterra equations. The regularity
results have been expressed in terms of the spectral measure $\mu$ and the 
covariance kernel $\Gamma$ of the Wiener process $W_\Gamma$. Moreover, we give
necessary and sufficient conditions for the existence of a limit measure to the
equation  (\ref{e0.2}).

In the case of the equation (\ref{e0.3}), we study a particular case of weak
solutions under the basis of an explicit representation of the solution to 
(\ref{e0.3}). The regularity results have been expressed in terms of the Fourier
coefficients of the space covariance $\Gamma$ of the process $W_\Gamma$.\\

\section*{A guided tour through the paper}

Chapter \ref{DVEch:1} has a preliminary character. Its goal is to introduce the
reader to the theory of deterministic Volterra equations in Banach space 
and to provide facts used in the paper. 
Section \ref{DVEsec:1} gives notations used in the paper and section
\ref{DVEsec:2} gives basic definitions connected with resolvent operators.
Sections \ref{DVEsec:3} and \ref{DVEsec:4} contain definitions and facts
concerning kernel functions, particularly regular ones for parabolic Volterra
equations. Some ideas are illustrated by examples.
Section \ref{DVEsec:5} provides new results due to Karczewska and Lizama
\cite{KL07c}, that is, the approximation theorems, Theorems \ref{DVEt1} 
and \ref{DVEt1a}, not yet published. 
These results are resolvent analogies of the Hille-Yosida
theorem in semigroup case and play the same role like the Hille-Yosida 
theorem does. These results 
concerning convergence of resolvents for the deterministic version of
the equation (\ref{e0.1}) in Banach space play the key role for existence
theorems for strong solutions and they are used in Chapter \ref{SVEHSch:1}.

Chapter \ref{PBch:2} contains concepts and results from the infinite
dimensional stochastic analysis recalled from well-known monographs \cite{CP78},
\cite{DZ92} and \cite{GT95}. Among others, we recall an infinite dimensional
versions of the Fubini theorem and the It\^{o} formula. Additionally, we recall a
construction, published in \cite{Ka98}, of stochastic integral with respect to
cylindrical Wiener process. The construction bases on
the Ichikawa idea for the stochastic integral with respect to classical 
infinite dimensional Wiener process and it is an alternative to the
construction given in \cite{DZ92}.

Chapter \ref{SVEHSch:1} contains the main results for stochastic Volterra 
equations in Hilbert space. In Section \ref{SVEHSsec:1} we introduce the
definitions of solutions to the equation (\ref{e0.1}) and formulate 
auxiliary results being
a framework for the main theorems. In Section \ref{SVEHSsec:2}
we prove existence of strong solutions for two classes of equation (\ref{e0.1}).
Basing on convergence of resolvents obtained in Chapter \ref{DVEch:1}, 
we can formulate Lemma \ref{pSW5} and Theorem \ref{coSW4} giving sufficient
conditions under which stochastic convolution corresponding to (\ref{e0.1})
is strong solution to the equation (\ref{e0.1}). Section \ref{SVEHSsec:3}
deals with the so-called fractional Volterra equations.
First, we prove using other tools than in Section \ref{DVEsec:5}, approximation
results,that is, Theorems \ref{th2f} and  \ref{th3af},
for $\alpha$-times resolvents corresponding to the fractional equations.
Next, we prove the existence of strong solutions to fractional 
equations. These results are formulated in Lemma \ref{pSW5f} and  
Theorem \ref{coSW4f}.
In Section \ref{c3s4Ex} we give several examples illustrating the class of
equations fulfilling conditions of theorems providing existence of strong
solutions.

In Chapter \ref{SEEMDch:4} we study regularity of two classes of stochastic
Volterra equations in the space of tempered distributions. Section \ref{PizSS2} 
has an introductory character. It contais notions and facts concerning
generalized and classical homogeneous Gaussian random fields needed in the sequel.
In Section \ref{SEEMTDsec:2} we introduce the stochastic integral in the space
of distributions. Then we formulate Theorem \ref{PizThDawson} which
characterizes the stochastic convolution corresponding to (\ref{e0.2}).
The main regularity results obtained in Section \ref{SEEMTDsec:2} are collected
in Theorems \ref{PizTh2}, \ref{PizTh3} and \ref{PizTh4}. These theorems give
sufficient conditions under which solutions to the equation (\ref{e0.2}) are
function-valued and even continuous with respect 
to the space variable. These conditions are
given in terms of the covariance kernel $\Gamma$ of the Wiener process 
$W_\Gamma$ and the spectral measure of $W_\Gamma$, as well. The results obtained
in this section are illustrated by several examples.
Section \ref{SEEMTDsec:3} is a natural continuation of the previous one. In 
this section we give necessary and sufficient conditions for the existence of a
limit measure to the equation (\ref{e0.2}) and then we describe all limit
measures to  (\ref{e0.2}). The main results of this section, that is Lemmas 
\ref{LiLl1} - \ref{LiLl2} and Theorems \ref{LiTh1} - \ref{LiTh2}, are in a sense 
analogous to those formulated in \cite[Chapter 6]{DZ96}, obtained for semigroup
case. 

Section \ref{SEEMTDsec:4} is devoted to regularity of solutions to the equation
(\ref{e0.3}). Here we study a particular case of weak solutions basing on an
explicit representation of the solution to (\ref{e0.3}). 
We find the expression for the solution in terms of the kernel $b$ and
next we reduce the questions of regularity of solutions  to
problems arising in harmonic analysis. Our main results of this section, that is
Theorem \ref{th3.4} and Proposition \ref{pr4.21}, provide necessary and sufficient
conditions under which solutions to (\ref{e0.3}) are function-valued. These
conditions are given in terms of the Fourier coefficients of the covariance
$\Gamma$ of the Wiener process $W_\Gamma$. Additionally, Section  
\ref{SEEMTDsec:4} contains some corollaries which are consequences of the main 
results.

\pagebreak

\section*{Bibliographical notes}\label{sBiblN}

Sections \ref{DVEsec:1}--\ref{DVEsec:4} contain introductory material which in a
similar form can be found in the mongraph \cite{Pr93}. Section \ref{DVEsec:5}
originates from \cite{KL07b} and \cite{KL07c}, the latter not yet published. \\

Sections \ref{PBsec:1}, \ref{PBsec:3} and \ref{PBs:4} contain classical
results recalled from \cite{CP78}, \cite{Ic82}, \cite{DZ92} and \cite{GT95}.
Section \ref{PBs:3} originates from \cite{Ka98}.\\

The results of Section \ref{SVEHSsec:1} come from \cite{Ka05}.
Section \ref{SVEHSsec:2} originates from \cite{KL07b} and \cite{KL07c}.
The content of Sections \ref{SVEHSsec:3} and \ref{c3s4Ex} can be found in
\cite{KL07d}.\\

Section \ref{PizSS2} contains material coming from \cite{GS64}, \cite{GV64},
\cite{Ad81} and \cite{PZ97}. The results of Section \ref{SEEMTDsec:2} come from
\cite{KZ00a}. Section \ref{SEEMTDsec:3} originates from \cite{Ka03}. 
The results of Section \ref{SEEMTDsec:4} can be found in \cite{KL07a}.

\chapter{Deterministic Volterra equations} \label{DVEch:1} 

This chapter contains  notations and concepts concerning 
Volterra equations used
throughout the monograph and collects some results necessary to make the work
self-contained.
The notations are standard and follow the book by Pr\"uss \cite{Pr93}.

Section \ref{DVEsec:1} has a preliminary character. 
In section \ref{DVEsec:2} we recall the definition of the resolvent family to
the deterministic Volterra equation and the concept of well-posedness connected
with the resolvent. Kernel functions, particularly $k$-regular ones, are
recalled in sections \ref{DVEsec:3} and \ref{DVEsec:4}.  The above mentioned
definitions and results are described and commented in detail in Pr\"uss'
monograph \cite{Pr93} and appropriate references therein.

Section \ref{DVEsec:5} originates from \cite{KL07c}. Theorems \ref{DVEt1} and 
\ref{DVEt1a}, yet non-published, are deterministic approximation theorems
which play a key role for existence of strong solutions to stochastic Volterra
equations. These approximation theorems are resolvent analogies to 
Hille-Yosida's theorem for semigroup and play the same role as that theorem.

\section{Notations and preliminaries} \label{DVEsec:1}

Let $B$\index{$B$}
be a complex Banach space\index{space!Banach}
 with the norm $|\cdot|$. We consider in $B$ 
the Volterra equation of the form
\begin{equation} \label{DVEe1}
  u(t) = \int_0^t a(t-\tau)\,Au(\tau) d\tau + f(t)\,.
\end{equation}
In (\ref{DVEe1}), 
 $a\in L_{\mathrm{loc}}^1(\mathbb{R}_+;\mathbb{R})$
 \index{$L_{\mathrm{loc}}^1(\mathbb{R}_+;\mathbb{R})$} 
 is a non-zero scalar kernel;
for abbreviation we will write  $a\in L_{\mathrm{loc}}^1(\mathbb{R}_+)$.
\index{$L_{\mathrm{loc}}^1(\mathbb{R}_+)$} $A$
\index{$A$}
\index{$a(t)$}
is a closed unbounded linear operator in $B$ with a dense domain 
$D(A)$\index{$D(A)$} and
$f$ is a continuous $B$-valued function. In the sequel we assume that the domain
$D(A)$ is equipped with the graph norm $|\cdot|_A$ of $A$, i.e.\ 
$|x|_A := |x|+|Ax|$ for $x\in D(A)$. Then $(D(A),|\cdot|_A)$ is a Banach space
because $A$ is closed (see e.g.\ \cite{EN00}) and it is continuously and 
densely embedded into $(B,|\cdot|)$.
\index{$\mid\cdot\mid_A$}

By $\sigma(A)$\index{$\sigma(A)$} and $\varrho(A)$  
we shall denote spectrum and resolvent set of
the operator $A$, respectively.
\index{$\varrho(A)$}

The equation (\ref{DVEe1}) includes a big class of equations and is an abstract
version of several deterministic problems, see e.g.\  \cite{Pr93}.
For example, if $a(t)=1$ and $f$ is a function of $C^1$-class, the equation 
(\ref{DVEe1}) is equivalent to the Cauchy problem 
$$ \dot{u}(t) = Au(t) + \dot{f}(t) \quad \mbox{with} \quad u(0)=f(0).$$
Analogously, in the case $a(t)=t\,$ and $f$ of $C^2$-class, the equation 
(\ref{DVEe1}) is equivalent to 
$$ \ddot{u}(t) = Au(t) + \ddot{f}(t) \quad \mbox{with initial conditions} \quad
 u(0)=f(0) \quad \mbox{and} \quad \dot{u}(0) = \dot{f}(0).$$
Several other examples of problems which lead to Volterra equation (\ref{DVEe1}) 
can be found in \cite[Section 5]{Pr93}.

In the first three chapters of the monograph we shall use the abbreviation
$$ (g\star h)(t) = \int_0^t g(t-\tau)h(\tau)d\tau, \quad t\ge 0, $$
for the convolution
\index{$g\star h$}
\index{convolution}
 of two functions $g$ and $h$.

In the paper we write $a(t)$ for the kernel function $a$. The notation $a(t)$
will mean the function and not the value of the function $a$ at $t$.
Such notation will allow to distinguish the function $a(t)$ and the article a.
\index{$a(t)$}

If a function $v\in L_\mathrm{loc}^1(\mathbb{R}_+;B)$ is of {\tt 
exponential growth}, i.e. \linebreak
 $\int_0^\infty e^{-\omega t} |v(t)|dt<+\infty$ for some $\omega\in\mathbb{R}$,
we can define  the {\tt Laplace transform} of the function~$v$ 
$$
  \hat{v} (\lambda) = \int_0^\infty e^{-\lambda t} v(t) dt, \qquad \mathrm{Re}
  \lambda \ge \omega.
$$
In the whole paper we shall denote by $\hat{v}$ the Laplace transform 
of the function $v$. \index{$\hat{v} (\lambda)$}

In the whole paper the operator norm will be denoted by $||\cdot ||$.
\index{$\parallel\cdot\parallel$}

\section{Resolvents and well-posedness} \label{DVEsec:2}

The concept of the resolvent is very important for the theory of linear 
Volterra equations. 
The so-called resolvent approach to the Volterra equation (\ref{DVEe1}) has
been introduced many years ago, probably by Friedman and Shinbrot \cite{FS67};  
recently the approach has been presented in detail in the great monograph 
by Pr\"uss \cite{Pr93}. 
The resolvent approach is a generalization of the semigroup approach.

By $S(t),~t\geq 0$, we shall denote the family of resolvent 
operators corresponding to
the Volterra equation (\ref{DVEe1}) and defined as follows.
\index{$S(t)$}

\begin{definition} \label{DVEd1} (see e.g.\ \cite{Pr93})\\
A family $(S(t))_{t\geq 0}$ of bounded linear operators in the space $B$ is
called {\tt resolvent} for (\ref{DVEe1}) if the following conditions are
satisfied:
\begin{enumerate}
\item $S(t)$ is strongly continuous on $\mathbb{R}_+$ and $S(0)=I$;
\item $S(t)$ commutes with the operator $A$, that is, $S(t)(D(A))\subset D(A)$
 and $AS(t)x=S(t)Ax$ for all $x\in D(A)$ and $t\geq 0$;
 \item the following {\tt resolvent equation} holds
\begin{equation} \label{DVEe2}
 S(t)x = x + \int_0^t a(t-\tau) AS(\tau)x d\tau 
\end{equation}
for all $x\in D(A),~t\geq 0$.
\end{enumerate}
\end{definition}

We shall assume that the equation (\ref{DVEe1}) is 
{\tt well-posed}\index{equation!well-posed} in the sense
that (\ref{DVEe1}) admits the resolvent $S(t),~t\geq 0$. (Precise definition of
well-posedness is given in \cite{Pr93}). That defintion is a direct extension of
well-posedness of Cauchy problems. The lack of well-posedness of (\ref{DVEe1})
leads to distribution resolvents, see e.g.\ \cite{DI84}.
\begin{proposition} \label{pr1Pr} (\cite[Proposition 1.1]{Pr93})
The equation (\ref{DVEe1}) is well-posed if and only if (\ref{DVEe1}) admits 
a resolvent $S(t)$. If this is the case then, in addition, \linebreak
the range $R(a\star S(t))\subset D(A)$ for all $t\ge 0$ and
\begin{equation} \label{deq2Pr}
 S(t)x = x + A \int_0^t a(t-\tau) S(\tau)x d\tau \quad \mbox{for~all} 
 \quad x\in H, ~t\ge 0.
\end{equation}
\end{proposition}
\noindent{\bf Comment} 
Let us emphasize that the resolvent $S(t),~t\geq 0$,
is determined by the operator $A$ and the function $a(t)$.
Moreover, as a consequence of the strong continuity of $S(t)$ we have
for any $T>0$
\begin{equation} \label{DVEe3}
 \sup_{t\leq T} ||S(t)||<+\infty \;.
\end{equation}

Suppose $S(t)$ is the  resolvent for (\ref{DVEe1}) and let 
$-\mu\in\sigma(A)$ be an eigenvalue of $A$ with eigenvector $x\neq 0$.
Then
\begin{equation} \label{DVEe4}
 S(t)x=s(t;\mu)x, \qquad t\geq 0,
\end{equation}
where $s(t;\mu)$ is the solution of the one-dimensional Volterra equation
\begin{equation} \label{DVEe5}
 s(t;\mu)+\mu \int_0^t a(t-\tau)s(\tau;\mu)d\tau = 1, \qquad t\geq 0.
\end{equation}
\index{$s(t;\mu)$}

By $W_\mathrm{loc}^{1,p}(\mathbb{R}_+;B)$ we denote the Sobolev space of
order $(1,p)$ of Bochner locally $p$-integrable functions acting from  
$\mathbb{R}_+$ into the space $B$, see e.g.\ \cite{EN00}.

\begin{definition}\label{DVEd2a}
A resolvent  $S(t)$, for the equation (\ref{DVEe1}),
is called 
{\tt diffe\-rentiable} if $S(\cdot )x\in
W_\mathrm{loc}^{1,1}(\mathbb{R}_+;B)$ for any $x\in D(A)$ and 
there \linebreak exists a function $\varphi \in L_\mathrm{loc}^1(\mathbb{R}_+)$ 
such that $|\dot{S}(t)x| \le \varphi (t)|x|_A$ a.e.\ on $\mathbb{R}_+$,
for every $x\in D(A)$.
\end{definition}
\index{resolvent!differentiable}
\index{$W_\mathrm{loc}^{1,1}(\mathbb{R}_+;B)$}

Similarly, if $S(t)$ is differentiable then
\begin{equation} \label{DVEe6}
 \dot{S}(t)x=\mu r(t;\mu)x, \qquad t\geq 0,
\end{equation}
where $r(t;\mu)$ is the solution of the one-dimensional equation
\begin{equation} \label{DVEe7}
 r(t;\mu)+\mu \int_0^t a(t-\tau)r(\tau;\mu)d\tau = a(t), \qquad t\geq 0.
\end{equation}
\index{$r(t;\mu)$}

In some special cases the functions $~s(t;\mu)~$ and $~r(t;\mu)~$ may be found
explicitely. For example, for $a(t)=1$, we have $s(t;\mu)=r(t;\mu)=e^{-\mu t}$,
$~t\ge 0, ~\mu\in\mathbb{C}$. For  $a(t)=t$, we obtain 
$s(t;\mu)=\cos (\sqrt{\mu}t)$, $~r(t;\mu)=\sin (\sqrt{\mu}t)/\sqrt{\mu}$,
$~t\ge 0, ~\mu\in\mathbb{C}$.

\begin{definition}\label{DVEd2}
Suppose $S(t),~t\geq 0$, is a resolvent for (\ref{DVEe1}). $S(t)$
is called {\tt exponentially bounded} if there are constants
$M\geq 1$ and $\omega\in\mathbb{R}$ such that
$$ ||S(t)|| \leq M\,e^{\omega t}, \mbox{~~for all~~} t\geq 0. $$
$(M,\omega)$ is called a {\tt type} of $S(t)$.
\end{definition}
\index{$(M,\omega)$ - type of S(t)}
\index{resolvent!exponentially bounded}

Let us note that in contrary to the case of semigroups, not every
resolvent needs to be exponentially bounded even if the kernel
function $a(t)$ belongs to $L^1(\mathbb{R}_+)$. The Volterra equation
version of the Hille--Yosida theorem (see e.g.\ \cite[Theorem
1.3]{Pr93}) provides the class of equations that admit
exponentially bounded resolvents. An important class of kernels providing
such class of resolvents are $a(t)=t^{\alpha-1}/\Gamma(\alpha), ~ \alpha\in
(0,2)$. For details, counterexamples and comments we refer to \cite{DP93}.
\vspace{-2mm}

\section{Kernel functions} \label{DVEsec:3}

Two classes of kernel functions defined below play a prominent role in the
theory of Volterra equations.

\begin{definition} \label{DVEd3a} 
A $C^\infty$-function $a:(0,+\infty)\rightarrow \mathbb{R}$ is called 
{\tt completely monotonic}\index{function!completely monotonic}
if $(-1)^n a^{(n)}(t)\ge 0$ for all $t>0$, $n\in \mathbb{N}$.
\end{definition}

\begin{definition} \label{DVEd3}
We say that function $a\in L^1(0,T)$ is 
{\tt completely positive}\index{function!completely positive}
on $[0,T]$ if for any $\mu\geq 0$, the solutions of the convolution
equations (\ref{DVEe5}) and (\ref{DVEe7})  
satisfy $s(t;\mu)\geq 0$ and $ r(t;\mu) \geq 0 $ on $[0,T]$, respectively.
\end{definition}

We recall that if $a\in L_\mathrm{loc}^1(\mathbb{R}_+)$ is completely positive,
then $s(t;\mu)$, the solution to (\ref{DVEe5}), is nonnegative and nonincreasing
for any $t\ge 0, ~\mu\ge 0$. In the consequence, one has  
$0\le s(t;\mu)\le 1$. This is a
special case of the result due to \cite{Fr63}.

Kernels with this property have been introduced by
Cl\'{e}ment and Nohel \cite{CN79}. We note that the class of
completely positive kernels appears quite naturally in applications,
particularly in the theory of
viscoelasticity. Several properties and examples of such kernels
appear in \cite[Section 4.2]{Pr93}.

\noindent{\bf Examples} 1.  
Let $\displaystyle a(t)=
\frac{t^{\alpha-1}}{\Gamma(\alpha)}$, $~\alpha >0$, where $\Gamma$ is the gamma
function. For $\alpha\in (0,1]$, function $a(t)$ is completely monotonic and 
completely positive.

2. Another example of completely positive function is $a(t)=e^{-t}, ~t\ge 0$.
An easy computation shows that then 
$s(t;\mu)=(1+\mu)^{-1}[1+\mu\,e^{-(1+\mu)t}]$, for $t,\mu > 0$.
\vspace{-1mm}

\section{Parabolic equations and regular kernels} \label{DVEsec:4}

This section is devoted to the so-called parabolic Volterra equations defined
by Pr\"uss \cite{Pr91}.

Let $B$ be a complex Banach space and 
$$ \sum (\omega,\theta):=\{ \lambda\in\mathbb{C} : 
|arg(\lambda -\omega)|<\theta \}.$$

\begin{definition}\label{d2.1Pr} (\cite[Definition 2.1]{Pr93})
 A resolvent\index{resolvent!analytic}
 $S(t)$ for (\ref{DVEe1}) is called {\tt analytic}, if the function
 $S(\cdot): \mathbb{R}_+ \rightarrow L(B)$ admits analytic extension to a sector
 $\sum (0,\theta_0)$ for some $0<\theta_0 <\pi/2$. An analytic resolvent $S(t)$ 
 is said to be of {\tt analyticity type}\index{analyticity type}
 $(\omega_0,\theta_0)$ if for each 
 $\theta <\theta_0$ and $\omega >\omega_0$ there is $M=M(\omega,\theta)$ such
 that
 \begin{equation}\label{eq2.2Pr}
  ||S(z)||\le M\,e^{\omega \mathrm{Re}z}, \qquad z\in \sum (0,\theta_0).
 \end{equation}
\end{definition}

\begin{corollary}\label{c2.1Pr} (\cite[Corollary 2.1]{Pr93})
Suppose $S(t)$ is an analytic resolvent for (\ref{DVEe1}) of analyticity type 
$(\omega_0,\theta_0)$. Then for each $\omega >\omega_0$ 
and $\theta < \theta_0$ there is 
$M=M(\omega,\theta)$ such that 
 \begin{equation}\label{eq2.3Pr}
  ||S^{(n)}(t)||\le M\,n! \,e^{\omega t(1+\alpha)}(\alpha t)^{-n}, 
  \quad t>0,~~ n\in\mathbb{N},
 \end{equation}
where $\alpha=\sin \theta$.
\end{corollary}

Analytic resolvents, the analog of analytic semigroups for Volterra equations,
have been introduced by Da Prato and Iannelli \cite{DI80}. Analogously like in
the theory of analytic semigroups, a characterization of analytic resolvents in
terms of the spectrum of the operator $A$ and the Laplace transform of the
kernel function $a(t)$ is possible.

\begin{theorem}\label{t2.1Pr} (\cite[Theorem 2.1]{Pr93})
 Let $A$ be a closed unbounded operator in $B$ with dense domain $D(A)$ and let
 $a\in L_\mathrm{loc}^1(\mathbb{R}_+)$ satisfy $\int_0^t |a(t)|\,e^{-\omega_a
 t}dt<\infty$ for some $\omega_a\in\mathbb{R}$. Then (\ref{DVEe1}) admits an 
 analytic resolvent $S(t)$ of analyticity type $(\omega_0,\theta_0)$ iff the
 following conditions hold:
 \begin{description}
  \item[1)] $\hat{a}(\lambda)$ admits meromorphic extension to 
    $\sum (\omega_0,\theta_0+\pi/2)$;
  \item[2)] $\hat{a}(\lambda)\neq 0$, and
  $1/\hat{a}(\lambda)\in\varrho(A)$ for all 
  $\lambda\in\sum (\omega_0,\theta_0+\pi/2)$;
  \item[3)] For each $\omega >\omega_0$ and $\theta <\theta_0 $ there is
  a constant $C=C(\omega,\theta)$ such that 
  $H(\lambda):=(1/\hat{a}(\lambda)-A)^{-1}/(\lambda\,\hat{a}(\lambda))$ 
  satisfies estimate
  \begin{equation}\label{e2.10Pr}
   ||H(\lambda)|| \le C/|\lambda-\omega| \quad \mbox{for all} \quad 
   \lambda \in sum (\omega,\theta+\pi/2).
\index{$H(\lambda)$}   
  \end{equation}
 \end{description}
\end{theorem} 

Typical examples of the kernel functions $a(t)$ and the operator $A$ fulfilling
conditions of Theorem \ref{t2.1Pr} are the following.\\

\noindent{\bf Examples} 1. 
Let kernels be $a(t)= t^{\beta-1}/\Gamma(\beta)$, $t>0$, where $\beta\in
(0,2)$ and $\Gamma$ denotes the gamma function.
The pair $(t^{\beta-1}/\Gamma(\beta),A)$ generates a bounded analytic
resolvent iff $\varrho(A)\supset \sum (0,\beta\pi/2)$ and 
$||\mu(\mu-A)^{-1}||\le M$ for all $\mu\in\sum (0,\beta\pi/2)$.

2. An important class of kernels $a(t)$ which satisfy the above conditions
 is the class of completely monotonic kernels.
By \cite[Corollary 2.4]{Pr93} if additionally $a\in C(0,\infty)\cap L^1(0,1)$ 
and $A$ generates
an analytic semigroup $T(t)$ such that $||T(t)||\le M$ on $\Sigma (0,\theta)$
then (\ref{DVEe1}) admits an analytic resolvent $S(t)$ of type $(0,\theta)$.\\

Parabolic Volterra equations appear in a context of Volterra equations 
admiting analytical resolvents.

\begin{definition}\label{d3.1Pr}
 Equation (\ref{DVEe1}) is called {\tt parabolic},\index{equation!parabolic} 
 if the following conditions hold:
 \begin{enumerate}
 \item $\hat{a}(\lambda)\neq 0$, $1/\hat{a}(\lambda)\in\varrho(A)$ for all
 $\mathrm{Re}\, \lambda >0$.
 \item There is a constant $M\ge 1$ such that 
 $H(\lambda)=(I-\hat{a}(\lambda)A)^{-1}/\lambda$ satisfies 
 $||H(\lambda)||\le M/|\lambda|$ for all $\mathrm{Re}\, \lambda >0$. 
 \end{enumerate}
\end{definition}
From the resolvent point of view, the
concept of parabolicity is between the bounded and the analytic
resolvents: if (\ref{DVEe1}) admits an analytic resolvent $S(t)$
then (\ref{DVEe1}) is parabolic. On the other hand, 
if the equation (\ref{DVEe1}) is parabolic and the kernel function $a(t)$  has some
properties, like convexity, then the resolvent corresponding to (\ref{DVEe1})
has, roughly speaking, similar properties like analytic resolvent.

\begin{definition}\label{d3.2Pr}
 Let $a\in L_\mathrm{loc}^1(\mathbb{R}_+)$ be of subexponential growth and
 suppose $\hat{a}(\lambda)\neq 0$ for all $\mathrm{Re}\, \lambda >0$.
 The function  
 $a(t)$ is called {\tt sectorial with angle} $\theta >0$ (or merely 
 $\theta$-{\tt sectorial})\index{function!sectorial} 
 if $|\arg \hat{a}(\lambda)|\le \theta$ for all  $\mathrm{Re}\, \lambda >0$.
\end{definition}
The standard situation leading to parabolic equations is provided by sectorial 
kernels and some closed linear densely defined operators $A$.
 
 The following criteria provide parabolic equations. 
 
\begin{proposition}\label{pr3.1Pr} (\cite[Proposition 3.1]{Pr93}) 
Let $a\in L_\mathrm{loc}^1(\mathbb{R}_+)$ be
$\theta$-sectorial for some $\theta <\pi$, suppose $A$ is closed linear densely
defined, such that $\varrho(A)\supset \sum (0,\theta)$, and 
$||(\mu-A)^{-1}||\le M/|\mu|$ for all $\mu\in \sum (0,\theta)$.
Then (\ref{DVEe1}) is parabolic.
\end{proposition}
The particular case is when  $A$ is the generator of a bounded analytic 
$C_0$-semigroup and the function $a(t)$ is $\pi/2$-sectorial.
Because $a(t)$ is $\pi/2$-sectorial if and only if $a(t)$ is of positive type,
we obtain the following class of parabolic equations.

\begin{corollary}\label{c3.1Pr} (\cite[Corollary 3.1]{Pr93})
Let $a\in L_\mathrm{loc}^1(\mathbb{R}_+)$ be of subexponential growth and of
positive type, and let $A$ generate a bounded analytical $C_0$-semigroup in $B$.
Then (\ref{DVEe1}) is parabolic.

\end{corollary}

In the sequel we will need some regular kernels.

\begin{definition}\label{d3.3Pr}
Let $a\in L_\mathrm{loc}^1(\mathbb{R}_+)$ be of subexponential growth and
$k\in\mathbb{N}$. The function $a(t)$ is called 
$k$-{\tt regular} if there is a constant
$c>0$ such that $|\lambda^n\,a^{(n)}(\lambda)|\le c|\hat{a}(\lambda)|$ for all 
$\mathrm{Re}\, \lambda >0$, $0\le n\le k$.
\end{definition}
\index{function!$k$-regular}

\noindent
{\bf Comment} Any $k$-regular kernel $a(t),~k\ge 1$ has the property that 
$\hat{a}(\lambda)$ has no zeros in the open right halfplane. 

We would like to emphasize that
convolutions of $k$-regular kernels are again $k$-regular what follows 
from the product rule of convolutions. The integration and
differentiation preserve $k$-regularity, as well. Unfortunately, 
sums and differences of $k$-regular kernels need not be $k$-regular.
We may check it taking  
\index{$b(t)$}
$a(t)=1$ and $b(t)=t^2$. However, if $a(t)$ and $b(t)$ are 
$k$-regular and $|\arg \hat{a}(\lambda)-\arg \hat{b}(\lambda)|\le\theta\le\pi$,
$\mathrm{Re}\,\lambda >0$ then $a(t)+b(t)$ is $k$-regular.

If $a(t)$ is real-valued and 1-regular then $a(t)$ is sectorial. The converse of
this is not true. As the counterexample we can take 
$a(t)=1$ for $t\in [0,1]$, $a(t)=0$ for $t>1$.

\begin{proposition}\label{pr3.2Pr}
(\cite[Proposition 3.2]{Pr93}) Suppose $a\in L_\mathrm{loc}^1(\mathbb{R}_+)$ 
is such that $\hat{a}(\lambda)$ admits analytic extension to $\sum (0,\varphi)$,
where $\varphi >\pi/2$, and there is $\theta\in (0,\infty)$ such that 
$|\arg \hat{a}(\lambda)|\le \theta$ for all $\lambda\in\sum (0,\varphi)$.
Then $a(t)$ is $k$-regular  for every $k\in\mathbb{N}$.
\end{proposition}

So, nonnegative and nonicreasing kernels are in general not 1-regular but if the
kernel is also convex, then it is 1-regular.

\begin{definition}\label{d3.4Pr}
 Let $a\in L_\mathrm{loc}^1(\mathbb{R}_+)$ and $k\ge 2$. 
 The function $a(t)$ is called \linebreak
 $k$-{\tt monotone}\index{function!$k$-monotone} 
 if $a\in C^{k-2}(0,\infty),~(-1)^n\,a^{(n)}(t)\ge 0$ for all
 $t>0,~ 0\le n\le k-2$, and $(-1)^{k-2}\,a^{(k-2)}(t)$ is nonincreasing and
 convex.
\end{definition}
\index{$C^{k-2}(0,\infty)$}
By definition, a 2-monotone kernel $a(t)$ is nonnegative, nonicreasing and
convex, and  $a(t)$ is completely monotonic if and only if $a(t)$ is
$k$-monotone for all $k\ge 2$.
 
\begin{proposition}\label{pr3.3Pr}
(\cite[Proposition 3.3]{Pr93}) Suppose $a\in L_\mathrm{loc}^1(\mathbb{R}_+)$ 
is $(k+1)$-monotone, $k\ge 1$. Then $a(t)$ is $k$-regular and of positive type.
\end{proposition}

Now, we recall the main theorem on resolvents for parabolic Volterra equations.

\begin{theorem}\label{th3.1Pr}
(\cite[Theorem 3.1]{Pr93}) Let $B$ be a Banach space, $A$ a closed linear
operator in $B$ with dense domain $D(A)$, $a\in L_\mathrm{loc}^1(\mathbb{R}_+)$.
Assume (\ref{DVEe1}) is parabolic, and $a(t)$ is $k$-regular, for some $k\ge 1$.
Then there is a resolvent $S\in C^{k-1}((0,\infty);L(B))$ for (\ref{DVEe1}), and
there is a constant $M\ge 1$ such that estimates 
\begin{equation}\label{eq3.12Pr}
 ||t^n\,S^{(n)}(t)||\le M, \quad \mbox{for all} \quad t\ge 0,~~n\le k-1,
\end{equation}
and
\begin{equation}\label{eq3.13Pr}
 ||t^k\,S^{(k-1)}(t)-s^k\,S^{(k-1)}(s)|| \le M|t-s|[1+\log \frac{t}{t-s}],
 ~~0\le s<t<\infty,
\end{equation}
are valid.
\index{$C^{k-1}((0,\infty);L(B))$}
\end{theorem}

\section{Approximation theorems} \label{DVEsec:5}

In this paper the following results contained in \cite{KL07c} 
concerning convergence of
resolvents for the equation (\ref{DVEe1}) in Banach space $B$ will
play the key role. They extend some
results of Cl\'{e}ment and Nohel obtained in \cite{CN79} for
contraction semigroups. Theorems \ref{DVEt1}-\ref{DVEt1a} and
Proposition \ref{AScom} are not yet published.

\begin{theorem} \label{DVEt1}
Let $A$ be the generator of a $C_0$-semigroup in $B$ and suppose
the kernel function $a(t)$ is completely positive. Then $(A,a)$
admits an exponentially bounded resolvent $S(t)$. Moreover, there
exist bounded operators $A_n$ such that $(A_n,a)$ admit resolvent
families $S_n(t)$ satisfying $ ||S_n(t) || \leq Me^{w_0 t}~ (M\geq
1,~w_0\geq 0)$ for all $t\geq 0,~n\in \mathbb{N},$ and
\begin{equation} \label{DVEe10}
S_n(t)x \to S(t)x \quad \mbox{as} \quad n\to +\infty
\end{equation}
for all $x \in B,\; t\geq 0.$ Additionally, the convergence is
uniform in $t$ on every compact subset of $ \mathbb{R}_+$.
\end{theorem}

\begin{proof}
The first assertion follows directly from \cite[Theorem
5]{Pr87} (see also \cite[Theorem 4.2]{Pr93}). Since $A$ generates a
$C_0$-semigroup $T(t),~t\geq 0$, the resolvent set $\rho(A) $ of
$A$ contains the ray $ [w,\infty)$ and
$$
||R(\lambda,A)^k || \leq \frac{M}{(\lambda -w)^k } \qquad
\mbox{for } \lambda > w, \qquad k\in \mathbb{N},
$$
where $R(\lambda,A)=(\lambda I-A)^{-1}$, $~\lambda\in \rho(A)$.
\index{$R(\lambda,A)$}

Define
\begin{equation} \label{DVEe11}
A_n := n AR(n,A) = n^2 R(n,A) - nI, \qquad n> w
\end{equation}
the {\it Yosida approximation} of $A$.

\pagebreak
Then
\begin{eqnarray*}
||e^{t A_n} || &=& e^{-nt} || e^{n^2 R(n,A)t} || \leq
e^{-nt} \sum_{k=0}^{\infty} \frac{n^{2k} t^k}{k!} ||R(n,A)^k|| \\
&\leq& M e^{(-n + \frac{n^2}{n-w})t} = M e^{ \frac{nwt}{ n-w}}.
\end{eqnarray*}
Hence, for $n > 2w$ we obtain
\begin{equation}{\label{eSW6}}
|| e^{A_n t} || \leq M e^{2wt}.
\end{equation}
Taking into account the above estimate  and the complete
positivity of the kernel function $a(t)$, we can follow the same
steps as in \cite[Theorem 5]{Pr87} to obtain that there exist
constants $M_1
>0 $ and $ w_1 \in \mathbb{R} $ (independent of $n$, due to
(\ref{eSW6})) such that
$$
||[H_n(\lambda)]^{(k)} || \leq \frac{M_1}{ \lambda - w_1} \quad
\mbox{ for } \lambda > w_1,
$$
where $H_n(\lambda):= (\lambda - \lambda \hat
a(\lambda)A_n)^{-1}.$ Here and in the sequel the hat indicates the
Laplace transform. Hence, the generation theorem for resolvent
families implies that for each  $ n \in \mathbb{N}$, the pair
$(A_n,a)$ admits resolvent family $S_n(t)$ such that
\begin{equation}{\label{DVEe12}}
||S_n(t) || \leq M_1 e^{w_1 t} \quad \mbox{ for all } n \in
\mathbb{N}.
\end{equation}
In particular, the Laplace transform $ \hat S_n(\lambda) $ exists
and satisfies
$$
\hat S_n(\lambda ) = H_n(\lambda) = \int_0^{\infty} e^{-\lambda t}
S_n(t)dt, \qquad \lambda > w_1.
$$
Now recall from semigroup theory that for all $\mu$ sufficiently
large we have
$$ R(\mu,A_n)= \int_0^\infty e^{-\mu t} \,e^{A_nt}\,dt $$
as well as,
$$ R(\mu,A)= \int_0^\infty e^{-\mu t} \,T(t)\,dt\,. $$
\index{$R(\mu,A_n)$}

Since $\hat a(\lambda) \to 0$ as $ \lambda \to \infty$, we deduce
that for all $\lambda$ sufficiently large, we have
$$ 
H_n(\lambda) = \frac{1}{\lambda \hat a (\lambda) } R(
\frac{1}{\hat a (\lambda) }, A_n) = \frac{1}{\lambda \hat a
(\lambda)} \int_0^{\infty} e^{(-1/\hat a(\lambda))t} e^{A_n t}
dt\,,
$$ 
and
$$ 
H(\lambda) = \frac{1}{\lambda \hat a (\lambda) } R( \frac{1}{\hat
a (\lambda) }, A) = \frac{1}{\lambda \hat a (\lambda)}
\int_0^{\infty} e^{(-1/\hat a(\lambda))t} T(t) dt\,.
$$ 
\index{$H(\lambda)$}
\index{$H_n(\lambda)$}

Hence, from the identity
$$ H_n(\lambda) - H(\lambda)  = \frac{1}{\lambda \hat a (\lambda) }
[R(\frac{1}{\hat a (\lambda) }, A_n)-R(\frac{1}{\hat a (\lambda)
}, A)]
$$
and the fact that $R(\mu,A_n)\to R(\mu,A)$ as $n\to\infty$ for all
$\mu$ sufficiently large (see e.g.\ \cite[Lemma~7.3]{Pa83}, we
obtain that
\begin{equation}{ \label{DVEe13}}
H_n(\lambda) \to H(\lambda) \quad \mbox{as } n \to \infty\;.
\end{equation}
Finally, due to (\ref{DVEe12}) and (\ref{DVEe13}) we can use the
Trotter-Kato theorem for resolvent families of operators (cf.
\cite[Theorem 2.1]{Li90}) and the conclusion follows. 
\qed
\end{proof}

Let us recall, e.g.\ from \cite{Fa83}, that a family $\mathcal{C}(t)$,
 $t\ge 0$, of linear bounded operators on $H$ is called {\tt cosine family}
 if for every $t,s \ge 0$, $t>s$: \linebreak
 $\mathcal{C}(t+s)+\mathcal{C}(t-s)=
 2 \mathcal{C}(t)\mathcal{C}(s)$. \medskip

Theorem \ref{DVEt1} may be reformulated in the following version.
\begin{theorem} \label{DVEt1a}
Let $A$ generate a cosine family $T(t)$ in $B$ such that 
$||T(t)||\le M e^{\omega t}$ for $t>0$ holds, and suppose
the kernel function $a(t)$ is completely positive. Then $(A,a)$
admits an exponentially bounded resolvent $S(t)$. Moreover, there
exist bounded operators $A_n$ such that $(A_n,a)$ admit resolvent
families $S_n(t)$ satisfying $ ||S_n(t) || \le \widetilde{M}e^{w_0 t}~ 
(\widetilde{M}\geq 1,~w_0\geq 0)$ for all $t\geq 0,~n\in \mathbb{N},$ and
$$ 
S_n(t)x \to S(t)x \quad \mbox{as} \quad n\to +\infty
$$ 
for all $x \in B,\; t\geq 0.$ Additionally, the convergence is
uniform in $t$ on every compact subset of $ \mathbb{R}_+$.
\end{theorem}

\noindent {\bf Remarks} \label{comment}

\noindent 1. By \cite[Theorem 4.3]{Pr93} or \cite[Theorem 6]{Pr87}
Theorem \ref{DVEt1a} holds also in two other cases:\\
a) $a(t)$ is a creep function with the function $a_1(t)$ log-convex;\\
b) $a=c\star c$ with some completely positive $c\in
L_\mathrm{loc}^1(\mathbb{R}_+)$.

(Let us recall the definition \cite[Definition 4.4]{Pr93}:\\
A function $a: \mathbb{R}_+\mapsto\mathbb{R}$ is called a {\tt creep function}
if $a(t)$ is nonnegative, nondecreasing, and concave. A creep function $a(t)$
has a {\tt standard form} 
$$a(t)=a_0 +a_\infty t +\int_0^t a_1(\tau)d\tau,~ t>0,$$
where $a_0=a(0+)\ge 0$, $a_\infty=\lim_{t\to\infty}a(t)/t=\inf_{t>0} a(t)/t 
\ge 0$, and $a_1(t)= \dot{a}(t)-a_\infty$ is nonnegative, nonincreasing,
 $\lim_{t\to\infty} a_1(t)=0$.)\\

\noindent 2. Other examples of the convergence (\ref{DVEe10}) for
the resolvents are given, e.g., in \cite{CN79} and \cite{Fr69}. In
the first paper, the operator $A$ generates a linear continuous
contraction semigroup. In the second one, $A$
belongs to some subclass of sectorial operators and the kernel $a(t)$
is an absolutely continuous function fulfilling some technical
assumptions.\\

\noindent{\bf Comment} 
The above theorem gives a partial answer to the following open
problem for a resolvent family $S(t)$ generated by a pair $(A,a)$:
do exist bounded linear operators $A_n$ generating resolvent
families $S_n(t)$ such that $S_n (t)x \to S(t)x$~? Note that in
case $a(t)\equiv 1$ the answer is yes, namely $A_n$ are provided
by the Hille-Yosida approximation of $A$ and $S_n(t) = e^{t A_n}.$\\

The following result will be used in the sequel.

\begin{proposition}\label{AScom}
 Let $A, A_n$ and $ S_n( t)$ be given as in Theorem \ref{DVEt1} or Theorem
 \ref{DVEt1a}. Then the operators $ S_n(t)$ commute with  the operator $A$, 
 for every $n$ sufficiently large and $t\ge 0$.
\end{proposition}

\begin{proof}
 For each $n$ sufficiently large the bounded operators
$A_n$ admit a resolvent family $S_n(t)$, so by the complex
inversion formula for the Laplace transform we have
$$S_n(t)=\frac{1}{2\pi i}\int_{\Gamma_n} e^{\lambda t}H_n(\lambda) d\lambda$$
where $\Gamma_n$ is a simple closed rectifiable curve surrounding
the spectrum of $A_n$ in the positive sense.

 On the other hand, $H_n(\lambda) := (\lambda - \lambda\hat
a (\lambda) A_n)$ where $ A_n := nA(n -A)^{-1},$ so each $A_n$
commutes with $A$ on $D(A)$ and then each $H_n(\lambda)$ commutes
with $A$, on $D(A)$, too.

Finally, because $A$ is closed and all the following integrals
are convergent (exist), for all $n$ sufficiently large and
$x \in D(A)$ we have
\begin{eqnarray*}
AS_n(t)x & = & A \int_{\Gamma_n} e^{\lambda t} H_n(\lambda)x
d\lambda
 = \int_{\Gamma_n} e^{\lambda t} AH_n(\lambda)x d\lambda \\
 & = & \int_{\Gamma_n} e^{\lambda t} H_n(\lambda)Ax d\lambda
  = S_n(t) Ax \,.  \qquad\qquad\qquad \qed
\end{eqnarray*}
\end{proof}

\chapter{Probabilistic background} \label{PBch:2} 

In this chapter we recall from \cite{CP78}, \cite{Ic82}, \cite{DZ92} 
and \cite{GT95} basic  and important concepts and results of the 
infinite dimensional 
stochastic analysis needed in the sequel. In particular, we present 
construction of stochastic integral with respect to a cylindrical 
Wiener process, published in \cite{Ka98}.

\section{Notations and conventions} \label{PBsec:1}

Assume that $(\Omega,\mathcal{F},P)$ is a probability space equipped with an 
increasing family of $\sigma$-fields $(\mathcal{F}_t),~t\in I$, where 
$I=[0,T]$ or $I=[0,+\infty )$, called {\em filtration}. We shall denote by
$\mathcal{F}_{t^+}$ the intersection of all $\sigma$-fields $\mathcal{F}_s$,
$s>t$. We say that filtration is {\em normal} if $\mathcal{F}_0$ contains 
all sets $B\in\mathcal{F}$ with measure $P(B)=0$ and if 
$\mathcal{F}_t=\mathcal{F}_{t^+}$ for any $t\in I$, that is, the filtration is 
right continuous. 
\index{$(\Omega,\mathcal{F},P)$}
\index{$(\mathcal{F}_t)$}

In the paper we assume that filtration 
$(\mathcal{F}_t)_{t\in I}$ is normal. This assumption enables to choose 
modifications of considered stochastic processes with required measurable 
properties.

Let $H$ and $U$ be two separable Hilbert spaces.
\index{$H$}\index{$U$}
In the whole paper we write explicitely
indexes indicating the appropriate space in norms $|\cdot |_{(\cdot)}$
and inner products $\langle\cdot ,\cdot\rangle_{(\cdot)}$ .

\begin{definition}\label{PBd:1}
 The $H$-valued process $X(t),~t\in I$, is {\tt adapted} to the family 
 $(\mathcal{F}_t)_{t\in I}$, if
 for arbitrary $t\in I$ the random variable $X(t)$ is 
 $\,\mathcal{F}_t$-measurable.
\end{definition}
\begin{definition}\label{PBd:2}
 The $H$-valued process $X(t),~t\in [0,T]$, is {\tt progressively measurable}
 if for every $t\in [0,T]$ the mapping
 $[0,t]\times\Omega\rightarrow H,~(s,w)\mapsto X(s,w)$ is 
 $\mathcal{B}([0,t])\times \mathcal{F}_t$-measurable.
\end{definition}
\index{process!adapted}
\index{process!progressively measurable}

We will use the following well-known result, see e.g.\ \cite{DZ92}.

\begin{proposition} \label{PBp1} (\cite[Proposition 3.5]{DZ92})
Let $X(t),~t\in [0,T]$, be a stochastically  
continuous and adapted process with values
in $H$. Then $X$ has a progressively measurable modification.
\end{proposition}

By $\mathcal{P}_\infty$ we denote a $\sigma$-field of subsets of
 $[0,+\infty)\times\Omega$ defined as follows: $\mathcal{P}_\infty$ is
 the $\sigma$-field generated by sets of the form: $(s,t]\times F$,
 where $0\le s\le t<+\infty$, $F\in \mathcal{F}_s$ and $\{0\}\times F$,
 when $F\in \mathcal{F}_0$. The restriction of the $\sigma$-field 
 $\mathcal{P}_\infty$ to $[0,T]\times\Omega$ will be denoted by 
 $\mathcal{P}_T$.
\index{$\mathcal{P}_\infty$}
\index{$\mathcal{P}_T$}
 
\begin{definition}\label{PBd:3}
 An arbitrary measurable mapping from
 $([0,+\infty)\times\Omega,\mathcal{P}_\infty)$ or
 $([0,T]\times\Omega,\mathcal{P}_T)$ into $(H;\mathcal{B}(H))$ is called
 a {\tt predictable process}.
\end{definition}
\index{process!predictable}

\noindent{\bf Comment} A predictable process is an adapted one.
 
\begin{proposition}\label{PBp1a} (\cite[Proposition 3.6]{DZ92})
 Assume that $X(t),~t\in [0,T]$, is an adapted and stochastically continuous
 process. Then the process $X$ has a predictable version on $[0,T]$.
\end{proposition}


By $L(U,H)$, $L(U)$ we denote spaces of linear bounded operators
\index{$L(U,H)$, $L(U)$} 
from $U$ into $H$ and in $U$, respectively.
As previously, the operator norm is denoted by $||\cdot||$.

An important role will be played by the space of 
Hilbert-Schmidt operators. 
Let us recall the following definition.
\vskip2mm
\begin{definition} \label{def2} 
(\cite{Ba81} or \cite{DZ92}) 
Assume that $\{e_k\}\subset U$ and  $\{f_j\}\subset H$ are orthonormal
bases of $U$ and $H$, respectively.
 A linear bounded operator $T:\ U\to H$ is called
{\tt ~Hilbert-Schmidt operator~}\index{operator!Hilbert-Schmidt} 
if $\sum_{k=1}^{\infty}|Te_k|_H^2<+\infty$.
\end{definition}
\vskip2mm
Because
$$
\sum_{k=1}^{\infty}|Te_k|_H^2=\sum_{k=1}^{\infty}\sum_{j=1}^{\infty}
(Te_k,f_j)_H^2=\sum_{j=1}^{\infty}|T^*f_j|_U^2,
$$
where $T^*$ denotes the operator adjoint to $T$, then the definition of
Hilbert-Schmidt operator and the number
$||T||_{HS}=\left(\sum_{k=1}^{\infty} |Te_k|_H^2\right)^{\frac12}$
do not depend on the basis $\{e_k\}$, $k\in\mathbb{N}$.
Moreover $||T||_{HS}=||T^*||_{HS}.$
\index{$\parallel T\parallel_{HS}$}

Additionally, $L_2(U,H)$ -- the set of all Hilbert-Schmidt operators
from $U$ into $H$, endowed with the norm $||\cdot ||_{HS}$ defined
above, is a separable Hilbert space. \label{H-Sn}
\index{$L_2(U,H)$} 
\vskip2mm

\section{Classical infinite dimensional Wiener process} \label{PBsec:3}

Here we recall from \cite{CP78} and \cite{Ic82} the definition of
Wiener process with values in a real separable Hilbert space $U$
and the stochastic integral with respect to this process.

\begin{definition} \label{def1} Let $Q:\ U \to U$ be a linear symmetric
non-negative nuclear operator\index{operator!nuclear} 
($\mathrm{Tr}\, Q<+\infty$).  A square integrable
$U$--valued stochastic process $W(t)$, $t\ge0$,
defined on a probability space $(\Omega,\mathcal{F},
(\mathcal{F}_t)_{t\ge0},P)$,
where $\mathcal{F}_t$ denote $\sigma$-fields such that
$\mathcal{F}_t\subset \mathcal{F}_s \subset \mathcal{F}$ for $t<s$, 
is called classical or genuine
{\tt Wiener process}\index{Wiener process!classical or genuine} 
with covariance operator $Q$ if:
\index{$Q$}
\index{covariance operator}
\begin{enumerate}
\item $W(0)=0$,
\item $EW(t)=0$, $\mathrm{Cov} [W(t)-W(s)]=(t-s)Q$ for all $s,t\ge 0$,
\item $W$ has independent increments, that is $W(s_4)-W(s_3)$ and 
 $W(s_2)-W(s_1)$ are independent whenever $0\le s_1\le s_2\le s_3\le s_4$, 
\item $W$ has continuous trajectories,
\item $W$ is adapted with respect to the filtration $(\mathcal{F}_t)_{t\ge 0}$.
\end{enumerate}
\end{definition}

If we choose $\mathcal{F}_t$ to be the $\sigma$-field generated by 
$\{W(s); 0\le s\le t\}$, then $W(t)-W(s)$ is independent of $\mathcal{F}_s$
for all $t>s$ from condition 3.~of the above definition. Then 
$\mathbb{E}\{W(t)-W(s)|\mathcal{F}_s\}=\mathbb{E}\{W(t)-W(s)\}=0$ by
condition 2.
Hence, $\mathbb{E}\{W(t)|\mathcal{F}_s\}=W(s)$ w.p.1 and 
$\{W(t),\mathcal{F}_t\}$ is a martingale on $[0,+\infty)$.

We remark that an alternative definition is to replace condition 4.~of 
definition \ref{def1} by assuming that $W(t)$ is Gaussian for all $t\ge 0$,
see \cite{CP78} for details.\\

In the light of the above, Wiener process is Gaussian and has the
following expansion (see e.g. \cite[Lemma 5.23]{CP78}).
Let $\{e_i\}\subset U$ be an orthonormal set of eigenvectors of $Q$ with
corresponding eigenvalues $\zeta_i$ (so 
$\mathrm{Tr}\, Q=\sum_{i=1}^{\infty}\zeta_i$),
then $$W(t)=\sum_{i=1}^{\infty}\beta_i(t)e_i\;,$$ 
where $\beta_i$ are mutually independent
real Wiener processes with $E(\beta_i^2(t))=\zeta_it$.\\

\noindent{\bf Remark} If $W(t)$ is a Wiener process in $U$ with covariance
operator $Q$, then $\mathbb{E}|W(t)-W(s)|_U^{2n}\le (2n-1)!(t-s)^n
(\mathrm{Tr\,Q})^n$, where the equality holds for $n=1$.\vskip1mm

The above type of structure of Wiener process will be used in the definition 
of the stochastic integral.
\vskip2mm
For any Hilbert space $H$ we denote by $M(H)$ the space of all stochastic
processes
\index{$M(H)$}
$g:\ [0,T]\times \Omega \rightarrow L(U,H)$ such that
$$
E\left(\int_0^T\|g(t)\|_{L(U,H)}^2dt\right)<+\infty
$$
and for all $u\in U$, $(g(t)u), t\in [0,T]$ is an $H$--valued and 
$\mathcal{F}_t$-adapted stochastic process. 
\vskip1mm
For each $t\in [0,T]$,
the stochastic integral $\int_0^tg(s)dW(s)\in H$ is defined for all
$g\in M(H)$ by
$$
\int_0^tg(s)dW(s)=\lim_{m\to\infty}\sum_{i=1}^m\int_0^tg(s)e_id\beta_i(s)
$$
in $L^2(\Omega;H)$.
\index{$L^2(\Omega;H)$}

We shall show that the series in the above formula is
convergent.
\vskip5mm
Let
$W^{(m)}(t)=\sum_{i=1}^me_i\beta_i(t).$
Then, the integral
$$ \int_0^tg(s)dW^{(m)}(s)=\sum_{i=1}^m\int_0^tg(s)e_id\beta_i(s)$$
is well defined for $g\in M(H)$ and additionally
$$ \int_0^tg(s)dW^{(m)}(s) ~{\longrightarrow}_{{\hspace{-4ex}}_{m\to\infty}}
\int_0^tg(s)dW(s) $$
in $L^2(\Omega;H)$.

This convergence comes from the fact that the sequence
$$
y_m=\int_0^tg(s)dW^{(m)}(s),\qquad m\in \mathbb{N}
$$
is Cauchy sequence in the space of square integrable random variables.
Using properties of stochastic integrals with respect to
$\beta_i(s)$, for any $m,n\in \mathbb{N}$, $m<n$, we have:
\begin{eqnarray}\label{Ch2e1}
E\left(\left|y_n-y_m\right|_H^2\right) &=&
\sum_{i=m+1}^n\zeta_iE\int_0^t\left(g(s)e_i,g(s)e_i\right)_Hds \\
&\leq& \left(\sum_{i=m+1}^n\zeta_i\right)E\int_0^t\|g(s)\|_{L(U,H)}^2ds
 ~~\longrightarrow_{{\hspace{-4ex}}_{m,n\to\infty}}~ 0. \nonumber
\end{eqnarray}
Hence, there exists a limit of the sequence $(y_m)$ which defines the stochastic
integral $\int_0^tg(s)dW(s)$.

The stochastic integral defined above has the following properties (see e.g.\
\cite{Ic82}).

\begin{proposition} \label{CH2pr1}
 Let $g\in M(H)$. Then
 \begin{enumerate}
  \item $\displaystyle \mathbb{E}\left( \int_0^T g(t)dW(t)\right)=0 $ ;
  \vskip1.5mm
  \item $\displaystyle \mathbb{E}\left| \int_0^T g(t)dW(t)\right|_H^2 =
        \int_0^T \mathbb{E} \left( \mathrm{Tr} (g(t)Qg^*(t))\right) dt =$\\
	  $\displaystyle = \int_0^T \mathbb{E} \left( \mathrm{Tr} (g^*(t)g(t)Q)
	  \right)dt
	  \le \mathrm{Tr} Q\int_0^T \mathbb{E} |g(t)|_H^2 dt$ ;
  \vskip1.5mm
  \item $\displaystyle 	\mathbb{E}\left[ \sup_{t\in [0,T]} \left|\int_0^t\!\!
        g(s) dW(s)\right|_H^2 \right] \le 4 \mathbb{E}\left| \int_0^T\!\!\!
	  g(s) dW(s)\right|_H^2\!\!\le 4 \mathrm{Tr} Q \!\int_0^T \!\!\!
	  \mathbb{E} |g(s)|_H^2 ds$ ;
  \vskip1.5mm
  \item $\displaystyle \mathbb{E} \left[ \sup_{t\in [0,T]} \left|\int_0^t
        g(s) dW(s)\right|_H \right] \le 3 \mathbb{E}\left[\int_0^T
	  \mathrm{Tr}\left( g(s)Qg^*(s)\right)ds \right]^{\frac{1}{2}}$ .
 \end{enumerate}
\end{proposition}
\vskip2mm
(Since $|\int_0^t g(s) dW(s)|_H^2$ is a submartinagale, property 3 follows from
Doob's inequality. Property 4 is also a consequence of a general inequality for
martingales.)

\section{Stochastic integral with respect to cylindrical Wiener process}
\label{PBs:3}
 
The construction of the stochastic integral in section \ref{PBsec:3}
required that $Q$ was a nuclear operator. 
In some cases, this assumption seems to be artificial. 
For instance, all processes stationary with respect to space variable, have
non-nuclear covariance operator.  So, we shall  extend the 
definition of the stochastic integral to the case of general 
bounded self-adjoint, non-negative operator $Q$ on Hilbert space $U$. 
In this section we provide a construction, published in \cite{Ka98},
of stochastic integral with respect to an infinite dimensional cylindrical 
Wiener process alternative to that given in \cite{DZ92}.
The construction is based 
on the stochastic integrals with respect to real-valued Wiener processes.
The advantage
of using of such a construction is that we can use basic results and
arguments of the  finite dimensional case.
To avoid trivial complications we shall assume that $Q$ is strictly
positive, that is: $Q$ is non-negative and $Qx\ne 0$ for $x\ne 0$.\\

Let us introduce the subspace $U_0$ of the space $U$ defined by
\index{$U_0$}
$U_0=Q^{\frac12}(U)$ with the norm
$$
|u|_{U_0}=\left| Q^{-\frac12}u\right|_U,\qquad u\in U_0.
$$

Assume that $U_1$ is an arbitrary Hilbert space such that $U$ is
continuously embedded
into $U_1$ and the embedding of $U_0$ into $U_1$ is a Hilbert-Schmidt operator.
\vskip2mm
In particular,
when $Q=I$, then $U_0=U$ and the embedding of $U$ into $U_1$ is
Hilbert-Schmidt operator.
When $Q$ is a nuclear operator, that is, $\mathrm{Tr}\,Q<+\infty$, then
$U_0=Q^{\frac12}(U)$ and we can take $U_1=U$. Because in this case
$Q^{\frac12}$ is Hilbert-Schmidt operator then the embedding $U_0\subset U$ is
Hilbert-Schmidt operator.

\vskip2mm
We denote by $L_2^0=L_2(U_0,H)$ the space of Hilbert-Schmidt 
\index{$L_2^0, L_2(U_0,H)$}
operators acting from $U_0$ into $H$. 

\vskip1mm
Let us consider the norm of the operator $\psi \in L_2^0$:
\begin{eqnarray*}
\Vert \psi\Vert_{L_2^0}^2 &=&
\sum_{h,k=1}^{\infty}\left(\psi g_h,f_k\right)_H^2
=\sum_{h,k=1}^{\infty}\lambda_h\left(\psi e_h,f_k\right)_H^2\\
&=&\Vert \psi Q^{\frac12}\Vert_{HS}^2=\mathrm{Tr}(\psi Q\psi^*),
\end{eqnarray*}
where $g_j=\sqrt{\lambda_j}e_j$, and $\{\lambda_j\}$, $\{e_j\}$ are
eigenvalues and eigenfunctions of the operator $Q$;
\index{$\Vert \psi\Vert_{L_2^0}^2$}
\newline
$\{g_j\}$, $\{e_j\}$ and
$\{f_j\}$ are orthonormal bases of spaces $U_0$, $U$ and $H$,
respectively.
\vskip2mm
The space $L_2^0$ is a separable Hilbert space with the norm
$\Vert\psi\Vert_{L_2^0}^2=\mathrm{Tr}\left(\psi Q\psi^*\right)$.\\

Particular cases:
\begin{enumerate}
\item If $Q=I$ then $U_0=U$ and the space $L_2^0$ becomes $L_2(U,H)$.
\item When $Q$ is a nuclear operator then $L(U,H)\subset L_2(U_0,H)$. 
Assume that $K\in L(U,H)$ and let us consider the operator
$\psi =K|_{U_0}$, that is the restriction of operator $K$ to the space $U_0$.
Because $Q$ is nuclear operator, then $Q^{\frac12}$
is Hilbert-Schmidt operator. So, the embedding $J$ of the space $U_0$ into
$U$ is Hilbert-Schmidt operator. We have to compute the norm
$\Vert\psi\Vert_{L_2^0}$ of the operator $\psi:\ U_0\to H$. We obtain
$\Vert\psi\Vert_{L_2^0}^2\equiv \Vert KJ\Vert_{L_2^0}^2=
\mathrm{Tr} KJ(KJ)^*,$
where $J:\ U_0\to U$.
\vskip2mm
Because $J$ is Hilbert--Schmidt operator and $K$ is linear bounded operator
then, basing on the theory of Hilbert--Schmidt operators (e.g. \cite{GV61},
Chapter I), $KJ$ is Hilbert--Schmidt operator, too. Next, $(KJ)^*$ is
Hilbert--Schmidt operator. In consequence, $KJ(KJ)^*$ is nuclear operator, so
$\mathrm{Tr} KJ(KJ)^*<+\infty$. Hence, $\psi =K|_{U_0}$ is Hilbert-Schmidt operator
on the space $U_0$, that is $K\in L_2(U_0,H)$.
\end{enumerate}

Although Propositions \ref{prop1} and \ref{prop2} introduced below are known 
(see e.g.\
Proposition 4.11 in the monograph \cite{DZ92}), because of their importance
we formulate them again.
In both propositions, $\{g_j\}$ denotes an orthonormal basis in $U_0$ and 
$\{\beta_j\}$ is
a family of independent standard real-valued Wiener processes.

\begin{proposition} \label{prop1}
The formula
\begin{equation}\label{PBe2}
W_c(t)=\sum_{j=1}^{\infty}g_j\beta_j(t),\qquad t\ge 0
\end{equation}
defines Wiener process in $U_1$ with the covariance operator $Q_1$ such that
$\mathrm{Tr}\,Q_1<+\infty$.
\end{proposition}
\proof{
This comes from the fact that the series (\ref{PBe2}) is convergent in space
$L^2(\Omega,\mathcal{F}, P;U_1)$. We have
\begin{eqnarray*}
&~& E\left(\left|\sum_{j=1}^ng_j\beta_j(t)-\sum_{j=1}^mg_j\beta_j(t)
\right|_{U_1}^2
\right)=E\left(\left| \sum_{j=m+1}^n g_j\beta_j(t)
\right|_{U_1}^2\right) = \\
&=& E\left(\sum_{j=m+1}^ng_j\beta_j(t),\ \sum_{k=m+1}^ng_k\beta_k(t)\right)
_{U_1}
=E\sum_{j=m+1}^n\left(g_j\beta_j(t),g_j\beta_j(t)\right)_{U_1}\\
&=&E\left(\sum_{j=m+1}^n(g_j,g_j)_{U_1}\beta_j^2(t)\right)
=t\sum_{j=m+1}^n| g_j|_{U_1}^2,\qquad n\ge m\ge 1.
\end{eqnarray*}
From the  assumption, the embedding $J:\ U_0\to U_1$ is Hilbert--Schmidt operator,
then for the basis $\{g_j\}$, complete and orthonormal in $U_0$, we have
$\sum_{j=1}^{\infty}\left| Jg_j\right|_{U_1}^2<+\infty.$
Because $Jg_j=g_j$ for any $g_j\in U_0$, then
$\sum_{j=1}^{\infty}\left| g_j\right|_{U_1}^2\!<\!+\!\infty$
which means
$\sum_{j=m+1}^n\left| g_j\right|_{U_1}^2 \to 0$
when $m,n\to \infty.$
\vskip1mm
Conditions 1), 2), 3) and 5) of the definition of Wiener process are
obviously satisfied. The process defined by (\ref{PBe2}) is Gaussian because
$\beta_j(t),\ j\in\mathbb{N}$, are independent Gaussian processes.
By Kolmogorov test theorem (see e.g.\ \cite{DZ92}, Theorem 3.3), trajectories of
the process $W_c(t)$ are continuous (condition 4) of the definition of Wiener
process) because $W_c(t)$ is Gaussian.

Let $Q_1:\ U_1\to U_1$ denote the covariance operator of the process $W_c(t)$
defined by (\ref{PBe2}). From the definition of covariance, for $a,b\in U_1$ we have:
\begin{eqnarray*}
\left( Q_1 a,b \right)_{_{U_1}}&=&E(a,W_c(t))_{_{U_1}}(b,W_c(t))_{_{U_1}}
=E\left(\sum_{j=1}^{\infty}(a,g_j)_{_{U_1}}(b,g_j)_{_{U_1}}
\beta_j^2(t)\right)\\
&=&t\sum_{j=1}^{\infty}(a,g_j)_{_{U_1}}(b,g_j)_{_{U_1}}
=t\left(\sum_{j=1}^{\infty}g_j(a,g_j)_{_{U_1}},b\right)_{U_1}.
\end{eqnarray*}
Hence $Q_1a=t\sum_{j=1}^{\infty}g_j(a,g_j)_{_{U_1}}.$
\vskip2mm
Because the covariance operator $Q_1$ is non--negative, then
(by Proposition C.3 in \cite{DZ92}) $Q_1$ is a nuclear operator if and only if
$\sum_{j=1}^{\infty}(Q_1h_j,h_j)_{_{U_1}}<+\infty,$
where $\{h_j\}$ is an orthonormal basis in $U_1$.
\vskip1mm
From the above considerations
$$
\sum_{j=1}^{\infty}(Q_1h_j,h_j)_{_{U_1}}
\le t\sum_{j=1}^{\infty}|g_j|_{U_1}^2 \quad \mbox{and then}\quad
\sum_{j=1}^{\infty}(Q_1h_j,h_j)_{_{U_1}}\equiv \mathrm{Tr}\,Q_1<+\infty.
$$
\qed  }

\begin{proposition} \label{prop2}
For any $a\in U$ the process
\begin{equation}\label{PBe3}
\left(a,W_c(t)\right)_U=\sum_{j=1}^{\infty}\left(a,g_j\right)_U\beta_j(t)
\end{equation}
is real-valued Wiener process and
$$
E\left(a,W_c(s)\right)_U\left(b,W_c(t)\right)_U=(s\land t)\left(Qa,b\right)_U
\qquad\mbox{for }\ a,b\in U.
$$
Additionally, $\mathrm{Im}\, Q_1^{\frac12}=U_0$ and $| u|_{U_0}=
\left| Q_1^{-\frac12}u\right|_{U_1}$.
\end{proposition}

\proof{
We shall prove that the series (\ref{PBe3}) defining the process
$\left(a,W_c(t)\right)_U$
is convergent in the space $L^2(\Omega;\mathbb{R})$.
\index{$L^2(\Omega;\mathbb{R})$}
\vskip2mm
Let us notice that the series (\ref{PBe3}) is the sum of independent random 
variables with zero mean. Then the series does converge in 
$L^2(\Omega;\mathbb{R})$ if and
only if the following series
$\sum_{j=1}^{\infty}E\left(\left(a,g_j\right)_U\beta_j(t)\right)^2
$ converges.
\vskip1mm
Because $J$ is Hilbert--Schmidt operator, we obtain
\begin{eqnarray*}
\sum_{j=1}^{\infty}E\left(\left(a,g_j\right)_U^2\beta_j^2(t)\right)&=&
\sum_{j=1}^{\infty}\left(a,g_j\right)_U^2
\le |a|_U^2\sum_{j=1}^{\infty}\left|g_j\right|_U^2\\
&\le& C|a|_U^2\sum_{j=1}^{\infty}\left|Jg_j\right|_{U_1}^2<
+\infty.
\end{eqnarray*}
Hence, the series (\ref{PBe3}) does converge.
Moreover, when $t\ge s\ge 0$, we have
\begin{eqnarray*}
E\left(\left(a,W_c(t)\right)_U\left(b,W_c(s)\right)_U\right)&=&
E\left(\left(a,W_c(t)-W_c(s)\right)_U\left(b,W_c(s)\right)_U\right)\\
&+&E\left(\left(a,W_c(s)\right)_U\left(b,W_c(s)\right)_U\right)\\
&=&E\left(\left(a,W_c(s)\right)_U\left(b,W_c(s)\right)_U\right)\\
&=&E\left(\left[\sum_{j=1}^{\infty}\left(a,g_j\right)_U\beta_j(s)\right]
\left[\sum_{k=1}^{\infty}\left(b,g_k\right)_U\beta_k(s)\right]\right).
\end{eqnarray*}
Let us introduce
$$
S^a:=\sum_{j=1}^{\infty}(a,g_j)_{_U}\beta_j(t),\quad 
S^b:=\sum_{k=1}^{\infty}(b,g_k)_{_U}\beta_k(t),\ \mbox{for}\ a,b\in U.
$$
Next, let $S_N^a$ and $S_N^b$ denote the partial sums of the series $S^a$
and $S^b$, respectively.
From the above considerations the series $S^a$ and $S^b$ are convergent
in $L^2(\Omega;\mathbb{R})$.
Hence
$E(S^aS^b)=\lim_{N\to\infty}E(S_N^aS_N^b).$
In fact,
\begin{eqnarray*}
E|S_N^aS_N^b-S^aS^b|&=&E|S_N^aS_N^b-S_n^aS^b+S^bS_n^a-S^bS^a|\\
&\le& E|S_N^a||S_N^b-S^b|+E|S^b||S_N^a-S^a|\\
&\le&\left(E|S_N^a|^2\right)^{\frac12}\left(E|S_N^b-S^b|^2\right)^{\frac12}\\
&+&\left(E|S^b|^2\right)^{\frac12}\left(E|S_N^a-S^a|^2\right)^{\frac12}
~\longrightarrow_{\hspace{-4ex}_{N\to\infty}} 0
\end{eqnarray*}
because $S_N^a$ converges to $S^a$ and $S_N^b$ converges to $S^b$ in
quadratic mean.
\vskip1mm
Additionally,
$E(S_N^aS_N^b)=t\sum_{j=1}^{N}(a,g_j)_{_U}(b,g_j)_{_U}$
and when $N\to +\infty$
$$
E(S^aS^b)=t\sum_{j=1}^{\infty}(a,g_j)_{_U}(b,g_j)_{_U}.
$$
Let us notice that
\begin{eqnarray*}
\left(Q_1a,b\right)_{U_1}&=&E\left(a,W_c(1)\right)_{U_1}\left(b,W_c(1)
\right)_{U_1}
=\sum_{j=1}^{\infty}\left(a,g_j\right)_{U_1}\left(b,g_j\right)_{U_1}\\
&=&\sum_{j=1}^{\infty}\left(a,Jg_j\right)_{U_1}\left(b,Jg_j\right)_{U_1}
=\sum_{j=1}^{\infty}\left(J^*a,g_j\right)_{U_0}\left(J^*b,g_j\right)_{U_0}\\
&=&\left(J^*a,\ \sum_{j=1}^{\infty}\left(J^*b,g_j\right)g_j\right)_{U_0}
=\left(J^*a,J^*b\right)_{U_0}=\left(JJ^*a,b\right)_{U_1}.
\end{eqnarray*}
That gives
$Q_1=JJ^*.$
In particular
\begin{equation}\label{PBe4}
\left|Q_1^{\frac12}a\right|_{U_1}^2=\left(JJ^*a,a\right)_{U_1}
=\left|J^*a\right|_{U_0}^2,\qquad a\in U_1.
\end{equation}
\vskip2mm
Having (\ref{PBe4}), we can use theorems on images of linear operators, e.g.
\cite[Appendix B.2, Proposition B.1 (ii)]{DZ92}.
\vskip1mm
By that proposition
$\mathrm{Im}\, Q_1^{\frac12}=\mathrm{Im}\, J.$
But for any $j\in\mathbb{N}$, and $g_j\in U_0$, $Jg_j=g_j$, 
that is $\mathrm{Im}\, J=U_0$.
Then $\mathrm{Im}\, Q_1^{\frac12}=U_0.$
\vskip1mm
Moreover, the operator $G=Q_1^{-\frac12}J$ is a bounded operator from $U_0$
on $U_1$. From (\ref{PBe4}) the adjoint operator
$G^*=J^*Q_1^{-\frac12}$
is an isometry, so $G$ is isometry, too. Thus
$$
\left|Q_1^{-\frac12}u\right|_{U_1}=\left|Q_1^{-\frac12}Ju\right|_{U_1}=
|u|_{_{U_0}}.
$$
\qed
}\vskip2mm

In the case when $Q$ is nuclear operator, $Q^{\frac12}$ is Hilbert-Schmidt
operator. Taking $U_1=U$, the process $W_c(t)$, $t\ge 0$, defined by (\ref{PBe2}) is
the classical Wiener process introduced in Definition \ref{def1}. 
\vskip5mm
\begin{definition} \label{def3}
The process $W_c(t)$, $t\ge 0$, defined in (\ref{PBe2}), 
is called {\tt cylin\-drical Wiener process \index{Wiener process!cylindrical}
on $U$} when $\mathrm{Tr}\,Q=+\infty$.
\end{definition}

The stochastic integral with respect to cylindrical Wiener process is defined
as follows.
\vskip2mm
As we have already written above, the process $W_c(t)$ defined by (\ref{PBe2})
is a Wiener process in
the space $U_1$ with the covariance operator $Q_1$ such
that $\mathrm{Tr}\,Q_1<+\infty$. Then the stochastic integral
$\int_0^tg(s)dW_c(s)\in H$ is well defined in $U_1$, 
where $g(s)\in L(U_1,H)$. 
\vskip2mm
Let us notice that $U_1$ is not uniquely determined. The space $U_1$ can be
an arbitrary Hilbert space such that $U$ is continuously embedded into $U_1$
and the embedding of $U_0$ into $U_1$ is a Hilbert-Schmidt operator. We would
like to define the stochastic integral with respect to cylindrical Wiener
proces $W_c(t)$ (given by (\ref{PBe2})) in such a way that the integral is well
defined on the space $U$ and does not depend on the choice of the space
$U_1$.
\vskip5mm
We denote by $\mathcal{N}^2(0,T;L_2^0)$ the space of all stochastic processes
\index{$\mathcal{N}^2(0,T;L_2^0)$}
\begin{equation}\label{PBe5}
\Phi:\ [0,T]\times \Omega \to L_2(U_0,H)
\end{equation}
such that
\begin{equation}\label{PBe6}
||\Phi||_T^2 := 
E\left(\int_0^T ||\Phi (t) ||_{L_2(U_0,H)}^2dt\right)<+\infty
\end{equation}
and for all $u\in U_0$, $(\Phi (t)u)$, $t\in [0,T]$, is an 
$H$--valued and $\mathcal{F}_t$--adapted stochastic process.
\vskip1mm
The stochastic integral $\int_0^t\Phi (s)dW_c(s)\in H$ with respect to
cylindrical Wiener process, given by (\ref{PBe2}) for any process 
$\Phi \in \mathcal{N}^2(0,T;L_2^0)$,
can be defined as the limit
\begin{equation}\label{PBe7}
\int_0^t\Phi (s)dW_c(s)=\lim_{m\to \infty}\sum_{j=1}^m
\int_0^t\Phi (s)g_jd\beta_j(s)
\end{equation}
in $L^2(\Omega;H)$. \vskip1mm

\noindent
{\bf Comment} Before we prove that the stochastic integral given
by the formula (\ref{PBe7}) is well defined, let us recall properties of the
operator $Q_1$. From Proposition \ref{prop1}, 
cylindrical Wiener process $W_c(t)$
given by (\ref{PBe2}) has the covariance operator $Q_1:\ U_1 \to U_1$, which
is a nuclear operator in the space $U_1$, that is $\mathrm{Tr}\,Q_1<+\infty$.
Next, basing on Proposition \ref{prop2}, $Q_1^{\frac12}:\ U_1 \to U_0$,
$\mathrm{Im}\, Q_1^{\frac12}=U_0$ and 
$|u|_{U_0}=\left|Q_1^{-\frac12}u\right|_{U_1}$
for $u\in U_0$.

 Moreover, from the above considerations and properties of the operator
$Q_1$  we may deduce that $L(U_1,H)\subset L_2(U_0,H)$.
This means that each
operator $\Phi \in L(U_1,H)$, that is linear and bounded from $U_1$ into $H$,
is Hilbert-Schmidt operator acting from $U_0$ into $H$, that is
$\Phi \in L_2(U_0,H)$ when $\mathrm{Tr}\,Q_1<+\infty$ in $U_1$. This means that
conditions (\ref{PBe5}) and (\ref{PBe6}) for the family $\mathcal{N}^2(0,T;L_2^0)$ 
of integrands are natural
assumptions for the stochastic integral given by (\ref{PBe7}).

Now, we shall prove that the series from the right hand side of (\ref{PBe7}) is
convergent.
\vskip2mm
Denote
$$
W_c^{(m)}(t):=\sum_{j=1}^mg_j\beta_j(t)
$$
and
$$
Z_m:=\int_0^t\Phi (s)W_c^{(m)}(s),\qquad t\in [0,T].
$$
Then, we have
\begin{eqnarray*}
E\left(\left|Z_n-Z_m\right|_H^2\right)
& = & E\left|\sum_{j=m+1}^n\int_0^t\Phi (s)g_jd\beta_j
(s)\right|_H^2\qquad\mbox{for }\ n\ge m\ge 1\\
&\le & E\sum_{j=m+1}^n\int_0^t\left|\Phi(s)g_j\right|_H^2ds
~~\longrightarrow_{\hspace{-4.5ex}_{m,n\to\infty}}  0,
\end{eqnarray*}
because from the assumption (\ref{PBe6})
$$
E\int_0^t\left(\sum_{j=1}^{\infty}\left|\Phi (s)g_j\right|_H^2\right)ds<
+\infty.
$$
\vskip5mm
Then, the sequence $(Z_m)$ is Cauchy sequence in the space of
square--integrable random variables. So, the stochastic integral with respect
to cylindrical Wiener process given by (\ref{PBe7}) is well defined.
\vskip2mm
As we have already mentioned, the space $U_1$ is not uniquely determined.
Hence, the cylindrical Wiener proces $W_c(t)$ defined by (\ref{PBe2})
is not uniquely determined either.
\vskip2mm
Let us notice that the stochastic integral defined by (\ref{PBe7}) does not depend
on the choice of the space $U_1$. Firstly, in the formula (\ref{PBe7}) there are not
elements of the space $U_1$ but only $\{g_j\}$--basis of $U_0$.
Additionally, in (\ref{PBe7}) there are not eigenfunctions of the covariance operator
$Q_1$.
Secondly, the class $\mathcal{N}^2(0,T;L_2^0)$ of integrands  does not depend on the choice of
the space $U_1$ because (by Proposition \ref{prop2}) the spaces $Q_1^{\frac12}(U_1)$
are identical for any spaces $U_1$:
$$
Q_1^{\frac12}:\ U_1\to U_0\qquad\mbox{and}\qquad 
\mathrm{Im}\, Q_1^{\frac12}=U_0.
$$

\subsection{Properties of the stochastic integral}\label{Ch2s3ss1}

In this subsection we recall from \cite{DZ92} some properties of stochastic 
integral used in the paper.

\begin{proposition}\label{Ch2p2}
 Assume that $\Phi \in \mathcal{N}^2(0,T;L_0^2)$. Then the stochastic integral
 $\Phi\bullet W(t) := \int_0^t \Phi(s) dW(s) $ is a continuous square integrable
 martingale and its quadratic variation is of the form 
 $<\!<\Phi\bullet W(t)>\!> = \int_0^t Q_\Phi (s) ds$, where 
 $Q_\Phi (s)= (\Phi(s) Q^{1/2})(\Phi(s) Q^{1/2})^*, ~s,t\in[0,T]$.
\end{proposition}
\index{$\Phi\bullet W(t)$}

\begin{proposition}\label{Ch2p3}
 If  $\Phi \in \mathcal{N}^2(0,T;L_0^2)$, then $\mathbb{E}(\Phi\bullet W(t))=0$
 and\\ $\mathbb{E}|\Phi\bullet W(t)|_H^2<+\infty, ~t\in [0,T]$.
\end{proposition}

\begin{proposition}\label{Ch2p4}
 Assume that $\Phi_1, \Phi_2 \in \mathcal{N}^2(0,T;L_0^2)$. Then the correlation
 operators $V(s,t):= \mathrm{Cor}(\Phi_1\bullet W(s),\Phi_2\bullet W(t))$,
 $~s,t\in [0,T]$ are given by the formula 
 $V(s,t)=\mathbb{E}\int_0^{s\wedge t} (\Phi_2(r)Q^{\frac{1}{2}})
 (\Phi_1(r)Q^{\frac{1}{2}})^* dr$.
\end{proposition}

\begin{corollary}\label{Ch2co1}
 From the definition of the correlation operator we have 
 $\mathbb{E}(\Phi_1\bullet W(s), \Phi_2\bullet W(t))_H=
 \mathbb{E}\int_0^{s\wedge t} \mathrm{Tr} [(\Phi_2(r)Q^{\frac{1}{2}})
 (\Phi_1(r)Q^{\frac{1}{2}})^*] dr $.
\end{corollary}

\section{The stochastic Fubini theorem\index{theorem!stochastic Fubini}
 and the It\^o formula\index{formula!It\^o}}\label{PBs:4}
 
The below theorems are recalled directly from the book by Da Prato and Zabczyk
\cite{DZ92}.

Assume that $(\Omega,\mathcal{F},(\mathcal{F}_t)_{t\ge 0},P)$ is 
a probability space, 
$\Omega_T:=[0,T]\times\Omega$ and recall that 
$\mathcal{P}_T$ is the $\sigma$-field defined in section \ref{PBsec:1}, 
that is $\mathcal{P}_T$ is the $\sigma$-field generated by sets of the form:
$(s,t]\times F$, where $0\le s\le t\le T$, $F\in \mathcal{F}_s$ and
$\{0\}\times F$, when $F\in \mathcal{F}_0$.\vskip1mm

Let $(E,\mathcal{E})$ be a measurable space and let 
\begin{eqnarray}\label{DZ433}
 & ~&\Phi : (t,\omega,x) \rightarrow \Phi(t,\omega,x)~~ be~a~measurable~
  mapping~from \nonumber\\
 &~& (\Omega_T\times E, \mathcal{P}_T\times\mathcal{B}(E))~~ into~~
 (L_2^0,\mathcal{B}(L_2^0))\;,\end{eqnarray}
 where $\mathcal{B}(E)$ and $\mathcal{B}(L_2^0)$ denote Borel $\sigma$-fields 
 on $E$ and $L_2^0$, respectively.
Thus, in particular, for arbitrary $x\in E,~\Phi(\cdot,\cdot,x)$ is a
predictable $L_2^0$-valued process. Let in addition $\mu$ be a finite positive
measure on $(E,\mathcal{E})$.

\index{theorem!stochastic Fubini} 
\begin{theorem}\label{PBtFub} (The stochastic Fubini theorem)\\
Assume (\ref{DZ433}) and that
$$ 
  \int_E ||\Phi(\cdot,\cdot,x)||_T \,\mu(dx) <+\infty 
$$ 
 Then P-a.s.
$$ 
 \int_E\left[ \int_0^T \Phi(t,x)\, dW(t)\right] \mu(dx) =
 \int_0^T \left[\int_E \Phi(t,x)\, \mu(dx)\right] dW(t)\;.
$$ 
\end{theorem}

\vspace{4mm}
Assume that $\Phi$ is an $L_2^0$-valued process stochastically integrable in
$[0,T]$, $\phi$ is an $H$-valued predictable process Bochner integrable on
$[0,T]$, P-a.s., and $X(0)$ is an $\mathcal{F}_0$-measurable $H$-valued random
variable. Then the following process
$$ 
  X(t)=X(0)+\int_0^t \phi(s)\,ds + \int_0^t \Phi(s)\,dW(s), \qquad t\in [0,T], 
$$ 
is well defined. Assume that a function $F: [0,T]\times H\rightarrow R^1$ and
its partial derivatives $F_t,F_x,F_{xx}$, are uniformly continuous on bounded
subsets of  \linebreak $[0,T]\times H$. 

\pagebreak
\begin{theorem}\label{PBtIto} (The It\^o formula)\index{formula!It\^o}\\
Under the above conditions
 \begin{eqnarray*} 
  F(t,X(t)) & = & F(0,X(0))  
  + \int_0^t \langle F_x(s,X(s)),\Phi(s)\,dW(s)\rangle_H \\
&&+ \int_0^t \left\{ F_t(s,X(s))+ \langle F_x(s,X(s)), \phi(s)\rangle_H \right. \\
&&+ \frac{1}{2} \mathrm{Tr} [ F_{xx}(s,X(s))(\Phi(s)Q^{\frac{1}{2}})
  (\Phi(s)Q^{\frac{1}{2}})^*]\}ds.
 \end{eqnarray*}
holds P-a.s., for all $t\in [0,T]$. 
\end{theorem}

\chapter{Stochastic Volterra equations in Hilbert space} \label{SVEHSch:1} 

The aim of this chapter is to study some fundamental questions related to 
the linear convolution type stochastic Volterra equations of the form 
\begin{equation} \label{SVEHe1} 
 X(t) = X(0) + \int_0^t a(t-\tau) AX(\tau)\, d\tau 
 + \int_0^t \Psi(\tau)\,dW(\tau), \quad  t\geq 0,
\end{equation}
in a separable Hilbert space $H$. 
Particularly, we provide sufficient conditions for the existence of 
strong solutions to 
some classes of the  equation (\ref{SVEHe1}), which is a stochastic
version of the equation (\ref{DVEe1}). 

Let $(\Omega,\mathcal{F},(\mathcal{F}_t)_{t\ge 0},P)$ be a probability
space. In (\ref{SVEHe1}), the kernel function $a(t)$ and the operator 
$A$ are the same as
previously, $X(0)$ is an $H$-valued $\mathcal{F}_0$-measurable random variable,
$W$ is a cylindrical Wiener process on a separable Hilbert space $U$ and $\Psi$
is an appropriate process defined below.

This chapter is organized as follows. In section \ref{SVEHSsec:1} we give
definitions of solutions to (\ref{SVEHe1}) and some introductory results
concerning stochastic convolution arising in (\ref{SVEHe1}). 
Additionally, we show that under some
conditions a weak solution to (\ref{SVEHe1}) is a mild solution and vice
versa. These results have been recalled from \cite{Ka05}.

Section \ref{SVEHSsec:2} deals with strong solution to (\ref{SVEHe1}). We 
formulate  sufficient conditions for a stochastic convolution to be 
a strong solution to (\ref{SVEHe1}). The above results come from the 
paper \cite{KL07c}, not yet published.

In section \ref{SVEHSsec:3}, based on \cite{KL07d}, we study particular class
of equations (\ref{SVEHe1}), that is, so-called fractional Volterra equations.
We decided to consider that class of equations separately
because of specific problems appearing during the study of such equations.
First, we formulate the deterministic results which play the key role for
stochastic results. We study in detail $\alpha$-times resolvent
families corresponding to fractional Volterra equations. Next, we consider
mild, weak and strong solutions to those equations.\\[-1mm]

In the whole chapter we shall use the following 
{\sc Volterra Assumptions} \index{{\sc Volterra Assumptions}, (VA)}
(abbr. ({\sc VA})): 
\begin{enumerate}
\item $A:D(A)\subset H\rightarrow H$, is a closed linear operator
with the dense domain $D(A)$ equipped with the graph norm $|\cdot|_{D(A)}$;
\item $a\in L_\mathrm{loc}^1(\mathbb{R}_+)$ is a scalar kernel;
\item $S(t),~t\geq 0$, are resolvent operators for the Volterra equation
(\ref{DVEe1}) 
determined by the operator $A$ and the function $a(t), ~t\ge 0$.
\end{enumerate}

The domain $D(A)$ is equipped with the graph norm defined as follows:
$|h|_{D(A)}:= (|h|_H^2+|Ah|_H^2)^{\frac{1}{2}}$ for $h\in D(A)$, where 
$|\cdot|_H$ denotes a norm in $H$.
Because $H$ is a separable Hilbert space and $A$ is a closed operator, the
space $(D(A),|\cdot|_{D(A)})$ is a separable Hilbert space, too.

$W(t),~t\geq 0$, is 
a cylindrical Wiener process on $U$ with the covariance operator $Q$
and $\mathrm{Tr}\,Q=+\infty$.

By $L_2^0:=L_2(U_0,H)$, as previously, we denote the set of all 
Hilbert-Schmidt operators acting from $U_0$ into $H$, where
$U_0=Q^{\frac{1}{2}}(U)$. \\

For shortening, we introduce 
{\sc Probability  Assumptions} (abbr. ({\sc PA})):
\index{{\sc Probability  Assumptions}, (PA)}
\begin{enumerate}
\item $X(0)$ is an $H$-valued, $\mathcal{F}_0$-measurable random variable;
\item $\Psi$ belongs to the space $\mathcal{N}^2(0,T;L_2^0)$, where the finite
interval $[0,T]$ is fixed.
\end{enumerate}

\section{Notions of solutions to stochastic Volterra equations} 
\label{SVEHSsec:1}
\sectionmark{Notions of solutions to sVe}

In this section we introduce the definitions of solutions to the stochastic 
Volterra equation (\ref{SVEHe1}) and then formulate some results, 
not yet published, setting a framework for further research.

\begin{definition} \label{def5}
Assume that conditions (VA) and (PA) hold. 
An $H$-valued predictable process $X(t),~t\in [0,T]$, is said to be a 
{\tt strong solution}\index{solution!strong} to  (\ref{SVEHe1}), if $X$ 
has a version such that $P(X(t)\in D(A))=1$ for almost all $t\in [0,T]$;
for any $t\in [0,T]$
 \begin{equation} \label{eSW3.1}
 \int_0^t |a(t-\tau)AX(\tau)|_H \,d\tau<+\infty,\quad P-a.s.
 \end{equation}
and for any $t\in [0,T]$ the equation (\ref{SVEHe1}) holds $P-a.s$.
\end{definition}

\noindent{\bf Comment} Because the integral
$\int_0^t \Psi(\tau)\,dW(\tau)$ is a continuous $H$-valued process
then the above definition yields continuity of the strong solution.\\

\index{$A^*$}
\index{$\vert\cdot \vert_{D(A^*)}$}
\index{$D(A^*)$}
Let $A^*$ denotes the adjoint of the operator $A$, with dense domain 
$D(A^*)\subset H$ and the graph norm $|\cdot |_{D(A^*)}$ defined as follows:
$|h|_{D(A^*)}:= (|h|_H^2+|A^*h|_H^2)^{\frac{1}{2}}$, for $h\in D(A^*)$. The
space $(D(A^*),|\cdot|_{D(A^*)})$ is a separable Hilbert space.

\begin{definition} \label{def6}
Let conditions (VA) and (PA) hold.
An $H$-valued predictable process $X(t),~t\in [0,T]$, is said to be a 
{\tt weak solution}\index{solution!weak} to  (\ref{SVEHe1}), if \linebreak
$P(\int_0^t|a(t-\tau)X(\tau)|_H d\tau<+\infty)=1$
and if for all $\xi\in D(A^*)$ and all 
$t\in [0,T]$ the following equation holds
\begin{eqnarray*}
 \langle X(t),\xi\rangle_H = \langle X(0),\xi\rangle_H & + &
 \langle \int_0^t a(t-\tau)X(\tau)\,d\tau, A^*\xi\rangle_H\\ 
 & + &  \langle \int_0^t \Psi(\tau)\,dW(\tau),\xi\rangle_H, 
 \quad P\mathrm{-a.s.}
\end{eqnarray*}
\end{definition}
\begin{definition} \label{def7}
Assume that (VA) are satisfied and $X(0)$ is an $H$-valued 
$\mathcal{F}_0$-measurable random variable.
An $H$-valued predictable process $X(t),~t\in [0,T]$, is said to be a 
{\tt mild solution}\index{solution!mild} to the stochastic 
Volterra equation (\ref{SVEHe1}), if
\begin{equation}\label{deq40}
\mathbb{E}\left(\int_0^t ||S(t-\tau)\Psi(\tau)||_{L_2^0}^2 d\tau \right)
< +\infty \quad for \quad t\leq T
\end{equation}
and, for arbitrary $t\in [0,T]$,
\begin{equation}\label{deq4}
 X(t) = S(t)X(0) + \int_0^t S(t-\tau)\Psi(\tau)\,dW(\tau),\quad P-a.s.
\end{equation}
\end{definition}

We will show that 
in some cases weak solution to the equation (\ref{SVEHe1}) 
coincides with mild solution to (\ref{SVEHe1}) (see, subsection \ref{ss313}). 
In consequence, having results for the convolution 
\begin{equation}\label{deq5}
 W^\Psi(t) :=  \int_0^t S(t-\tau)\Psi(\tau)\,dW(\tau),\quad t\in [0,T],
\end{equation}
where $S(t)$ and $\Psi$ are the same as in (\ref{deq4}),
we will obtain results for weak solution to (\ref{SVEHe1}), too.
\index{$W^\Psi(t)$}


\subsection{Introductory results}\label{ss312}

In this section we collect some properties of the stochastic convolution
 of the form \index{stochastic convolution}
\begin{equation} \label{deq6} 
 W^B(t) := \int_0^t S(t-\tau)B\,dW(\tau)
\end{equation} 
in the case when $B\in L(U,H)$ and $W$ is a cylindrical Wiener process.

\begin{lemma} \label{dl1} 
 Assume that the operators $S(t),~t\geq 0$, and $B$ are as above,
 $S^*(t), B^*$ are their adjoints, and
\begin{equation} \label{deq7} 
  \int_0^T ||S(\tau)B||_{L_2^0}^2\,d\tau =  \int_0^T \mathrm{Tr} 
  [S(\tau) BQB^*S^*(\tau)]\,d\tau < +\infty.
\end{equation} 
Then we have:
\begin{enumerate}
 \item the process $W^B$ is Gaussian, mean-square continuous on [0,T] 
 and then has a predictable version;
 \item 
\begin{equation} \label{deq8} 
  \mathrm{Cov}~W^B(t) =  \int_0^t [S(\tau) BQB^*S^*(\tau)]\,d\tau, \quad 
  t\in [0,T];
\end{equation} 
 \item trajectories of the process $W^B$ are P-a.s.\ square integrable
 on [0,T].
\end{enumerate}
\end{lemma}

\proof{\begin{enumerate}
 \item Gaussianity of the process $W^B$ follows from the definition and
 properties of stochastic integral. 
 Let us fix $0\leq t < t+h \leq T$. Then
\begin{eqnarray*} 
 W^B(t+h)-W^B(t) &=& \int_0^t [S(t+h-\tau)-S(t-\tau)]BdW(\tau)\\ &+&
 \int_t^{t+h} \!\!  S(t+h-\tau)BdW(\tau).
\end{eqnarray*}
Let us note that the above integrals are stochastically independent.
Using the extension of the process $W$ (mentioned in section \ref{SVEHSsec:1})
and properties of stochastic integral with respect to real Wiener processes
(see e.g.\ \cite{Ic82}), we have
\begin{eqnarray*} 
 \mathbb{E} | W^B(t+h)-W^B(t) |_H^2 & =& \sum_{k=1}^{+\infty} \lambda_k
 \int_0^t |[ S(t+h-\tau)-S(t-\tau)]Be_k |_H^2 \,d\tau \\
 & + & \sum_{k=1}^{+\infty} \lambda_k \int_t^{t+h} 
 | S(t+h-\tau)Be_k|_H^2 \,d\tau\\
 & := & I_1(t,h) + I_2(t,h) \;.
\end{eqnarray*}
Then, invoking (\ref{DVEe3}), the strong continuity of $S(t)$ and the 
Lebesgue dominated convergence theorem, we can pass in $I_1(t,h)$ with 
$h\rightarrow 0$ under the sum and integral signs. 
Hence, we obtain $I_1(t,h) \to 0$ as $h\to 0$.  

Observe that
$$
 I_2(t,h) = \int_t^{t+h} || S(t+h-\tau)BQ^{\frac{1}{2}} ||_{HS}^2 \,d\tau \;, 
$$
where $||\cdot||_{HS}$ denotes the norm of Hilbert-Schmidt operator.
By the condition (\ref{deq7}) we have 
$$
 \int_0^T || S(t) B Q^{\frac{1}{2}} ||_{HS}^2 \,dt < +\infty \;,
$$
what implies that $\lim_{h\rightarrow 0} I_2(t,h) =0$.

The proof for the case $0 \leq t-h < t \leq T$ is similar.
Existence of a predictable version is a consequence of the continuity 
and Proposition~\ref{PBp1a}. 
\item The shape of the
covariance (\ref{deq8}) follows from theory of stochastic integral,
see e.g.\ \cite{DZ92}.
\item From the definition (\ref{deq6}) and assumption (\ref{deq7}) we have 
 the following estimate
\begin{eqnarray*}
\mathbb{E} \int_0^T |W^B(\tau)|_H^2\,d\tau =
  \int_0^T \! \mathbb{E}|W^B(\tau)|_H^2\,d\tau & = &\\
 = \int_0^T \!\mathbb{E}\, \left|\int_0^\tau \!S(\tau-r)BdW(r)\right|_H^2\!\!d\tau  
 \!& = \!& \!\int_0^T\!\! \int_0^\tau \!\!
 ||S(r) B||_{L_2^0}^2 \,dr\,d\tau \!<\! +\infty\,.
\end{eqnarray*}
Hence, the function $W^B(\cdot)$ may be regarded like random variable with values
in the space $L^2(0,T;H)$.
\qed\\
\end{enumerate}
}

Now, we formulate an auxiliary result which will be used in the next section.

\begin{lemma} \label{dl2}
~Let\/ {\sc Volterra assumptions} hold with the function 
$a \in \,W^{1,1}(\mathbb{R}_+)$. 
Assume that $X$ is a weak solution to (\ref{SVEHe1}) in the case when 
 $\Psi(t)=B$, where $B\in L(U,H)$ and trajectories of $X$
are integrable P-a.s.\ on $[0,T]$. Then, for any function 
$\xi\in C^1([0,t];D(A^*))$, $t\in [0,T]$, the following formula holds
\begin{eqnarray} \label{deq10} 
  \langle X(t),\xi(t)\rangle_H & = & \langle X(0), \xi(0) \rangle_H +
  \int_0^t  \langle (\dot{a}\star X)(\tau)
  + a(0)X(\tau),A^*\xi(\tau)\rangle_H  d\tau  \nonumber \\ & + &
  \int_0^t \langle \xi(\tau),BdW(\tau)\rangle_H  
  +\int_0^t \langle X(\tau),\dot{\xi}(\tau)\rangle_H d\tau,
\end{eqnarray} 
where dots above $a$ and $\xi$ mean time derivatives and $\;\star\,$ means the
convolution.
\end{lemma}

\proof{First, we consider functions of the form 
$\xi(\tau):=\xi_0\varphi(\tau)$,
$\tau\in [0,T]$, where $\xi_0\in D(A^*)$ and $\varphi\in C^1[0,T]$.
For simplicity we omit index $_H$ in the inner product. 
Let us denote 
$ F_{\xi_0}(t) := \langle X(t),\xi_0 \rangle,~ t\in[0,T]. $
 
Using It\^o's formula to the process $ F_{\xi_0}(t)\varphi(t)$, we have
\begin{equation} \label{deq11} 
 d[F_{\xi_0}(t)\varphi(t)] = \varphi(t)dF_{\xi_0}(t) 
 + \dot{\varphi}(t) F_{\xi_0}(t)dt, \quad\quad t\in[0,T].
\end{equation} 
Because $X$ is weak solution to (\ref{SVEHe1}), we have 
\begin{eqnarray} \label{deq12} 
 dF_{\xi_0}(t) & = &\langle\int_0^t \dot{a}(t-\tau)X(\tau)d\tau + a(0)X(t),
 A^*\xi_0\rangle dt +\langle BdW(t),\xi_0\rangle \nonumber \\
 & = & \langle (\dot{a}\star X)(t) +a(0)X(t),A^*\xi_0\rangle dt + 
 \langle BdW(t),\xi_0 \rangle .
\end{eqnarray} 
From (\ref{deq11}) and (\ref{deq12}), we obtain
\begin{eqnarray*}  
 F_{\xi_0}(t)\varphi(t) & = &  F_{\xi_0}(0)\varphi(0) +
 \int_0^t \varphi(s) \langle (\dot{a}\star X)(s) 
 + a(0)X(s), A^*\xi_0\rangle ds \\ &&  + \int_0^t 
 \langle\varphi(s)BdW(s),\xi_0\rangle + \int_0^t \dot{\varphi}(s)
 \langle X(s), \xi_0\rangle ds \\
  & = & \langle X(0), \xi(0) \rangle_H + 
  \int_0^t \langle (\dot{a}\star X)(s)
 + a(0)X(s), A^*\xi(s)\rangle ds \\ 
 && + \int_0^t \langle BdW(s),\xi(s)\rangle 
 +\int_0^t \langle X(s),\dot{\xi}(s)\rangle ds .
\end{eqnarray*} 
Hence, we proved the formula (\ref{deq10}) for functions $\xi$ of the form 
$\xi(s)=\xi_0\varphi(s)$, $s\in [0,T]$. Because such functions form a dense 
subspace in the space  $C^1([0,T];D(A^*))$, the proof is completed.
\qed}

\subsection{Results in general case}\label{ss313}

In this subsection we consider weak and mild solutions to the equation 
(\ref{SVEHe1}). 

First we study the stochastic convolution defined by (\ref{deq5}),
that is,   \index{stochastic convolution}
$$
 W^\Psi(t) :=  \int_0^t S(t-\tau)\Psi(\tau)\,dW(\tau),\quad t\in [0,T].
$$

\begin{proposition} \label{pr3}
Assume that $S(t),~t\geq 0$, are (as earlier) the resolvent operators corresponding
to the Volterra equation (\ref{DVEe1}). Then, for arbitrary process 
$\Psi\in\mathcal{N}^2(0,T;L_2^0)$, the process $W^\Psi(t),~t\geq 0$, given by 
(\ref{deq5}) has a predictable version.
\end{proposition}

\proof{Because proof of Proposition \ref{pr3} is analogous to some schemes 
in the theory of stochastic integrals (see e.g.\ \cite[Chapter 4]{LS73}) 
we provide only an outline of the proof.

First, let us notice that the process $S(t-\tau)\Psi(\tau)$, where 
$\tau\in[0,t]$,
belongs to $\mathcal{N}^2(0,T;L_2^0)$, because
$\Psi\in\mathcal{N}^2(0,T;L_2^0)$.
Then we may use the apparently well-known estimate
(see e.g.\ Proposition 4.16 in \cite{DZ92}):  
for arbitrary $a>0,~b>0$ and $t\in[0,T]$ 
\begin{equation} \label{g1}
P(|W^\Psi(t)|_H>a) \leq \frac{b}{a^2}+P\left( \int_0^t  
 ||S(t-\tau)\Psi(\tau)||_{L_2^0}^2 d\tau >b\right)\;.
\end{equation}
Because the resolvent operators $S(t),~t\geq 0$, are uniformly bounded on compact
itervals (see \cite{Pr93}), there exists a constant $C>0$ such that 
$||S(t)||\leq C$ for $t\in[0,T]$. So, we have 
$||S(t-\tau)\Psi(\tau)||_{L_2^0}^2\leq C^2||\Psi(\tau)||_{L_2^0}^2$, 
$\tau\in[0,T]$.

Then the estimate (\ref{g1}) may be rewritten as 
\begin{equation} \label{g2}
P(|W^\Psi(t)|_H>a) \leq \frac{b}{a^2}+P\left( \int_0^t  
||\Psi(\tau)||_{L_2^0}^2 d\tau >\frac{b}{C^2}\right)\;.
\end{equation}

Let us consider predictability of the process $W^\Psi$ in two steps. 
In the first step we assume that $\Psi$ is an elementary process understood 
in the sense given in section 4.2 in \cite{DZ92}. In this case 
the process $W^\Psi$ has a predictable version by Proposition \ref{PBp1a}. 

In the second step $\Psi$ is an arbitrary process belonging to 
$\mathcal{N}^2(0,T;L_2^0)$. Since elementary processes form a dense 
set in the space $\mathcal{N}^2(0,T;L_2^0)$, there exists a sequence 
$(\Psi_n)$ of elementary processes such that for arbitrary $c>0$
\begin{equation} \label{g3}
P\left( \int_0^T ||\Psi(\tau)-\Psi_n(\tau)||_{L_2^0}^2 d\tau 
 >c\right)\stackrel{n\rightarrow +\infty}{\longrightarrow} 0\;.
\end{equation}

By the previous part of the proof the sequence $W_n^\Psi$ of convolutions
$$ W_n^\Psi (t):= \int_0^t S(t-\tau)\Psi_n(\tau)dW(\tau)$$
converges in probability. Hence, it has a subsequence converging almost
surely. This implies the predictability of the convolution 
$W^\Psi (t), ~t\in [0,T]$.
\hfill\qed}\\

\begin{proposition} \label{pr3a}
Assume that $\Psi\in\mathcal{N}^2(0,T;L_2^0)$. Then the process 
$W^\Psi(t)$, $t\in [0,T]$, defined by (\ref{deq5}) has square integrable 
trajectories.
\end{proposition}

\proof{We have to prove that $\mathbb{E}\int_0^T |W^\Psi(t)|_H^2 dt<+\infty$.
From Fubini's theorem and properties of stochastic integral
\begin{eqnarray*}
 \mathbb{E}\int_0^T \left| \int_0^t S(t-\tau)\Psi(\tau)dW(\tau)\right|_H^2 dt
  & =& \int_0^T \left[\mathbb{E}\left| 
 \int_0^t S(t-\tau)\Psi(\tau)dW(\tau)\right|_H^2\right] dt \\
 = \int_0^T\int_0^t ||S(t-\tau)\Psi(\tau)||_{L_2^0}^2\; d\tau dt & \leq &
 M \int_0^T\int_0^t ||\Psi(\tau)||_{L_2^0}^2 \; d\tau dt  < +\infty \\
   \mbox{(from~boundness~of operators~} S(t) &\mbox{and} &  
  \mbox{~because~} \Psi(\tau) \mbox{~are~Hilbert-Schmidt).} 
\end{eqnarray*}
\qed\\}

\begin{proposition} \label{pr4}
Assume that $a\in BV(\mathbb{R}_+)$, (VA) are satisfied and, additionally
$S\in C^1(0,+\infty;L(H))$.
Let $X$ be a predictable process with integrable trajectories.
Assume that $X$ has a version such that $P(X(t)\in D(A))=1$ 
for almost all $t\in [0,T]$ and (\ref{deq40}) holds.
If for any $t\in [0,T]$ and $\xi\in D(A^*)$
\begin{eqnarray}\label{deq9}
 \langle X(t),\xi\rangle_H = \langle X(0),\xi\rangle_H &+&
 \int_0^t \langle 
 a(t-\tau)X(\tau),A^*\xi\rangle_H d\tau \\
 &+& \int_0^t \langle\xi,\Psi(\tau) dW(\tau)\rangle_H,  ~~P-a.s.,
 \nonumber
\end{eqnarray}
then
\begin{equation}\label{deq9a}
X(t) = S(t)X(0) +
 \int_0^t S(t-\tau) \Psi (\tau) dW(\tau), \quad t\in[0,T].
\end{equation}
\end{proposition}

\proof{For simplicity we omit index $_H$ in the inner product. 
Since $a\in BV(\mathbb{R}_+)$, we see, analogously like in Lemma \ref{dl2}, 
that if (\ref{deq9}) is satisfied, then 
\begin{eqnarray}\label{deq14}
 \langle X(t),\xi(t)\rangle &=& \langle X(0),\xi(0)\rangle +
 \int_0^t \langle (\dot{a}\star X)(\tau) +
 a(0)X(\tau),A^*\xi(\tau)\rangle d\tau  \\
 &+& \int_0^t \langle \Psi(\tau) dW(\tau),\xi(\tau)\rangle  
 + \int_0^t \langle X(\tau),\dot{\xi}(\tau)  \rangle d\tau, \quad
 ~~P-a.s. \nonumber
\end{eqnarray}
holds for any $\xi\in C^1([0,t],D(A^*))$ for any $t\in [0,T]$.

Now, let us take $\xi(\tau):=S^*(t-\tau)\zeta$ with $\zeta\in D(A^*)$, 
 $\tau\in [0,t]$. 
The equation (\ref{deq14}) may be written like
\begin{eqnarray*}  
 \langle X(t),S^*(0)\zeta\rangle & = & \langle X(0),S^*(t)\zeta\rangle
 + \int_0^t \langle (\dot{a}\star X)(\tau)
 + a(0)X(\tau),A^*S^*(t-\tau)\zeta\rangle d\tau \nonumber \\
  &+& \int_0^t \langle \Psi(\tau) dW(\tau),S^*(t-\tau)\zeta\rangle
 + \int_0^t \langle X(\tau),(S^*(t-\tau)\zeta)'\rangle d\tau, 
\end{eqnarray*}
where derivative ()' in the last term is taken over $\tau$. 

Next, using $S^*(0)=I$, we rewrite
\begin{eqnarray} \label{deq15}
 \langle X(t),\zeta\rangle \!&=& \!\langle S(t)X(0),\zeta\rangle \!+\!
 \int_0^t \!\! \langle S(t-\tau)A\left[\int_0^\tau \!\!
 \dot{a}(\tau-\sigma)X(\sigma)d\sigma +a(0)X(\tau)\right]\!,\zeta\rangle 
 d\tau \nonumber\\
   \!&+&\! \int_0^t \!\langle S(t-\tau)\Psi(\tau)dW(\tau),\zeta\rangle +
   \int_0^t \langle \dot{S}(t-\tau)X(\tau),\zeta\rangle d\tau. 
\end{eqnarray}
To prove  (\ref{deq9a}) it is enough to show that the sum of the first 
integral and the third one in the equation (\ref{deq15}) gives zero.

Because $S\in C^1(0,+\infty;L(H))$
we can use properties of resolvent operators and the derivative 
$\dot{S}(t-\tau)$ with respect to $\tau$. Then
\begin{eqnarray*}  
 I_3 &:=& \langle \int_0^t \dot{S}(t-\tau)X(\tau) d\tau,\zeta \rangle =
  \langle -\int_0^t    \dot{S}(\tau)X(t-\tau) d\tau ,\zeta \rangle \\
 & = & \langle -\int_0^t\!\left[ \!\int_0^\tau\! 
 a(\tau-s)AS(s)ds\right]'\!X(t-\tau) d\tau ,\zeta\rangle \\
 &=& \langle\left(
  \!-\!\int_0^t \! \left[ \int_0^\tau \!\!\dot{a}(\tau-s)AS(s)ds\right]\!
 X(t-\tau) d\tau 
 \!-\! \int_0^t \! a(0)AS(\tau) X(t-\tau) d\tau\! \right)\! 
 ,\zeta \rangle \\[2mm]
 & = & \langle (- [A(\dot{a}\star S)(\tau) \star X](t)-a(0)A(S\star X)(t)) 
 ,\zeta \rangle .
\end{eqnarray*}
Note that $a\in BV(\mathbb{R}_+)$ and hence the convolution 
$(a\star S)(\tau)$ has sense (see \cite[Section 1.6]{Pr93} or\cite{AK89}).
Since
$$ \int_0^t \langle a(0)AS(t-\tau)X(\tau),\zeta\rangle d\tau =
  \int_0^t \langle a(0)AS(\tau)X(t-\tau),\zeta\rangle d\tau $$
 and
\begin{eqnarray*} 
I_1 \!&:=&\! \int_0^t \! \langle S(t\!-\!\tau)A\left[\int_0^\tau \!\!
 \dot{a}(\tau\!-\!\sigma)X(\sigma)d\sigma\right],\zeta\rangle d\tau \!=\!
 \int_0^t \! \langle AS(t\!-\!\tau)(\dot{a}\star X)(\tau),\zeta\rangle d\tau 
  \\[2.5mm]
 &\!=\!& \langle A(S\star (\dot{a}\star X)(\tau))(t),\zeta\rangle \!=\!
 \langle A((S\star\dot{a})(\tau)\star X)(t),\zeta\rangle
  \mbox{~~for~any~~} \zeta\in D(A^*),
\end{eqnarray*} 
so $I_1=-I_3$. 

This means that (\ref{deq9a}) holds for any $\zeta\in D(A^*)$. Since 
$D(A^*)$ is dense in $H^*$, then (\ref{deq9a}) holds. 
\qed\\}

\noindent{\bf Remark}
If (\ref{SVEHe1}) is parabolic and the kernel $a(t)$ is 3-monotone, understood
in the sense defined by Pr\"uss \cite[Section 3]{Pr93}, then $S\in
C^1(0,+\infty;L(H))$, and $a \in BV(\mathbb{R}_+)$, respectively.\\

\noindent{\bf Comment} 
Proposition \ref{pr4} shows that under particular conditions a weak 
solution to (\ref{SVEHe1}) is a mild solution to the equation (\ref{SVEHe1}).

\begin{proposition} \label{pr5}
Let {\sc Volterra assumptions} be satisfied. If  \linebreak
$\Psi \in \mathcal{N}^2(0,T;L_2^0)$,  then the stochastic   convolution 
$W^\Psi$ fulfills the equation (\ref{deq9}).
\end{proposition}

\proof{Let us notice that the process $W^\Psi$ has integrable 
trajectories. From the defintion of convolution (\ref{deq5}), using 
Dirichlet's formula and stochastic Fubini's theorem, for 
any $\xi\in D(A^*)$ we have
\begin{eqnarray*}
 \int_0^t \!\langle a(t-\tau)W^\Psi(\tau),A^*\xi\rangle_Hd\tau  
 &\!\equiv\!& \int_0^t\!\!  \langle a(t-\tau) \! 
 \int_0^\tau \! \!\!S(\tau-\sigma)\Psi(\sigma) 
 dW(\sigma),A^*\xi\rangle_Hd\tau \\[1mm]
 &\!=\!& \int_0^t \!\!\langle \left[\int_\sigma^t\!\!\! 
 a(t\!-\!\tau)S(\tau\!-\!\sigma)d\tau\right]
 \Psi(\sigma)  dW(\sigma),A^*\xi\rangle_H\\
 &\!=\!&  \int_0^t\!\!\! \langle A \left[\!\int_0^{t-\sigma}
 \!\!\!\!\!\! a(t\!-\!\sigma\!-\!z)S(z)dz\right]
 \!\Psi(\sigma) dW(\sigma),\xi\rangle_H\;.
\end{eqnarray*}
Next, using definition of convolution and the resolvent equation (\ref{deq2Pr}),
as \linebreak $A(a\star S)(t-\sigma)x= (S(t-\sigma)-I)x$, for $x\in H$, we can write 
\begin{eqnarray*} 
 \int_0^t \!\langle a(t-\tau)W^\Psi(\tau),A^*\xi\rangle_Hd\tau 
 & = & \langle \int_0^t \!\!A [ (a\star S)(t-\sigma)]\Psi(\sigma) dW(\sigma),
  \xi\rangle_H = \\
 & = & \langle \!\int_0^t\! [S(t-\sigma)-I]\Psi(\sigma) dW(\sigma),\xi\rangle_H = \\
   = \langle \!\int_0^t \!\!\!S(\!t-\!\sigma)\Psi(\sigma) dW(\sigma),\xi\rangle_H
 & - & \langle \int_0^t \Psi(\sigma) dW(\sigma),\xi\rangle_H .
\end{eqnarray*}

Hence, we obtained the following equation
$$
 \langle W^\Psi(t),\xi\rangle_H = \int_0^t \langle a(t-\tau)W^\Psi(\tau), 
 A^*\xi\rangle_Hd\tau + \int_0^t \langle \xi,\Psi(\tau)dW(\tau)\rangle_H
$$
for any $\xi\in D(A^*)$.
\qed}

\begin{corollary} \label{c1}
Let {\sc Volterra assumptions} hold with a bounded operator $A$. 
If $\;\Psi$ belongs to $\mathcal{N}^2(0,T;L_2^0)$  then
\begin{equation} \label{deq16}
 W^\Psi(t) = \int_0^t a(t-\tau)AW^\Psi(\tau)d\tau 
 + \int_0^t \Psi(\tau) dW(\tau), \quad P-a.s.
\end{equation}
\end{corollary}

\noindent{\bf Comment } The formula (\ref{deq16}) says that the 
convolution $W^\Psi$ is a strong solution to (\ref{SVEHe1}) 
if the operator $A$ is bounded.\vskip1.5mm

The below theorem is a consequence of the results obtained up to now.

\begin{theorem} \label{t3}
Suppose that (VA) and (PA) hold. 
Then a strong solution (if exists) is always a weak solution of (\ref{SVEHe1}). 
If, additionally, assumptions of Proposition \ref{pr4} are satisfied, 
a weak solution  is a mild solution to the 
Volterra equation (\ref{SVEHe1}). Conversely, under conditions of Proposition 
\ref{pr5}, a mild solution $X$ is also a weak solution to (\ref{SVEHe1}).
\end{theorem}

%
%
%

Now, we provide two estimates for stochastic convolution 
(\ref{deq5}). 

\begin{theorem} \label{t4}
If $\; \Psi\in \mathcal{N}^2(0,T;L_2^0)$ then the following estimate holds
 \begin{equation} \label{deq18}
  \sup_{t\leq T} \;\mathbb{E} (|W^\Psi (t)|_H) \leq C\,
  M_T\,  \mathbb{E} \left( \int_0^T|| \Psi(t) ||_{L_2^0}^2\, dt 
  \right)^{\frac{1}{2}},
 \end{equation}
 where $C$ is a constant and $M_T=\sup_{t\leq T} ||S(t)||. $
\end{theorem}

\noindent{\bf Comment} The estimate (\ref{deq18}) seems to be rather coarse.
It comes directly from the definition of stochastic integral. 
Since (\ref{deq18}) reduces to the Davis inequality for martingales if
$S(\cdot)=I$, the constant $C$ appears on the right hand side.
Unfortunately, we can 
not use more refined tools, for instance It\^o formula (see e.g.\ \cite{Tu84}
for Tubaro's estimate), because the process $W^\Psi$ is not enough regular.
\vskip1.5mm

The next result is a consequence of Theorem \ref{t4} 

\begin{theorem} \label{t5}
 Assume that $\Psi\in\mathcal{N}^2(0,T;L_2^0)$.  Then 
 $$\sup_{t\leq T}\, \mathbb{E}(|W^\Psi(t)|_H)
 \leq \widetilde{C}(T)\,|| \Psi ||_{\mathcal{N}^2(0,T;L_2^0)},$$
 where a constant $\widetilde{C}(T)$ depends on $T$.
\end{theorem}

\proof{From (\ref{deq5}) and property of stochastic integral, 
writing out the Hilbert-Schmidt norm, we obain 
\begin{eqnarray*}
 \mathbb{E}(|W^\Psi(t)|_H) & = & \mathbb{E}\left(|\int_0^t 
 S(t-\tau)\Psi(\tau)dW(\tau)|_H\right) \\  & \le &
 C\,\mathbb{E}\left(\int_0^t || S(t-\tau)\Psi(\tau)||_{L_2^0}^2\,d\tau
 \right)^{\frac{1}{2}}  \\
 \le  
 C\,\mathbb{E}\left(\int_0^t ||S(t-\tau)||^2 \, ||\Psi(\tau) ||_{L_2^0}^2
  \,d\tau \right)^{\frac{1}{2}}
   & \le & C\, M_T\, \mathbb{E}\left(\int_0^t ||\Psi(\tau) ||_{L_2^0}^2
  \,d\tau \right)^{\frac{1}{2}}\\ 
  \le C\,M_T \left( \mathbb{E} \int_0^t ||\Psi(\tau) ||_{L_2^0}^2 d\tau
  \right)^{\frac{1}{2}}
  & \le & \widetilde{C}(T)\, ||\Psi ||_{\mathcal{N}^2(0,T;L_2^0)}\;,
\end{eqnarray*}
where $M_T$ is as above and $\widetilde{C}(T)= C\,M_T$.
\qed}\\

\section{Existence of strong solution} \label{SVEHSsec:2}

In this section $W$ is a cylindrical Wiener process, that is, $\mathrm{Tr}
Q =+\infty$ and the spaces $U_0,~ L_2^0,~
\mathcal{N}^2(0,T;L_2^0)$ are the same like previously (see
definitions in chapter~\ref{PBch:2}). The results from this section originate
from \cite{KL07c} and are not yet published.

Let us recall the stochastic convolution introduced in (\ref{deq5})
$$ 
W^\Psi(t) := \int_0^t S(t-\tau)\Psi(\tau)\,dW(\tau),
$$ 
where $\Psi$ belongs to the space $\mathcal{N}^2(0,T;L_2^0)$. In
consequence, because resolvent operators $S(t),~t\geq 0$, are
bounded, then $S(t-\cdot)\Psi (\cdot)\in
\mathcal{N}^2(0,T;L_2^0)$, too.\medskip

In the sequel, by $A\Phi(t)$, $t \ge 0$, we will denote the composition
of the operators $\Phi(t)$ and $A$. \medskip

We will use the following
well-known result, where the operator $A$ is, as
previously, a closed linear operator with the dense domain $D(A)$ 
equipped with the graph norm $|\cdot |_{D(A)}$
and $\Phi(t),~t\in [0,T]$ is an $L_2(U_0,H)$-predictable process.
\begin{proposition} \label{pSW1} (see e.g.\ \cite[Proposition 4.15]{DZ92})\\
If $~\Phi(t)(U_0)\subset D(A), ~P-a.s.$ for all $t\in [0,T]$ and
$$ P\left( \int_0^T || \Phi(t)||_{L_2^0}^2\,dt <\infty \right) =1,~~
P\left( \int_0^T ||A \Phi(t)||_{L_2^0}^2\,dt <\infty \right) =1,
$$
then $\quad \displaystyle P\left( \int_0^T \Phi(t)\,dW(t) \in D(A)
\right) =1~~$ and
$$ A \int_0^T \Phi(t)\,dW(t) = \int_0^T A\Phi(t)\,dW(t), \quad P-a.s.
$$
\end{proposition}

Let us recall assumptions of approximation theorems (Theorems \ref{DVEt1}
and \ref{DVEt1a}) formulated for Hilbert space $H$:  
\begin{description}
\item[{\bf (AS1)}] The operator $A$ is the generator of a $C_0$-semigroup in $H$ and 
the kernel function $a(t)$ is completely positive.
\item[{\bf (AS2)}] $A$ generates an exponentialy bounded cosine family 
in $H$ and the function $a(t)$ is completely positive (or fulfills one 
of two other cases listed in the remark 1 on page~\pageref{comment}).
\end{description}
\index{(AS1)}
\index{(AS2)}

\begin{lemma} \label{pSW5}
Let assumptions (VA) be satisfied.
Suppose {\bf (AS1)} or {\bf (AS2)} hold. If $\Psi$ and
$A\Psi$ belong to $\mathcal{N}^2(0,T;L_2^0)$ and in addition
$\Psi(t)(U_0)\subset D(A),$ P-a.s., then the following
equality holds
\begin{equation} \label{eSW19}
W^\Psi(t) = \int_0^t a(t-\tau)A\, W^\Psi(\tau)\,d\tau
  + \int_0^t \Psi(\tau)\,dW(\tau) , \qquad P-a.s.
\end{equation}
\end{lemma}

\noindent {\bf Comment} Let us emphasize that assumptions concerning
the operators $\Psi (t)$, $t\ge 0$, particularly requirement that
$\Psi (t)(U_0)\subset D(A), P-a.s.$, are the same like in semigroup
case, see e.g.\ \cite[Proposition 6.4]{DZ92}.

\begin{proof}
%
%
%
%
Because formula (\ref{deq16}) holds for any bounded
operator, then it holds for the Yosida approximation $A_n$ of the
operator $A$, too, that is
$$ W_n^\Psi(t) =
\int_0^t a(t-\tau) A_n W_n^\Psi(\tau)d\tau +
\int_0^t\Psi(\tau)dW(\tau), $$ where
$$ W_n^\Psi(t) := \int_0^t S_n(t-\tau)\Psi(\tau)dW(\tau)$$
and
$$ A_n W_n^\Psi(t) =
A_n \int_0^t S_n(t-\tau)\Psi(\tau)dW(\tau).
$$
Recall that by assumption $\Psi\in \mathcal{N}^2(0,T;L_2^0)$.
Because the operators $S_n(t)$ are deterministic and bounded for
any $t\in [0,T]$, $n\in\mathbb{N}$, then the operators
$S_n(t-\cdot )\Psi(\cdot)$ belong to $\mathcal{N}^2(0,T;L_2^0)$,
too. In consequence, the difference
\begin{equation}\label{eq21}
 \Phi_n(t-\cdot ) := S_n(t-\cdot )\Psi(\cdot)
  - S(t-\cdot )\Psi(\cdot)
\end{equation}
belongs to $\mathcal{N}^2(0,T;L_2^0)$ for any $t\in [0,T]$ and
$n\in\mathbb{N}$. This means that
\begin{equation}\label{eq22}
 \mathbb{E}\left(\int_0^t ||\Phi_n(t-\tau)||_{L_2^0}^2d\tau \right)
 < +\infty
\end{equation}
for any $t\in [0,T]$.

Let us recall that the cylindrical Wiener process $W(t)$, $t\ge
0$, can be written in the form
\begin{equation}\label{eq23}
 W(t) =\sum_{j=1}^{+\infty} g_j\,\beta_j(t),
\end{equation}
where $\{g_j\}$ is an orthonormal basis of $U_0$ and $\beta_j(t)$
are independent real Wiener processes. From (\ref{eq23}) we have
\begin{equation}\label{eq24}
 \int_0^t \Phi_n(t-\tau)\,dW(\tau) = \sum_{j=1}^{+\infty}
 \int_0^t \Phi_n(t-\tau)\,g_j\,d\beta_j(\tau).
\end{equation}
From (\ref{eq22}), we obtain
\begin{equation}\label{eq25}
 \mathbb{E}\left[\int_0^t \left( \sum_{j=1}^{+\infty}
 |\Phi_n(t-\tau)\,g_j|_H^2 \right) d\tau \right]
 < +\infty
\end{equation}
for any $t\in [0,T]$. Next, from (\ref{eq24}), properties of
stochastic integral and (\ref{eq25}) we obtain for any
$t\in[0,T]$,
\begin{eqnarray*}
 \mathbb{E}\left| \int_0^t \Phi_n(t-\tau)\,dW(\tau) \right|_H^2
 &=& \mathbb{E}\left| \sum_{j=1}^{+\infty}\int_0^t
 \Phi_n(t-\tau)\,g_j\,d\beta_j(\tau) \right|_H^2 \le \\
  \mathbb{E}\left[ \sum_{j=1}^{+\infty} \int_0^t
 |\Phi_n(t-\tau)\,g_j|_H^2 d\tau \right]
 &\le & \mathbb{E}\left[ \sum_{j=1}^{+\infty} \int_0^T
  |\Phi_n(T-\tau)\,g_j|_H^2 d\tau \right] <+\infty.
\end{eqnarray*}

By Theorem \ref{DVEt1} or \ref{DVEt1a}, the convergence (\ref{DVEe10})
of resolvent families is uniform in $t$ on every compact subset of
$\mathbb{R}_+$, particularly on the interval $[0,T]$. Now, we use
(\ref{DVEe10}) in the Hilbert space $H$, so (\ref{DVEe10}) holds for
every $x\in H$. Then, for any fixed $j$,
\begin{equation}\label{eq26}
 \int_0^T |[S_n(T-\tau)-S(T-\tau)]\,
 \Psi(\tau)\,g_j|_H^2 d\tau  \longrightarrow 0
\end{equation}
for $n\to +\infty$. Summing up our considerations,
particularly using (\ref{eq25}) and (\ref{eq26}) we can write
\begin{eqnarray*}
\sup_{t\in [0,T]} \! & \mathbb{E} & \!\left| \int_0^t \!\!
\Phi_n(t\!-\!\tau)dW(\tau)
 \right|_H^2 \! \!   \equiv\! \sup_{t\in [0,T]}\! \mathbb{E}\left| \int_0^t
 [S_n(t\!-\!\tau)\!-\!S(t\!-\!\tau)] \Psi(\tau)dW(\tau)\right|_H^2 \\ 
 \!& \le &
 \mathbb{E}\left[ \sum_{j=1}^{+\infty} \! \int_0^T \!\!\!
 | [ S_n(T\!-\!\tau)\!-\!S(T\!-\!\tau)]\Psi(\tau)\,g_j |_H^2
 d\tau \right]  \longrightarrow 0 \quad \mbox{as} \quad n \to +\infty\,.
\end{eqnarray*}

Hence, by the Lebesgue dominated convergence theorem
\begin{equation}\label{eq27}
 \lim_{n\to +\infty} \sup_{t\in [0,T]} \mathbb{E} \left|
 W_n^\Psi(t)- W^\Psi(t)\right|_H^2 =0.
\end{equation}

By assumption, $\Psi(t)(U_0)\subset D(A), ~P-a.s.~$
Because $~S(t)(D(A))\subset D(A)$, then 
$S(t-\tau)\Psi(\tau)(U_0) \subset D(A),
~P-a.s.$, for any $\tau\in [0,t],~t\ge 0$.
Hence, by Proposition \ref{pSW1}, $P(W^\Psi (t)\in D(A))=1$. 

For any $n\in\mathbb{N}$, $t\ge 0$, we have
$$ |A_n W_n^\Psi (t) - A W^\Psi (t)|_H \le
    N_{n,1}(t) + N_{n,2}(t), $$
where
\begin{eqnarray*}
 N_{n,1}(t) & := & |A_n W_n^\Psi (t) - A_n W^\Psi (t)|_H , \\
 N_{n,2}(t) & := & |A_n W^\Psi (t) - A W^\Psi (t)|_H =
             |(A_n-A)W^\Psi (t)|_H \,.
\end{eqnarray*}
Then
\begin{eqnarray}\label{eq28}
 |A_n W_n^\Psi (t) - A W^\Psi (t)|_H^2 & \le &
 N_{n,1}^2 (t) + 2 N_{n,1}(t) N_{n,2}(t) + N_{n,2}^2(t) \nonumber \\
 & < & 3[N_{n,1}^2 (t)+N_{n,2}^2(t)].
\end{eqnarray}

Let us study the term $N_{n,1}(t)$. Note that the unbounded
operator $A$ generates a semigroup. Then we have for the Yosida
approximation the following properties:
\begin{equation}\label{eq30}
 A_nx=J_nAx \quad \mbox{for~any~} x\in D(A), \quad \sup_n ||J_n||
 < \infty
\end{equation}
where $A_nx=nAR(n,A)x=AJ_nx$ for any $x\in H$, with
$J_n:=nR(n,A).$ Moreover (see \cite[Chapter II, Lemma 3.4]{EN00}):
\begin{eqnarray}
 \lim_{n\to\infty}J_nx &=& x \qquad \mbox{for~any~} x\in H, \nonumber \\
 \lim_{n\to\infty} A_nx &=& Ax \qquad \mbox{for~any~} x\in D(A).
 \label{eq31}
\end{eqnarray}
By Proposition \ref{AScom}, $AS_n(t)x = S_n(t)Ax$ for every $n$ sufficiently
large and for all $x \in D(A)$. 
So, by Propositions \ref{AScom} and \ref{pSW1} and the 
closedness of $A$ we can write
\begin{eqnarray*}
 A_n W_n^\Psi(t) &\equiv& A_n
 \int_0^t S_n(t-\tau)\Psi(\tau)dW(\tau) \\
 &=& J_n \int_0^t AS_n(t-\tau)\Psi(\tau)dW(\tau)
 = J_n \left[\int_0^t S_n(t-\tau)A\Psi(\tau)dW(\tau)\right].
\end{eqnarray*}
Analogously,
$$ A_n W^\Psi(t) =
  J_n \left[\int_0^t S(t-\tau)A\Psi(\tau)dW(\tau)\right].
$$
By (\ref{eq30}) we have
\begin{eqnarray*}
N_{n,1}(t) &=& |J_n\int_0^t [S_n(t-\tau)-S(t-\tau)]
  A\Psi(\tau)dW(\tau)|_H \\
  &\le & |\int_0^t [S_n(t-\tau)-S(t-\tau)]
  A\Psi(\tau)dW(\tau)|_H \;.
\end{eqnarray*}
Since from assumptions $A\Psi \in \mathcal{N}^2(0,T;L_2^0)$, then the
term  apearing above, 
$[S_n(t-\tau)-S(t-\tau)]A\Psi(\tau)$ may be treated like
the difference $\Phi_n$ defined by (\ref{eq21}).

Hence, from (\ref{eq30}) and (\ref{eq27}), for the first term of
the right hand side of (\ref{eq28}) we have
$$ \lim_{n\to +\infty}\;\; \sup_{t\in [0,T]}
 \mathbb{E}(N_{n,1}^2 (t)) \to 0. $$
For the second term of (\ref{eq28}), that is $N_{n,2}^2(t)$, we can 
follow the same steps as above for proving (\ref{eq27}).
\begin{eqnarray*}
N_{n,2}(t) & = & |A_n W^\Psi (t) - A W^\Psi (t)|_H \\
 &\equiv & 
  \left| A_n \int_0^t S(t-\tau)\Psi (\tau)dW(\tau) 
       - A   \int_0^t S(t-\tau)\Psi (\tau)dW(\tau) \right|_H =  \\
 & = & \left|\int_0^t [A_n-A] S(t-\tau)\Psi (\tau)dW(\tau)\right|_H\;.
\end{eqnarray*}
 
From assumptions, $\Psi, A\Psi \in \mathcal{N}^2(0,T;L_2^0)$. 
Because $A_n, S(t), ~t\ge 0$ are bounded, then 
$A_n S(t-\cdot)\Psi(\cdot)
\in \mathcal{N}^2(0,T;L_2^0)$, too. 

Analogously, 
$AS(t-\cdot)\Psi(\cdot)=S(t-\cdot)A\Psi(\cdot)
\in \mathcal{N}^2(0,T;L_2^0)$.

Let us note that the set of all Hilbert-Schmidt operators acting from 
one separable Hilbert space into another one, equipped with the
operator norm defined on page \pageref{H-Sn} is a separable 
Hilbert space. Particularly, sum of two Hilbert-Schmidt operators
is a Hilbert-Schmidt operator, see e.g.\ \cite{Ba81}. Therefore, 
we can deduce that the operator  $(A_n-A)$ $S(t-\cdot)\Psi(\cdot)\in 
\mathcal{N}^2(0,T;L_2^0)$, ~for any $~t\in [0,T]$. Hence, the term 
$[A_n-A]S(t-\tau)\Psi(\tau)$ may be treated like the difference 
$\Phi_n$ defined by (\ref{eq21}). So, we obtain 
\begin{eqnarray*}
 \mathbb{E} \left( N_{n,2}^2(t)\right) &=& 
 \mathbb{E} \left( \int_0^t \left[ \sum_{j=1}^{+\infty} \left| 
 [A_n-A] S(t-\tau )\Psi (\tau)\,g_j \right|_H^2  \right]
 d\tau \right) \\
 &\le & \mathbb{E} \left(\sum_{j=1}^{+\infty} \int_0^T
 \left| [A_n-A] S(t-\tau)\Psi(\tau)\,g_j\right|_H^2 d\tau\right)
 < +\infty,
\end{eqnarray*}
for any $t\in [0,T]$.

By the convergence (\ref{eq31}), for any fixed $j$, 
$$ \int_0^T |[A_n-A]S(t-\tau) \Psi(\tau)\,g_j |_H^2 d\tau \longrightarrow 0
\quad \mbox{for} \quad n\to +\infty.$$

Summing up our considerations, we have 
$$ \lim_{n\to +\infty} \;\;\sup_{t\in [0,T]} \mathbb{E}
(N_{n,2}^2 (t)) \to 0\;.$$

So, we can deduce that
$$ \lim_{n\to +\infty} \;\;\sup_{t\in [0,T]} \mathbb{E}
  |A_n W_n^\Psi(t)-AW^\Psi(t)|_H^2=0 , $$
and then (\ref{eSW19}) holds. \hfill 
\qed
\end{proof}

These considerations give rise to the following result.

\begin{theorem} \label{coSW4}
Suppose that assumptions of Lemma \ref{pSW5} hold. Then the
equation (\ref{SVEHe1}) has a strong solution. Precisely, the
convolution $W^\Psi$ defined by (\ref{deq5}) is the strong
solution to~(\ref{SVEHe1}).
\end{theorem}
\begin{proof} 
Since Proposition \ref{pr3} and Lemma \ref{pSW5} hold, 
we have to show only 
the condition (\ref{eSW3.1}). Let us note that by Proposition \ref{pr3a},
the convolution 
$W^\Psi (t)$ has integrable trajectories.  Because the closed unbounded 
linear operator $A$  becomes bounded on ($D(A),|\cdot|_{D(A)}$), see e.g.\
\cite{We80}, 
we obtain that $AW^\Psi (\cdot )\in
L^1([0,T];H)$, P-a.s. Next, properties of convolution provide
integrability of the function $a(T-\cdot)AW^\Psi (\cdot)$, 
what finishes the proof. \qed
\end{proof}

\section{Fractional Volterra equations} \label{SVEHSsec:3}

Assume, as previously, that 
$H$ is a separable Hilbert space with a norm $|\cdot|_H$ and
$A$ is a closed linear operator with dense domain
$D(A)\subset H$ equipped with the graph norm $|\cdot|_{D(A)}$. The
purpose of this section is to study the existence of strong
solutions for a class of  stochastic Volterra equations of the
form
\begin{equation}{\label{eq4f}}
X(t) = X(0) + \int_0^t a_\alpha(t-\tau) AX(\tau)d\tau + \int_0^t
\Psi(\tau)\, dW(\tau), \quad t\geq 0,
\end{equation}
\index{$a_\alpha(t) :=\displaystyle \frac{t^{\alpha-1}}{\Gamma(\alpha)}$}
\index{$\Gamma(\alpha)$}
where $a_\alpha(t) :=\displaystyle \frac{t^{\alpha-1}}{\Gamma(\alpha)},~
\alpha>0$, $\Gamma(\alpha)$ is the gamma function 
and $W,\; \Psi$ are appropriate stochastic processes. 
There are several situations that can be
modeled by stochastic Volterra equations, see e.g. \cite[Section
3.4 ]{HO96} and references therein.  A similar equation
was studied in \cite{BT03}, too. 
Here we are interested in the study of strong
solutions when equation (\ref{eq4f}) is driven by a 
cylindrical Wiener process $W$. 
We give sufficient conditions for stochastic convolution to be a strong
solution to (\ref{eq4f}).

The equation (\ref{eq4f}) is a stochastic version of the deterministic 
Volterra equation
\begin{equation}{\label{eq3af}}
u(t) = \int_0^t a_\alpha(t-\tau) Au(\tau)d\tau + f(t) \,,
\end{equation}
where $f$ is an $H$-valued function. \pagebreak

In the case when $a_\alpha(t)$ is a completely positive function,
sufficient conditions for existence of strong solutions for
(\ref{eq4f}) may be obtained like in section \ref{SVEHSsec:2}, 
that is, using a
method which involves the use of a resolvent family associated to
the  deterministic version of equation (\ref{eq4f}).

However, there are two kinds of problems that arise when we study
(\ref{eq4f}). On the one hand, the kernels
$\frac{t^{\alpha-1}}{\Gamma(\alpha)} $ are $\alpha$-regular and
$\frac{\alpha \pi}{2}$-sectorial but not completely positive
functions for $\alpha >1$, so e.g. the results in \cite{KL07c}
can not be used directly for $\alpha >1$. On the other hand, for
$\alpha\in (0,1)$, we have a singularity of the kernel in $t=0.$
This fact strongly suggests the use of
 $\alpha$-times resolvent families associated to
equation (\ref{eq3af}). These new tools appeared in
\cite{Ba01} as well as their relationship with fractional
derivatives. For convenience of the reader, we
  provide below the main results on $\alpha$-times resolvent
families to be used in this paper.

Our second main ingredient to obtain strong solutions of
(\ref{eq4f}) relies on approximation of $\alpha$-times resolvent
families. This kind of result was very recently formulated by Li
and Zheng \cite{LZ04}. It enables us to prove a key result on
convergence of $\alpha$-times resolvent families (see Theorem \ref{th2f} 
below).  Then we can follow the methods employed in \cite{KL07c} to 
obtain existence of solutions - particularly strong - for the 
stochastic equation (\ref{eq4f}).

\subsection{Convergence of $\alpha$-times resolvent families}
\label{Sconv}

In this section we formulate the main deterministic results on
convergence of resolvents. 

\index{$S_{\alpha}(t)$}
By $S_{\alpha}(t),~t\geq 0$, we denote the family of
$\alpha$-times resolvent families corresponding to the Volterra
equation (\ref{eq3af}), if it exists, and defined analogously
like resolvent family, see Definition \ref{DVEd1}.

\begin{definition}\label{SVEHSd1} (see \cite{Ba01})\\
A family $(S_{\alpha}(t))_{t\geq 0}$ of bounded linear operators
in a Banach space $B$ is called {\tt $\alpha$-times resolvent
family}\index{resolvent family!$\alpha$-times} for (\ref{eq3af}) 
if the following conditions are satisfied:
\begin{enumerate}
\item $S_{\alpha}(t)$ is strongly continuous on $\mathbb{R}_+$ and
$S_{\alpha}(0)=I$;
\item $S_{\alpha}(t)$ commutes with the operator $A$, that is,
$S_{\alpha}(t)(D(A))\subset D(A)$ and
$AS_{\alpha}(t)x=S_{\alpha}(t)Ax$ for all $x\in D(A)$ and $t\geq
0$;
\item the following {\tt resolvent equation} holds
\begin{equation} \label{eq4af}
S_{\alpha}(t)x = x + \int_0^t a_{\alpha}(t-\tau)
AS_{\alpha}(\tau)x d\tau
\end{equation}
for all $x\in D(A),~t\geq 0$.
\end{enumerate}
\end{definition}

Necessary and sufficient conditions for existence of the
$\alpha$-times resolvent family have been studied in \cite{Ba01}.
Observe that the $\alpha$-times resolvent family corresponds to a
$C_0$-semigroup in case $\alpha =1$ and a cosine family in case
$\alpha=2.$ In consequence, when $1 < \alpha < 2$ such resolvent
families interpolate $C_0$-semigroups and cosine functions. In
particular, for $A = \Delta$, the integrodifferential equation
corresponding to such resolvent family interpolates the heat
equation and the wave equation, see \cite{Fu89} or \cite{SW89}.

\begin{definition}{\label{SVEHSd2}}
An $\alpha$-times resolvent family $(S_{\alpha}(t))_{t\geq 0}$ is
called {\tt exponen\-tially}  {\tt bounded}
\index{resolvent family!exponentially bounded} 
if there are constants $M\geq 1$ and $\omega \geq 0$ such that
\begin{equation}{\label{eq5f}}
\| S_{\alpha}(t) \| \leq Me^{\omega t}, \quad t\geq 0.
\end{equation}
\end{definition}

If there is the $\alpha$-times resolvent family
$(S_{\alpha}(t))_{t\geq 0}$ for $A$ and satisfying (\ref{eq5f}), we
write $A \in \mathcal{C}^{\alpha}(M,\omega). $ Also, set $
\mathcal{C}^{\alpha}(\omega) :=\cup_{M\geq 1}
\mathcal{C}^{\alpha}(M,\omega)$ and \linebreak
$ \mathcal{C}^{\alpha}
:=\cup_{\omega \geq 0} \mathcal{C}^{\alpha}(\omega).$
\vskip2mm

\noindent{\bf Remark} 
It was proved by Bazhlekova \cite[Theorem 2.6]{Ba01} that if $ A\in
\mathcal{C}^{\alpha}$ for some $\alpha > 2,$ then $A$ is bounded.\\

The following subordination principle is very important in the
theory of $\alpha$-times resolvent families (see \cite[Theorem
3.1]{Ba01}).

\begin{theorem}\label{th1f}
Let $0 < \alpha < \beta \leq 2, \gamma = \alpha /\beta, \omega\geq
0.$ If $A \in \mathcal{C}^{\beta}(\omega)$ then $A \in
\mathcal{C}^{\alpha}(\omega^{1/\gamma})$ and the following
representation holds
\index{$\varphi_{t,\gamma}(s)$}
\begin{equation}\label{eq6f}
S_{\alpha}(t)x = \int_0^{\infty} \varphi_{t,\gamma}(s)
S_{\beta}(s)xds, \quad t>0,
\end{equation}
where $\varphi_{t,\gamma}(s) := t^{-\gamma}
\Phi_{\gamma}(st^{-\gamma})$ and $\Phi_{\gamma}(z)$ is the Wright
function defined as
\index{function!Wright}
\begin{equation}\label{eq7f}
\Phi_{\gamma}(z):= \sum_{n=0}^{\infty} \frac{(-z)^n}
{n!\,\Gamma(-\gamma n + 1 - \gamma)}, \quad 0 < \gamma < 1.
\end{equation}
\index{$\Phi_{\gamma}(z)$}
\end{theorem}
\vskip1mm

\noindent{\bf Remarks}  1.
\label{rem2f1}
 We recall that the Laplace transform of the Wright function
corresponds to $E_{\gamma}(-z)$ where $E_{\gamma}$ denotes the
Mittag-Leffler\index{function!Mittag-Leffler} function. 
In particular, $\Phi_{\gamma}(z)$ is a
probability density function.

\label{rem2f2}
 2. Also we recall from \cite[(2.9)]{Ba01} that the continuity in
$t\geq 0$ of the Mittag-Leffler function together with the
asymptotic behavior of it, imply that for $\omega \geq 0$ there
exists a constant $C>0$ such that
\begin{equation}\label{eq8f}
E_{\alpha}(\omega t^{\alpha}) \leq C e^{{\omega^{1/\alpha}} t},
\quad t \geq 0,\,\, \alpha \in (0,2).
\end{equation}

As we have already written, in this paper the results concerning
convergence of $\alpha$-times resolvent families in a Banach space
$B$ will play the key role. Using a very recent result due to Li
and Zheng \cite{LZ04} we are able to prove the following theorem.

\begin{theorem} \label{th2f}
Let $A$ be the generator of a $C_0$-semigroup $(T(t))_{t\geq 0}$
in a Banach space $B$ such that
\begin{equation}\label{eq9f}
\|T(t) \| \leq Me^{\omega t}, \quad t \geq 0.
\end{equation}
Then, for each $ 0 < \alpha < 1$ we have $A \in
\mathcal{C}^{\alpha}(M,\omega^{1/\alpha}). $ Moreover, there exist
bounded operators $A_n$ and $\alpha$-times resolvent families
$S_{\alpha,n}(t)$ for $A_n$ satisfying $ ||S_{\alpha,n}(t) || \leq
MCe^{(2\omega)^{1/\alpha} t},$ for all $t\geq 0,~n\in \mathbb{N},$
and
\begin{equation} \label{eq10f}
S_{\alpha,n}(t)x \to S_{\alpha}(t)x \quad \mbox{as} \quad n\to
+\infty
\end{equation}
for all $x \in B,\; t\geq 0.$ Moreover, the convergence is uniform
in $t$ on every compact subset of $ \mathbb{R}_+$.
\end{theorem}

\begin{proof} Since $A$ is the generator of a $C_0$ semigroup
satisfying (\ref{eq9f}), we have $A\in C^1(\omega)$. Hence, the
first assertion follows directly from Theorem \ref{th1f}, that is,
for each $ 0<\alpha < 1$ there is an $\alpha$-times resolvent
family $(S_{\alpha}(t))_{t\geq 0}$ for $A$ given by
\begin{equation}\label{eq11f}
S_{\alpha}(t)x = \int_0^{\infty} \varphi_{t,\alpha}(s) T(s)xds,
\quad t>0.
\end{equation}

Since $A$ generates a $C_0$-semigroup, the resolvent set $\rho(A)
$ of $A$ contains the ray $ [w,\infty)$ and
$$
||R(\lambda,A)^k || \leq \frac{M}{(\lambda -w)^k } \qquad
\mbox{for } \lambda > w, \qquad k\in \mathbb{N}.
$$

Define
\begin{equation} \label{eq12f}
A_n := n AR(n,A) = n^2 R(n,A) - nI, \qquad n> w,
\end{equation}
the {\it Yosida approximation} of $A$.

Then
\begin{eqnarray*}
||e^{t A_n} || &=& e^{-nt} || e^{n^2 R(n,A)t} || \leq
e^{-nt} \sum_{k=0}^{\infty} \frac{n^{2k} t^k}{k!} ||R(n,A)^k|| \\
&\leq& M e^{(-n + \frac{n^2}{n-w})t} = M e^{ \frac{nwt}{ n-w}}.
\end{eqnarray*}
Hence, for $n > 2w$ we obtain
\begin{equation}{\label{eq13f}}
|| e^{A_n t} || \leq M e^{2wt}.
\end{equation}
Next, since each $A_n$ is bounded, it follows also from Theorem
\ref{th1f} that for each $0< \alpha < 1$ there exists an
$\alpha$-times resolvent family $ (S_{\alpha,n}(t))_{t\geq 0}$ for
$A_n$ given as

\begin{equation}{\label{eq14f}}
S_{\alpha,n}(t) = \int_0^{\infty} \varphi_{t,\alpha}(s)
e^{sA_n}ds, \quad t>0.
\end{equation}

By  (\ref{eq13f}) and remark 1, page~\pageref{rem2f1}, it follows that
\begin{eqnarray*}
\| S_{\alpha,n}(t) \| &\leq & \int_0^{\infty}
\varphi_{t,\alpha}(s) \| e^{s A_n} \| ds \leq M
\int_0^{\infty} \varphi_{t,\alpha}(s) e^{2\omega s} ds \\
 & = & M
\int_0^{\infty} \Phi_{\alpha}(\tau) e^{2\omega t^{\alpha} \tau }
d\tau= M E_{\alpha}(2 \omega t^{\alpha}), \quad t \geq 0.
\end{eqnarray*}

This together with remark 2, page~\pageref{rem2f2}, gives
\begin{equation}
\| S_{\alpha,n}(t) \| \leq MCe^{(2\omega)^{1/\alpha}t}, \quad t
\geq 0.
\end{equation}

Now, we recall the fact that $ R(\lambda,A_n)x \to R(\lambda,A)x $
as $ n\to \infty$ for all $\lambda $ sufficiently large (see e.g.
\cite[Lemma~7.3]{Pa83}), so we can conclude from \cite[Theorem
4.2]{LZ04} that
\begin{equation}
S_{\alpha,n}(t)x \to S_{\alpha}(t)x \quad \mbox{as} \quad n\to
+\infty
\end{equation}
for all $x \in B,$ uniformly for $t$ on every compact subset of
$\mathbb{R}_+$. \qed \end{proof}

An analogous result can be proved in the case when $A$ is the
generator of a strongly continuous cosine family.

\begin{theorem} \label{th3af}
Let $A$ be the generator of a $C_0$-cosine family $(T(t))_{t\geq
0}$ in a Banach space $B$. Then, for each $0<\alpha<2$ we have $A
\in \mathcal{C}^{\alpha}(M,\omega^{2/\alpha}). $ Moreover, there
exist bounded operators $A_n$ and $\alpha$-times resolvent
families $S_{\alpha,n}(t)$ for $A_n$ satisfying $
||S_{\alpha,n}(t) || \leq MCe^{(2\omega)^{1/\alpha} t},$ for all
$t\geq 0,~n\in \mathbb{N},$ and
$$ 
S_{\alpha,n}(t)x \to S_{\alpha}(t)x \quad \mbox{as} \quad n\to
+\infty
$$ 
for all $x \in B,\; t\geq 0.$ Moreover, the convergence is uniform
in $t$ on every compact subset of $ \mathbb{R}_+$.
\end{theorem}

Let us note that formulae (\ref{eq11f}) and (\ref{eq14f}) still hold
when $A$ is the $C_0$-cosine family and $0<\alpha <2$. \vskip1.5mm

In the following, we denote by $\Sigma_{\theta}(\omega)$ the open
sector with vertex $\omega \in \mathbb{R}$ and opening angle
$2\theta$ in the complex plane which is symmetric with respect to
the real positive axis, i.e.
$$ \Sigma_{\theta}(\omega) := \{ \lambda \in \mathbb{C}:
|arg(\lambda -\omega)| < \theta \}.$$

We recall from \cite[Definition 2.13]{Ba01} that an $\alpha$-times
resolvent family $S_{\alpha}(t)$ is called {\tt analytic} if
\index{resolvent family!analytic}
$S_{\alpha}(t)$ admits an analytic extension to a sector
$\Sigma_{\theta_0}$ for some $\theta_0 \in (0, \pi/2].$ An
$\alpha$-times analytic resolvent family is said to be of {\tt
analyticity type} \index{analyticity type}
$(\theta_0, \omega_0)$ if for each $\theta <
\theta_0$ and $\omega > \omega_0$ there is $M= M(\theta, \omega)$
such that
$$ \| S_{\alpha}(t)\| \leq Me^{\omega Re t}, \quad t \in
\Sigma_{\theta}.$$ The set of all operators $A \in
\mathcal{C}^{\alpha}$ generating $\alpha$-times analytic resolvent
families $S_{\alpha}(t)$ of type $(\theta_0, \omega_0)$ is denoted
by $ \mathcal{A}^{\alpha}(\theta_0, \omega_0).$ In addition,
denote \linebreak
$ \mathcal{A}^{\alpha}(\theta_0):= \bigcup \{
\mathcal{A}^{\alpha}(\theta_0, \omega_0); \omega_0 \in
\mathbb{R}_+ \}, \quad \mathcal{A}^{\alpha} := \bigcup \{
\mathcal{A}^{\alpha}(\theta_0); \theta_0 \in (0, \pi/2] \}.$ For
$\alpha =1$ we obtain the set of all generators of analytic
semigroups.

\vskip2mm
\label{rem3f}
\noindent{\bf Remark} 
We note that the spatial regularity condition
$\mathcal{R}(S_\alpha(t)) \subset D(A)$ for all $t>0$
is satisfied by $\alpha$-times resolvent families whose generator
$A$ belongs to the set $\mathcal{A}^{\alpha}(\theta_0, \omega_0)$
where $ 0<\alpha <2$ (see \cite[ Proposition 2.15]{Ba01}). In
particular, setting $\omega_0=0$ we have that $ A\in
\mathcal{A}^{\alpha}(\theta_0, 0)$ if and only if $-A$ is a
positive operator with spectral angle less or equal to $ \pi -
\alpha(\pi/2+\theta).$ Note that such condition is also equivalent
to the following
\begin{equation}{\label{eq29f}}
 \Sigma_{\alpha(\pi/2 +\theta)} \subset \rho(A) \mbox{ and } \|
 \lambda (\lambda I - A)^{-1} \| \leq M, \quad \lambda \in
 \Sigma_{\alpha(\pi/2 +\theta)}.
\end{equation}

The above considerations give us the following remarkable
corollary.

\begin{corollary} \label{cor5f}
Suppose $A$ generates an analytic semigroup of angle $\pi/2$ and
$\alpha \in (0,1)$. Then $A$ generates an $\alpha$-times analytic
resolvent family.
\end{corollary}
\begin{proof} Since $A$ generates an analytic semigroup of angle
$\pi/2$ we have
$$ \|
 \lambda (\lambda I - A)^{-1} \| \leq M, \quad \lambda \in
 \Sigma_{\pi -\epsilon}.$$

Then the condition (\ref{eq29f}) (see also \cite[Corollary
2.16]{Ba01}) implies $A \in \mathcal{A}^{\alpha}(\min
\{\frac{2-\alpha}{2\alpha}\pi, \frac{1}{2}\pi \},0)$, $\alpha \in
(0,2), $ that is $A$ generates an $\alpha$-times analytic
resolvent family.
\qed \end{proof}

In the sequel we will use the following assumptions concerning
Volterra equations:
\begin{description}
\item[{\bf (A1)}] $A$ is the generator of $C_0$-semigroup in $H$
 and $\alpha\in (0,1)$; ~~or
\item[{\bf (A2)}] $A$ is the generator of a strongly continuous
 cosine family in $H$ and $\alpha\in (0,2)$.
\end{description}

Observe that (A2) implies (A1) but not vice versa.

\subsection{Strong solution}\label{StrSol}

As previously $H$ and $U$ are separable Hilbert spaces and $W$ 
is a cylindrical Wiener process defined on a
stochastic basis $(\Omega,\mathcal{F},(\mathcal{F})_{t\geq 0},P)$,
with the positive symmetric covariance operator $Q\in L(U)$, 
$\mathrm{Tr}\,Q=+\infty$.
The spaces $U_0,~L_2^0=L_2(U_0,H)$ and $\mathcal{N}(0,T;L_2^0)$ 
are the same like in previous sections.

For the reader's convenience we formulate definitions of solutions to the
equation (\ref{eq4f}). 
We define solutions to the equation (\ref{eq4f}) analogously like in
section \ref{SVEHSsec:1}.

\begin{definition} \label{dSW4f}
Assume that (PA) hold. An $H$-valued predictable process
$X(t),~t\in [0,T]$, is said to be a 
~{\tt strong solution}\index{solution!strong}~ to
(\ref{eq4f}), if $X$ has a version such that $P(X(t)\in D(A))=1$,
for almost all $t\in [0,T]$; for any $t\in [0,T]$
\begin{equation} \label{eSW3.1f}
\int_0^t |a_\alpha(t-\tau)AX(\tau)|_H \,d\tau<+\infty,\quad
P-a.s., \quad \alpha >0,
\end{equation}
and for any $t\in [0,T]$ the equation (\ref{eq4f}) holds $P-a.s$.
\end{definition}

\begin{definition} \label{dSW5f}
Let (PA) hold. An $H$-valued predictable process $X(t),~t\in
[0,T]$, is said to be a {\tt weak solution}\index{solution!weak} 
to (\ref{eq4f}), if
$P(\int_0^t|a_\alpha(t-\tau)X(\tau)|_H d\tau<+\infty)=1,~ \alpha>0$, 
and if for all $\xi\in D(A^*)$ and all $t\in [0,T]$ the
following equation holds
\begin{eqnarray*}
\langle X(t),\xi\rangle_H  = \langle X(0),\xi\rangle_H &+& \langle 
\int_0^t  a_\alpha(t-\tau)X(\tau)\,d\tau, A^*\xi\rangle_H  \\
 &+&  \langle \
\int_0^t  \Psi(\tau)dW(\tau),\xi\rangle_H, \quad P-a.s.
\end{eqnarray*}
\end{definition}

\begin{definition} \label{dSW6f}
Assume that $X(0)$ is $\mathcal{F}_0$-measurable random variable.
An $H$-valued predictable process
$X(t),~t\in [0,T]$, is said to be a 
{\tt mild solution}\index{solution!mild} to the
stochastic Volterra equation (\ref{eq4f}), if~ $ \mathbb{E}(
\int_0^t ||S_\alpha(t-\tau) \Psi(\tau)||_{L_2^0}^2 \,d\tau
)<+\infty, ~ \alpha >0$, for $t\leq T$ and, for arbitrary $t\in
[0,T]$,
\begin{equation}\label{eSW9f}
X(t) = S_\alpha(t)X(0) + \int_0^t
S_\alpha(t-\tau)\Psi(\tau)\,dW(\tau), \quad P-a.s.
\end{equation}
where $S_\alpha(t)$ is the $\alpha$-times resolvent family.
\end{definition}

We define the stochastic convolution
\index{stochastic convolution}
\index{$W_\alpha^\Psi(t)$}
\begin{equation} \label{eSW18af}
W_\alpha^\Psi(t) := \int_0^t S_\alpha(t-\tau)\Psi(\tau)\,dW(\tau),
\end{equation}
where $\Psi\in\mathcal{N}^2(0,T;L_2^0)$. Because $\alpha$-times
resolvent families $S_\alpha(t),~t\geq 0$, are bounded, then
$S_\alpha(t-\cdot)\Psi (\cdot)\in \mathcal{N}^2(0,T;L_2^0)$, too.

\vskip1mm
Analogously like in section \ref{SVEHSsec:1}, we can formulate the 
following results.

\begin{proposition} \label{pro2}
 Assume that $S_\alpha (t), t\ge 0$, are the resolvent operators to
 (\ref{eq3af}). Then, for any process $\Psi \in \mathcal{N}^2(0,T;L_2^0)$,
 the convolution $W_\alpha^\Psi (t), ~t\ge 0, ~\alpha >0$, given by
 (\ref{deq5}) has a predictable version. Additionally, the process
 $W_\alpha^\Psi (t), ~t\ge 0, ~\alpha >0$, has square integrable trajectories.
\end{proposition}

Under some conditions every mild solution to (\ref{eq4f}) is a weak
solution to (\ref{eq4f}).

\begin{proposition} \label{pSW4f} 
If $\Psi\in\mathcal{N}^2(0,T;L_2^0)$, then the
stochastic convolution $W_\alpha^\Psi$ fulfills the equation
\begin{equation}\label{epSW4f}
\langle W_\alpha^\Psi(t),\xi\rangle_H =
\int_0^t \langle a_\alpha(t-\tau)W_\alpha^\Psi(\tau),
A^*\xi\rangle_H + \int_0^t \langle
\xi,\Psi(\tau)dW(\tau)\rangle_H, 
\end{equation}
$\alpha >0$ for any $t\in [0,T]$ and $\xi\in D(A^*)$.
\end{proposition}

Immediately from the equation (\ref{epSW4f}) we deduce the following result.

\begin{corollary} \label{cor2a}
 If $A$ is a bounded operator and $\Psi\in\mathcal{N}^2(0,T;L_2^0)$, then the
 following equality holds
 \begin{equation} \label{eq21a}
 W_\alpha^\Psi(t) =
 \int_0^t a_\alpha(t-\tau) A W_\alpha^\Psi(\tau)d\tau +
 \int_0^t\Psi(\tau)dW(\tau),
\end {equation}
for $t\in [0,T]$, $\alpha >0$.
\end{corollary}

\noindent{\bf Remark} The formula (\ref{eq21a}) says that the convolution
$W_\alpha^\Psi(t), ~t\ge 0, ~\alpha >0$, is a strong solution to
(\ref{eq4f}) if the operator $A$ is bounded.

\vskip1.5mm
We can formulate following result which plays a key role in this
subsection.

\begin{lemma} \label{pSW5f}
Let assumptions (VA) be satisfied. 
Suppose that {\bf (A1)} or {\bf (A2)} holds. If $\Psi$ and $A\Psi$
belong to $\mathcal{N}^2(0,T;L_2^0)$ and in addition
$\Psi(t)(U_0)\subset D(A),$ P-a.s., then the following
equality holds
\begin{equation} \label{eSW19f}
W_\alpha^\Psi(t) = \int_0^t a_\alpha(t-\tau)A\,
W_\alpha^\Psi(\tau)\,d\tau + \int_0^t \Psi(\tau)\,dW(\tau), \quad P-a.s.
\end{equation}
\end{lemma}
\noindent {\bf Remark } Let us emphasize that in {\bf (A1)}, $\alpha\in (0,1)$
and in {\bf (A2)}, $\alpha\in (0,2)$.\\

Although the proof is analogous to that given in section \ref{SVEHSsec:1},
we formulate it for the reader's convenience.
\begin{proof} 
Because formula (\ref{eq21a})
holds for any bounded operator, then it holds for the Yosida
approximation $A_n$ of the operator $A$, too, that is,
$$ W_{\alpha,n}^\Psi(t) =
\int_0^t a_\alpha(t-\tau) A_n W_{\alpha,n}^\Psi(\tau)d\tau +
\int_0^t\Psi(\tau)dW(\tau), $$ where
$$ W_{\alpha,n}^\Psi(t) := \int_0^t S_{\alpha,n}(t-\tau)\Psi(\tau)dW(\tau)$$
and
$$ A_n W_{\alpha,n}^\Psi(t) =
A_n \int_0^t S_{\alpha,n}(t-\tau)\Psi(\tau)dW(\tau).
$$
By assumption $\Psi\in \mathcal{N}^2(0,T;L_2^0)$. Because the
operators $S_{\alpha,n}(t)$ are deterministic and bounded for any
$t\in [0,T], ~\alpha >0, ~n\in\mathbb{N}$, then the operators
$S_{\alpha,n}(t-\cdot )\Psi(\cdot)$ belong to
$\mathcal{N}^2(0,T;L_2^0)$, too. In consequence, the difference
\begin{equation}\label{eq21b}
 \Phi_{\alpha,n}(t-\cdot ) := S_{\alpha,n}(t-\cdot )\Psi(\cdot)
  - S_\alpha(t-\cdot )\Psi(\cdot) 
\end{equation}
belongs to $\mathcal{N}^2(0,T;L_2^0)$ for any $t\in [0,T], ~\alpha >0$ and
$n\in\mathbb{N}$. This means that
\begin{equation}\label{eq22f}
 \mathbb{E}\left(\int_0^t ||\Phi_{\alpha,n}(t-\tau)||_{L_2^0}^2d\tau \right)
 < +\infty
\end{equation}
for any $t\in [0,T]$.

The cylindrical Wiener process $W(t)$, $t\ge 0$,  
can be expanded in the series 
\begin{equation}\label{eq23f}
 W(t) =\sum_{j=1}^{+\infty} g_j\,\beta_j(t),
\end{equation}
where $\{g_j\}$ is an orthonormal basis of $U_0$ and $\beta_j(t)$ are 
independent
real Wiener processes. From (\ref{eq23f}) we have
\begin{equation}\label{eq24f}
 \int_0^t \Phi_{\alpha,n}(t-\tau)\,dW(\tau) = \sum_{j=1}^{+\infty} 
 \int_0^t \Phi_{\alpha,n}(t-\tau)\,g_j\,d\beta_j(\tau).
\end{equation}
In consequence, from (\ref{eq22f})
\begin{equation}\label{eq25f}
 \mathbb{E}\left[\int_0^t \left( \sum_{j=1}^{+\infty}
 |\Phi_{\alpha,n}(t-\tau)\,g_j|_H^2 \right) d\tau \right]
 < +\infty
\end{equation}
for any $t\in [0,T]$. Next, from (\ref{eq24f}), properties of stochastic
integral and (\ref{eq25f}) we obtain for any $t\in[0,T]$,
\begin{eqnarray*}
 \mathbb{E}\left| \int_0^t \Phi_{\alpha,n}(t-\tau)\,dW(\tau) \right|_H^2
 &=& \mathbb{E}\left| \sum_{j=1}^{+\infty}\int_0^t 
 \Phi_{\alpha,n}(t-\tau)\,g_j\,d\beta_j(\tau) \right|_H^2  \le \\
  \mathbb{E}\left[ \sum_{j=1}^{+\infty} \int_0^t 
 |\Phi_{\alpha,n}(t-\tau)\,g_j|_H^2 d\tau \right] 
 &\le & \mathbb{E}\left[ \sum_{j=1}^{+\infty} \int_0^T 
  |\Phi_{\alpha,n}(T-\tau)\,g_j|_H^2 d\tau \right] <+\infty.
\end{eqnarray*}

By Theorem \ref{th2f} or Theorem \ref{th3af}, the convergence of
$\alpha$-times resolvent families is uniform in $t$ on the interval
$[0,T]$. So, for any fixed $\alpha$ and $j$,
\begin{equation}\label{eq26f}
 \int_0^T |[S_{\alpha,n}(T-\tau)-S_\alpha(T-\tau)]\,
 \Psi(\tau)\,g_j|_H^2 d\tau  \longrightarrow 0 \quad \mbox{for} \quad
 n\to +\infty\,.
\end{equation}
Then, using (\ref{eq25f}) and (\ref{eq26f}) we can write
\begin{eqnarray*}
\sup_{t\in [0,T]} & \mathbb{E} & \left|  \int_0^t 
\Phi_{\alpha,n}(t-\tau)dW(\tau) \right|_H^2  \\ & \equiv &
 \sup_{t\in [0,T]}\, \mathbb{E}\left| \int_0^t
 [S_{\alpha,n}(t-\tau)-S_\alpha(t-\tau)] \Psi(\tau)dW(\tau)\right|_H^2 \\
 & \le &
 \mathbb{E}\left[ \sum_{j=1}^{+\infty} \int_0^T 
 | [ S_{\alpha,n}(T-\tau)-S_\alpha(T-\tau)]\Psi(\tau)\,g_j |_H^2 
 d\tau \right] \longrightarrow 0
\end{eqnarray*}
as $n \to +\infty$ for any fixed $\alpha >0$.

Hence, by the Lebesgue dominated convergence theorem we obtained 
\begin{equation}\label{eq27f}
 \lim_{n\to +\infty} \sup_{t\in [0,T]} \mathbb{E} \left| 
 W_{\alpha,n}^\Psi(t)- W_\alpha^\Psi(t)\right|_H^2 =0.
\end{equation}
By Proposition \ref{pSW1}, $P(W_\alpha^\Psi (t)\in D(A))=1$.

For any $n\in\mathbb{N}$, $\alpha >0$, $t\ge 0$, we have 
$$ |A_n W_{\alpha,n}^\Psi (t) - A W_\alpha^\Psi (t)|_H \le 
    N_{n,1}(t) +  N_{n,2}(t), $$
where 
\begin{eqnarray*}    
 N_{n,1}(t) & := & |A_n W_{\alpha,n}^\Psi (t) - A_n W_\alpha^\Psi (t)|_H , \\
 N_{n,2}(t) & := & |A_n W_\alpha^\Psi (t) - A W_\alpha^\Psi (t)|_H =
           	 |(A_n-A)W_\alpha^\Psi (t)|_H  \,.
\end{eqnarray*}
Then 
\begin{equation}\label{eq28f}
 |A_n W_{\alpha,n}^\Psi (t) - A W_\alpha^\Psi (t)|_H^2  <
 3 [ N_{n,1}^2 (t) + N_{n,2}^2(t)] \,.
\end{equation}

Let us study the term $N_{n,1}(t)$. Note that, either in cases {\bf (A1)}
or {\bf (A1)} the unbounded operator $A$ generates a semigroup. Then we have
from the Yosida approximation the following properties:
\begin{equation}
A_nx=J_nAx \quad \mbox{for~any~~} x\in D(A), \quad \sup_n ||J_n||<\infty,
\end{equation}
where $A_nx=nAR(n,A)x=AJ_nx$ for any $x\in H$ with $J_n:=nR(n,A)$.
Moreover (see \cite[Chapter II, Lemma 3.4]{EN00}):
\begin{eqnarray}\label{eq30f}
 \lim_{n\to\infty} n\,R(n,A)x &=& x \qquad \mbox{for~any~} x\in H, \\
 \lim_{n\to\infty} A_nx &=& Ax \qquad \mbox{for~any~} x\in D(A).
 \label{eq31f}
\end{eqnarray}
Note that $AS_{\alpha,n}(t)x = S_{\alpha,n}(t)Ax$ for all $x \in
D(A),$  since $e^{t A_n}$ commutes with $A$ and $A$ is closed (see
(\ref{eq14f})). So, by Proposition \ref{pSW1} and again the
closedness of $A$  we can write
\begin{eqnarray*}
 A_n W_{\alpha,n}(t) & \equiv & 
 A_n \int_0^t S_{\alpha,n}(t-\tau)\Psi(\tau)dW(\tau) \\
 & = & nR(n,A) \left[\int_0^t S_{\alpha,n}(t-\tau)A\Psi(\tau)dW(\tau)\right].
\end{eqnarray*}
Analogously, 
$$  A_n W_\alpha(t) =
  nR(n,A) \left[\int_0^t S_\alpha(t-\tau)A\Psi(\tau)dW(\tau)\right].
$$
By (\ref{eq30f}) we have
\begin{eqnarray*}
N_{n,1}(t) &=& |J_n\int_0^t [S_{\alpha,n}(t-\tau)-S_\alpha(t-\tau)]
  A\Psi(\tau)dW(\tau)|_H \\
  &\le & |\int_0^t [S_{\alpha,n}(t-\tau)-S_\alpha(t-\tau)]
  A\Psi(\tau)dW(\tau)|_H \;.
\end{eqnarray*}
From assumptions, $A\Psi \in \mathcal{N}^2(0,T;L_2^0)$. Then the difference
$[S_{\alpha,n}(t-\tau)-S_\alpha(t-\tau)]A\Psi(\tau)$ may be estimated exactly
like the difference $\Phi_{\alpha,n}$ defined by (\ref{eq21b}).

Hence, from (\ref{eq30f}) and (\ref{eq27f})
for the first term of the right hand side of (\ref{eq28f}) we obtain
$$ \lim_{n\to +\infty}\;\; \sup_{t\in [0,T]} 
 \mathbb{E}(N_{n,1}^2 (t)) \to 0. $$
For the second and third terms of (\ref{eq28f}) 
we can follow the same steps as above for proving (\ref{eq27f}).
We have to use the properties of Yosida approximation, particularly the
convergence (\ref{eq31f}).
So, we can deduce that
$$ \lim_{n\to +\infty} \;\;\sup_{t\in [0,T]} \mathbb{E}
  |A_n W_{\alpha,n}^\Psi(t)-AW_\alpha^\Psi(t)|_H^2=0 , $$
what gives (\ref{deq5}).\qed \end{proof} 

Now, we are able to formulate the main result of this section.

\begin{theorem} \label{coSW4f}
Suppose that assumptions of Lemma \ref{pSW5f} hold. Then the
equation (\ref{eq4f}) has a strong solution.
Precisely, the convolution $W_\alpha^\Psi$ defined by
(\ref{eSW18af}) is the strong solution to~(\ref{eq4f}).
\end{theorem}
\begin{proof} We have to show only the condition (\ref{eSW3.1f}). The
convolution $W_\alpha^\Psi (t)$ has integrable trajectories (see
section \ref{SVEHSsec:1}), that is, $W_\alpha^\Psi (\cdot )\in L^1([0,T];H)$,
P-a.s. The closed linear unbounded operator $A$ becomes bounded on
($D(A),|\cdot|_{D(A)}$), see \cite[Chapter 5]{We80}. So, we obtain
$AW_\alpha^\Psi (\cdot )\in L^1([0,T];H)$, P-a.s. Hence, the
function $a_\alpha(T-\tau)AW_\alpha^\Psi (\tau)$ is integrable
with respect to $\tau$, what finishes the proof. \qed \end{proof}

The following result is an immediate consequence of Corollary
\ref{cor5f} and Theorem \ref{coSW4f}.

\begin{corollary} \label{cor6f}
Assume that (VA) hold, $A$ generates an analytic semigroup of 
angle $\pi/2$ and $\alpha \in (0,1)$.
If $\Psi$ and $A\Psi$ belong to $\mathcal{N}^2(0,T;L_2^0)$ and in addition
$\Psi(t)(U_0)\subset D(A), ~P-a.s.$,
then the equation (\ref{eq4f}) has a strong solution.
\end{corollary}

\section{Examples}\label{c3s4Ex}
In this short section we give several examples fulfilling conditions 
of theorems providing existence of strong solutions. 
The class of such equations depends on where
the operator $A$ is defined, in particular, the domain of $A$
depends on each considered problem, and also depends on the
properties of the kernel function $a(t)$, $t\ge 0$.

Let $G$ be a bounded domain in $\mathbb{R}^n$ with smooth
boundary $\partial G$. Consider the differential operator of
order $2m$:
\begin{equation}{\label{eq1}}
A(x,D)= \sum_{|\alpha|\leq 2m } b_{\alpha}(x) D^{\alpha}
\end{equation}
where the coefficients $b_{\alpha}(x)$ are sufficiently smooth
complex-valued functions of $x$ in $\overline G.$ The
operator $A(x,D)$ is called {\tt strongly elliptic} if there
exists a constant $c>0$ such that
$$ 
Re(-1)^m \sum_{|\alpha|= 2m } b_{\alpha}(x) \xi^{\alpha} \geq c
|\xi|^{2m}
$$ 
for all $ x\in \overline G$ and $\xi \in \mathbb{R}^n.$

Let $A(x,D)$ be a given strongly elliptic operator on a bounded
domain $G \subset \mathbb{R}^n$ and set $D(A)= H^{2m}(G)
\cap H_0^{m}(G).$ For every $u \in D(A)$ define
$$ 
Au = A(x,D)u.
$$ 
Then the operator $-A$ is the infinitesimal generator of an
analytic semigroup of operators on $H= L^2(G)$ (cf.
\cite[Theorem 7.2.7]{Pa83}). We note that if the operator $A$ has
constant coefficients, the result remains true for the domain
$G = \mathbb{R}^n.$

The next example is the Laplacian
$$ 
\Delta u = \sum_{i=1}^n \frac{\partial^2 u}{\partial x_i^2},
$$ 
since $- \Delta$ is clearly strongly elliptic. It follows that
$\Delta u$ on $D(A) = H^2( G) \cap H_0^1(G)$ is the
infinitesimal generator of an analytic semigroup on $L^2(G)$.

In particular, by \cite[Corollary 2.4]{Pr93} the operator $A$
given by (\ref{eq1}) generates an analytic resolvent $S(t)$
whenever $a \in C(0,\infty) \cap L^1(0,1)$ is completely
monotonic.

This example fits in our results if $a(t)$ is
also completely positive. For example: $a(t) =
t^{\alpha-1}/\Gamma(\alpha)$ is both, completely positive and
completely monotonic for $0 < \alpha \leq 1$ (but not for $\alpha
>1$).

Another  class of examples is provided by the following: suppose
$a \in L^1_{loc}(\mathbb{R}_+ )$ is of subexponential growth and
$\pi/2$-sectorial, and let $A$ generate a bounded analytic
$C_0$-semigroup  in a complex Hilbert space $H$.  Then it follows
from \cite[Corollary 3.1]{Pr93} that the Volterra equation of scalar
type $ u= a*Au +f$ is parabolic. If, in addition, $a(t)$ is
$k$-regular for all $k\geq 1$ we obtain from \cite[Theorem 3.1
]{Pr93} the existence of a
 resolvent $S\in C^{k-1}((0,\infty),
\mathcal{B}(H))$ such that $ \mathcal{R}(S(t)) \subset D(A)$ for
$t>0$ (see \cite[ p.82 (f)]{Pr93}).

\chapter{Stochastic Volterra equations in spaces of 
distributions} \label{SEEMDch:4} 
\chaptermark{Stochastic Volterra equations in spaces of distributions}

In this chapter we study two classes of linear Volterra 
equations driven by spatially homogeneous Wiener process.
We consider existence of solutions to these equations in the space of 
tempered distributions and then derive conditions under which the solutions
are function-valued or even continuous. The conditions obtained are expressed 
in terms of spectral measure and the space correlation of the noise process,
as well. Moreover, we give description of asymptotic
properties of solutions.

The chapter is organized as follows. In section \ref{PizSS2} we introduce
generalized and classical homogeneous Gaussian random fields basing on 
\cite{GV64}, \cite{Ad81} and \cite{PZ97}. We recall some
facts which connect the generalized random fields with their space corralations
and spectral measures. Moreover, we recall some results used in the proofs of
the main theorems. Section \ref{SEEMTDsec:2} originates from \cite{KZ00a}.
Here we study regularity of solutions to the equation (\ref{PizE1.1}) and
provide some applications of these results. Section \ref{SEEMTDsec:3} is a
natural continuation of section \ref{SEEMTDsec:2}. In this section we give
necessary and sufficient conditions for the existence of a limit measure to the
stochastic equation under consideration. Results of section \ref{SEEMTDsec:3}
come from \cite{Ka03}. In section \ref{SEEMTDsec:4} we study an
integro-differential stochastic equation with infinite delay. We provide 
necessary and sufficient conditions under which weak solution to that 
equation takes values in a Sobolev space. Section \ref{SEEMTDsec:4} 
originates from \cite{KL07a}.


\section{Generalized and classical  homogeneous \protect \\Gaussian random
fields} \label{PizSS2}

We start from recalling several concepts needed in this chapter. 

\index{$S(\mathbb{R}^d)$}
\index{$S_c(\mathbb{R}^d)$}
\index{$S'(\mathbb{R}^d)$}
\index{$S_c'(\mathbb{R}^d)$}
\index{$\langle \xi,\psi\rangle$}
\index{$\psi_{(s)}$}
Let $S(\mathbb{R}^d)\,$, $S_{c}(\mathbb{R}^d)\,$, denote respectively the spaces of
all infinitely differentiable  rapidly
decreasing real and complex  functions on $\mathbb{R}^d$ and $S'(\mathbb{R}^d)$,\,
$S_{c}'(\mathbb{R}^d)$
denote the
spaces of real and complex, tempered distributions. The value of a
distribution
$\xi\in S_{c}'(\mathbb{R}^d)$ on a test function $\psi$ will be written as
$\langle \xi,\psi\rangle$.
For $\psi\in S(\mathbb{R}^d)$ we set $\psi_{(s)}(\theta)=\overline
{\psi(-\theta)}$, $\theta\in\mathbb{R}^d$. Denote by $S_{(s)}(\mathbb{R}^d)$
the space
of all $\psi\in S(\mathbb{R}^d)$ such that $\psi=\psi_{(s)}$, and by
$S'_{(s)}(\mathbb{R}^d)$ the space of all $\xi\in S'(\mathbb{R}^d)$ such that
$\langle\xi,\psi\rangle=\langle\xi,\psi_{(s)}\rangle$ for every
$\psi\in S(\mathbb{R}^d)$.
\index{$S_{(s)}'(\mathbb{R}^d)$}

We define the derivative $\dot{\xi}$ of the distribution 
$\xi\in S'(\mathbb{R}^d)$ by the
formula $\; \langle\dot{\xi},\varphi\rangle =
 -\langle\xi,\dot{\varphi}\rangle\;$ for $\varphi\in S(\mathbb{R}^d)$, 
 see \cite{GS64}. 

In the chapter we denote by ${{\cal F}}$ the 
Fourier transform\index{transform!Fourier} both on
$S_{c}(\mathbb{R}^d)$, and on $S_{c}'(\mathbb{R}^d)$. In particular,
\index{${\cal F}\psi(\theta)$}
$$
{{\cal F}}\psi(\theta)=\int_{\mathbb{R}^d}e^{-2\pi
i\langle\theta,\eta\rangle}
\psi(\eta)d\eta, \,\,\,\psi \in S_{c}(\mathbb{R}^d),
$$
and for the inverse Fourier transform ${{\cal F}}^{-1},$
\index{${\cal F}^{-1}\psi(\theta)$}
$$
{{\cal F}}^{-1}\psi(\theta)=\int_{\mathbb{R}^d}e^{2\pi
i\langle\theta,\eta\rangle}\psi(\eta)d\eta,\,\,\,\psi \in
S_{c}(\mathbb{R}^d) .
$$
Moreover, if $\xi\in S_{c}'(\mathbb{R}^d)$,
$$
\langle{{\cal F}}\xi,\psi\rangle=\langle\xi,{{\cal F}}^{-1}\psi\rangle 
$$
for all $\psi\in S_{c}(\mathbb{R}^d)$. Let us note that ${{\cal F}}$
transforms the space of tempered
distributions
$S'(\mathbb{R}^d)$ into $S'_{(s)}(\mathbb{R}^d)$.

For any $h\in \mathbb{R}^d$, $\psi \in S(\mathbb{R}^d)$, $\xi \in
S'(\mathbb{R}^d)$, the {\tt translations} $\tau_{h}\psi$, $\tau_{h}' \xi$
are defined by the formulas
\index{translations}
\index{$\tau_{h}\psi$}
\index{$\tau_{h}' \xi$}
$$
\tau_{h} \psi (x) = \psi (x-h), \qquad \langle \tau_{h}' \xi\,,\psi\rangle =
\langle \xi\,,\tau_{h}\psi\,\rangle, \quad x\in \mathbb{R}^d.
$$

\index{${\cal B}(S'(\mathbb{R}^d))$}
\index{${\cal B}(S_{c}'(\mathbb{R}^d))$}
By ${\cal B}(S'(\mathbb{R}^d))$ and ${\cal B}(S_{c}'(\mathbb{R}^d))$  
we denote the smallest
$\sigma$--algebras of  subsets of $S'(\mathbb{R}^d)$ and  
$S_{c}'(\mathbb{R}^d)$, respectively, such that for  any
test function $\varphi $ the mapping $\xi\to\langle\xi,\varphi\rangle$ is
measurable.
\vskip1mm

The below notions of generalized random fields, their space correlations and
spectral measures are recalled directly from \cite{GV64}.
\vskip1mm

\index{generalized random field}
Let $(\Omega,{{\cal F}},P)$ be a complete probability space. Any
measurable mapping  $Y:\
\Omega\to S'(\mathbb{R}^d)$ is called a {\tt generalized random field}\,.
\index{generalized random field!Gaussian}
\index{generalized random field!homogeneous}
A generalized random field  $Y \,$ is  called
{\tt Gaussian}  if $\langle Y,\varphi\rangle$ is a Gaussian random variable
for any $\varphi\in S(\mathbb{R}^d)$. The   definition implies  that for any
functions
$\varphi_1,\ldots, \varphi_n\in S(\mathbb{R}^d)$ the  random vector $(\langle
Y,\varphi_1\rangle, \ldots, \langle Y,\varphi_n\rangle)$ is also Gaussian.
One says that a generalized random field $Y \,$ is
{\tt homogeneous} or {\tt stationary}  if for all $h\in\mathbb{R}^d$, the
translation $\tau'_h(Y)$ of $Y$ has the same probability law as $Y$.
\index{generalized random field!stationary}

\index{distribution!positive-definite}
A distribution $\Gamma$ on the space $S(\mathbb{R}^d)$ is called {\tt 
positive-definite} if \linebreak 
$\langle \Gamma,\varphi\star\varphi_{(s)}\rangle\ge 0$
for every $\varphi\in S(\mathbb{R}^d)$, where $\varphi\star\varphi_{(s)}$
denotes the convolution of the functions $\varphi$ and $\varphi_{(s)}$.

If $Y$ is a homogeneous, Gaussian random field then
for each $\psi \in S(\mathbb{R}^d)$, $\langle Y,\psi\rangle$ is a Gaussian
random variable and the  bilinear  functional $q:\ S(\mathbb{R}^d)\times
S(\mathbb{R}^d)\to\mathbb{R}$ defined by the formula,
\index{$q(\varphi,\psi)$}
$$
q(\varphi,\psi)=\mathbb{E}(\langle Y,\varphi\rangle\langle
Y,\psi\rangle) \quad \mbox{for} \quad \varphi,\ \psi\in S(\mathbb{R}^d)\,,
$$
is continuous and positive definite.
Since
$q(\varphi,\psi)=q(\tau_h\varphi,\tau_h\psi)$ for all
$\varphi$, $\psi\in S(\mathbb{R}^d)$, $h\in\mathbb{R}^d$, there exists, see e.g.
\cite[Chapter II]{GV64}, a unique    positive-definite distribution
$\Gamma\in S'(\mathbb{R}^d)$ such that for all $\varphi,\ 
\psi\in S(\mathbb{R}^d)\,,$ one has
$$
q(\varphi,\psi)=\langle\Gamma,\varphi*\psi_{(s)}\rangle\,.
$$
The distribution $\Gamma$ is called the {\tt space correlation} of
\index{$\Gamma$}
\index{space correlation}
the field $Y$. By Bochner--Schwartz theorem the positive-definite
distribution $\Gamma$ is the inverse Fourier transform of a 
unique positive, symmetric, tempered measure $\mu$ on  $\mathbb{R}^d: 
\Gamma={\cal F}^{-1}(\mu)$. The measure $\mu$ is
called the {\tt spectral measure} of $\Gamma$ and of the field $Y$.
\index{spectral measure}
\index{$\mu$}

Summing up a generalized homogeneous Gaussian random field $Y$ is
characterized by the following  properties:
\begin{enumerate}
\item for any $\psi\in S(\mathbb{R}^d)$, $\langle Y,\psi\rangle$ is
a real--valued Gaussian random variable,
\item there exists a positve-definite
distribution $\Gamma\in S'(\mathbb{R}^d)$ such that
for all $\varphi,\ \psi\in S(\mathbb{R}^d)$
$$
\mathbb{E}\left(\langle Y,\varphi\rangle
\langle Y,\psi\rangle\right)=\langle\Gamma,\varphi*\psi_{(s)}\rangle,
$$
\item the distribution $\Gamma$ is the inverse Fourier transform of a
positive and symmetric {\tt tempered measure} 
\index{tempered measure} 
$\mu$ on $\mathbb{R}^d$, that is, such that
$$
\int_{\mathbb{R}^d}\left(1+|\lambda|\right)^r \,\mu(d\lambda)<+\infty,\quad
{\rm for~some~} r<0\,.
$$
\end{enumerate}

\index{random field or classical random field}
Let $Y:\ \Omega\to S'(\mathbb{R}^d)$ 
be a generalized random field. When the values of $Y$ are
functions, with probability $1$, then $Y$ is  called a  {\tt classical
random field} or shortly {\tt random field}. In this case, by Fubini's
theorem, for any $\theta\in\mathbb{R}^d$ the function $Y(\theta)$ is
well--defined and 
$$
\langle Y,\varphi\rangle
=\int_{\mathbb{R}^d} Y(\theta)\varphi(\theta)d\theta\,,
$$
for any $\varphi\in S(\mathbb{R}^d)$.
Thus any random field $Y$ may be identified with the family of random
variables $\{Y_{\theta}\}_{\theta\in\mathbb{R}^d}$
parametrized  by
$\theta\in\mathbb{R}^d$.
In particular a homogeneous (stationary), Gaussian random field is a family
of Gaussian random
variables $Y(\theta)$, $\theta \in \mathbb{R}^d$, with Gaussian laws invariant
with  respect to
all translations. That is,  for any  $\theta_1,\ldots,\theta_n
\in\mathbb{R}^d$ and $h\in\mathbb{R}^d$, the law of
$(Y(\theta+h),\ldots,Y(\theta_n+h))$
does not depend on $h\in\mathbb{R}^d$.

For the sake of completness we sketch now the proof of the following
result, (see  also  \cite{PZ97}).

\begin{proposition}\label{PizPth1}
A generalized, homogeneous, Gaussian random field $Y$ is classical  if and
only if the space
correlation $\Gamma$
of $Y$ is a bounded function and if and only if the
spectral measure $\mu$ of $Y$ is finite.
\end{proposition}

\begin{proof}
First, let us prove that if a positive definite distribution $\Gamma$
is a bounded function then
it  is continuous and its spectral measure $\mu$ is finite.
It is enough to show only that the spectral measure $\mu$ of the
distribution $\Gamma$ is finite.

Let $p_t(\cdot)$ denote the normal density with the Fourier transform
$e^{-t|\lambda|^2}$, $t>0$, $\lambda\in\mathbb{R}^d$. Define measures
$\mu_t$,
$t>0$, by the formula
$$
\mu_t(B)=\int_B e^{-t|\lambda|^2}\mu(d\lambda),\quad
B\subset\mathbb{R}^d.
$$
Since the measure $\mu$ is tempered, the measures $\mu_t$ are finite.

\noindent The Fourier transform ${{\cal F}}(\mu_t)$ of $\mu_t$, for any $t>0$,
is a continuous function and
$$
\Gamma_t(\theta) ={{\cal F}}^{-1}(\mu_t)(\theta)
=(\Gamma*p_t)(\theta),\quad\theta\in\mathbb{R}^d.
$$
But
$$
\Gamma_t(\theta) = \int_{\mathbb{R}^d}\Gamma(\theta-
\eta)p_t(\eta)d\eta
$$
and
$$
|\Gamma_t(\theta)|
\le \left[\sup_{\zeta\in\mathbb{R}^d}|\Gamma(\zeta)|\right]
\int_{\mathbb{R}^d}p_t(\eta)d\eta
\le \left[\sup_{\zeta\in\mathbb{R}^d}|\Gamma(\zeta)|\right]\,.
$$
In particular
$$
|\Gamma_t(0)|=\int_{\mathbb{R}^d}e^{-t|\lambda|^2}\mu(d\lambda)
\le\left[\sup_{\zeta\in\mathbb{R}^d}|\Gamma(\zeta)|\right]\,.
$$
Letting $t\downarrow 0$ one obtains that
$$
\int_{\mathbb{R}^d}\mu(d\lambda)\le
\left[\sup_{\zeta\in\mathbb{R}^d}|\Gamma(\zeta)|\right]<+\infty\,,
$$
so the measure $\mu$ is finite as required.

Let now $Y$ be a classical, homogeneous, Gaussian random field. It means
that the field $Y$ is function--valued. Moreover 
$\mathbb{E}(Y(\theta_1)Y(\theta_2)) 
=\Gamma(\theta_1-\theta_2)$, for
$\theta_1$, $\theta_2\in\mathbb{R}^d$, the space correlation $\Gamma$ is
positive--definite and
$|\Gamma(\theta_1-\theta_2)|\le \Gamma(0)<+\infty$.

Let now $\mu$ be the finite spectral measure of a homogeneous
Gaussian random field $Y$. Then  $\Gamma$ is   a positive definite
continuous function. By Kolmogorov's existence theorem, there exists a
family ${\tilde Y}(\theta)$, $\theta \in\mathbb{R}^d$, such that:
$$
\mathbb{E}({\tilde Y}(\theta_1)\, {\tilde Y}(\theta_2)) = 
   \Gamma (\theta_1 - \theta_2)\,,~~~~~~\theta_1, \theta_2\,\in \mathbb{R}^d\,.
$$
From the continuity of $\Gamma$, it follows that the family  ${\tilde
Y}$, is stochastically continuous and therefore has a measurable version.
Since the laws of the random fields $Y$, ${\tilde Y}$ coincide, the
result follows. 
\qed
\end{proof}

We  finish the section recalling a  continuity criterium which will be used
in the proof of the continuity results, see (\cite{Ad81}, Th. 3.4.3).
\begin{proposition}\label{PizPl1}
Let $Y(\theta)$, $\theta\in\mathbb{R}^d$, be a homogeneous,
Gaussian random field  with the spectral measure $\mu$.
If, for some $\varepsilon>0$,
$$
\int_{\mathbb{R}^d} (\ln(1+|\lambda|))^{1+\varepsilon}
\mu(d\lambda) < +\infty,
$$
then $Y$ has a version with almost surely continuous sample functions.
\end{proposition}

\section{Regularity of solutions to stochastic Volterra equations} 
\label{SEEMTDsec:2}
\sectionmark{Regularity of solutions to sVe}
This section is concerned with the following stochastic Volterra equation
\begin{equation}\label{PizE1.1}
X(t,\theta)=X_0(\theta)+\int_0^tb(t-\tau)AX(\tau,\theta)d\tau
+ W (t,\theta),
\end{equation}
where $t\in\mathbb{R}_+$, $\theta\in \mathbb{R}^d$, $X_0\in S'(\mathbb{R}^d)$, 
$b\in L^1_{loc} (\mathbb{R}_+)$ and $W $ is a spatially
homogeneous Wiener process which takes values in the space of real, tempered
distributions $S'(\mathbb{R}^d)$. The class of operators $A\,$ covered in the
present chapter contains
in particular the Laplace operator $\Delta\,$ and its fractional powers $-
(-\Delta)^{\beta/2}$, $\beta \in (0,2]$.

The equation  (\ref{PizE1.1}) is a generalization of stochastic heat and wave
equations studied by many authors,
see e.g.\ \cite{DF98}, \cite{KZ01}, \cite{KZ00}, \cite {MM00},
\cite{MS99}, \cite{Mu97}, \cite{PZ00} and \cite{Wa86} and references therein.
In the context of infinite particle systems stochastic heat equation of 
a~similar type has been investigated by Bojdecki with Jakubowski
\cite{BJ90}, \cite{BJ97} and  \cite{BJ99} and by Dawson with Gorostiza in
\cite{DG90}.

As we have already said, our aim is to obtain conditions under which
solutions to the stochastic Volterra  equation (\ref{PizE1.1})  are 
function--valued and
even continuous with respect to the space variable. In the chapter we
treat the case of general dimension and the correlated, spatially 
homogeneous noise $W_{\Gamma}$ of the general form.

\subsection{Stochastic integration}\label{PizSS3}

In this chapter we will  integrate operator--valued functions
\index{${\cal R}(t)$}
${\cal R}(t)$, $t\ge 0$, with respect to a Wiener process $W$. The
operators ${\cal R}(t)$,
$t\ge 0$, will be  non-random  and will act from some  linear subspaces
of $S'(\mathbb{R}^d)$ into $S'(\mathbb{R}^d)$.
We shall assume that $W(t)$, $t\geq 0$, is a continuous process with
independent increments taking
values in $S'(\mathbb{R}^d)$. The process $W$ is space homogeneous in the
sense that, for each $t\geq
0$,   random  variables $W(t)\,$ are stationary, Gaussian, generalized
random fields. We denote by
$\Gamma$ the  covariance of $W(1)\,$ and the associated spectral measure by
$\mu$. To underline
the fact that the probability law of $W\,$ is determined by $\Gamma\,$ we
will write $W_{\Gamma}$.
From now on we denote by  $q$ a scalar product  on $S(\mathbb{R}^d)\,$ 
given by the formula
$$
q\langle\phi, \psi\rangle= \langle\Gamma,\phi*\psi_{(s)}\rangle\,,\,\,\phi, \psi \in
S(\mathbb{R}^d).
$$

Let us present three examples of spatially homogeneous Wiener
processes.

\noindent{\bf Examples} 1.
Important examples of random fields are provided by symmetric
$\alpha$--stable
distributions $\Gamma(x)=e^{-|x|^{\alpha}}$, where $\alpha\in]0,2]$.
For $\alpha=1$ and $\alpha=2$ the densities of the spectral measures
are given
by the formulas $c_1\left(1+|x|^2\right)^{-\frac{d+1}{2}}$ and
$c_2e^{-|x|^2}$, where $c_1$ and $c_2$ are appropriate constants.\\

2.
Let $q(\psi,\varphi)=\left\langle\left(-\Delta+m^2\right)^{-1}\psi,
\varphi\right\rangle$, where $\Delta$ is the Laplace operator on
$\mathbb{R}^d$ and
$m$ is a strictly positive constant. Then $\Gamma$ is a continuous
function on
$\mathbb{R}^d\setminus\{0\}$ and $\frac{d\mu}{dx}(x)=(2\pi)^{-\frac{d}{2}}
\left(|x|^2+m^2\right)^{-1}$. The law of $W(1)$ is the 
so--called Euclidean free field.

3. 
Let $q(\psi,\varphi)=\langle\psi,\varphi\rangle$. Then $\Gamma$ is
equal to the
Dirac $\delta_0$--function, its spectral density $\frac{d\mu}{dx}$ is
the
constant function $(2\pi)^{-\frac{d}{2}}$ and $\frac{\partial W}
{\partial t}$ is a white noise on $L^2([0,\infty[\times\mathbb{R}^d)$. If
$B(t,x)$, $t\ge 0$ and $x\in\mathbb{R}^d$, is a Brownian sheet on
$[0,\infty[\times\mathbb{R}^d$, then $W$ can be defined by the
formula
$$
W(t,x)=\frac{\partial^dB(t, x)}{\partial x_1\ldots\partial x_d},\quad
t\ge 0.
$$

The crucial role in the theory of  stochastic integration with respect
 to  $W_{\Gamma}$  is
played by the Hilbert space $S'_{q}\subset S'(\mathbb{R}^d)$ called
the {\tt kernel} or the {\tt reproducing kernel}
of $W_{\Gamma}$. Namely the space
$S'_{q}$ consists of all distributions $\xi\in S'(\mathbb{R}^d)$ for which
there exists a constant $C$
such that $$
|\langle\xi,\psi\rangle|\le C\sqrt{ q(\psi,\psi)},\quad \psi\in S(\mathbb{R}^d).
$$
The norm in $S'_{q}$ is given by the formula
$$
|\xi|_{S'_{q}}=\sup_{\psi\in S}\frac{|\langle\xi,\psi\rangle|}
{\sqrt{q(\psi,\psi)}}.
$$
Let us assume that we require that the stochastic integral should take
values in
a Hilbert space $H\,$ continuously imbedded into $S'(\mathbb{R}^d)\,.$ Let
$L_{HS}(S'_{q},H)$ be
the space of Hilbert--Schmidt operators  from
$S'_{q}$ into $H$. Assume that ${\cal R}(t)$,\, $t\ge 0$ is measurable
$L_{HS}(S'_q,H)$--valued function  such that
$$
\int_0^t\|{\cal R}(\sigma)\|^2_{L_{HS}(S'_{q},H)}d\sigma
 <+\infty,\quad {\rm for\ all~} t\ge 0\,.
$$
Then the stochastic integral
$$
\int_0^t{\cal R}(\sigma)dW_{\Gamma}(\sigma),\quad t\ge 0
$$
can be defined in a standard way, see \cite{It84}, \cite{DZ92} or \cite{PZ97}. 
The stochastic integral is an $H$--valued martingale for which
$$
\mathbb{E}\left(\int_0^t{\cal R}(\sigma)dW_{\Gamma}(\sigma)\right)=0,
\quad t\ge 0
$$
and
$$
\mathbb{E}\left|\int_0^t{\cal R}(\sigma)dW_{\Gamma}(\sigma)\right|_H^2
=\mathbb{E}\left(\int_0^t\|{\cal R}(\sigma)\|^2_{L_{HS}(S'_{q},H)}d\sigma
\right),\quad t\ge 0.
$$
We will need a characterization of the space $S'_{q}$. 
In the proposition below,
$L_{(s)}^2(\mathbb{R}^d,\mu)$ denotes the subspace of 
$L^2(\mathbb{R}^d,\mu;\mathbb{C})$ consisting of all
functions $u$ such that $u_{(s)}(\theta)=u(-\theta)$ for 
$\theta\in\mathbb{R}^d$.

\begin{proposition}\label{PizP3.1} \cite[Proposition 1.2]{PZ97}
A distribution $\xi$ belongs to $S'_{q}$ if and only if
$\xi\!=\!\widehat{u\mu}$ for some $u\in\! L_{(s)}^2(\mathbb{R}^d,\mu)$.
Moreover,
if $\xi=\widehat{u\mu}$ and $\eta=\widehat{v\mu}$, then
$$
\langle\xi,\eta\rangle_{S'_{q}}
=\langle u,v\rangle_{L_{(s)}^2(\mathbb{R}^d,\mu)}.
$$
\end{proposition}

\index{${\bf r}(t)$}
The operators ${\cal R}(t)$, $t\geq 0$, of convolution type are
of special interest 
\[
{\cal R} (t) \xi \,= {\bf r}(t)*\xi\,, 
\qquad t\geq 0\,,\quad\xi \in  S'(\mathbb{R}^d)\,,  \]
with ${\bf r}(t)\in S'(\mathbb{R}^d)\,$. The convolution operator is not, in
general, defined for
all $\xi \in  S'(\mathbb{R}^d)\,$ and for the stochastic integration it is
important to know
  under what conditions on ${\bf r}(\cdot)$ and $\xi$ the
convolution  is well--defined. For many
important cases
the Fourier transform\, ${{\cal F}}{\bf r}(t)(\lambda)  \,,t\geq 0\,,
\lambda \in \mathbb{R}^d $, is continuous in both variables and, 
for any $T\ge 0$\,,
\begin{equation}\label{PizEA2}
\sup_{t\in[0,T]}\sup_{\lambda \in \mathbb{R}^d}|{{\cal F}}{\bf r}(t)(\lambda)|=
M_{T} <+\infty.
\end{equation}
If this is the case then the operators ${\cal R}(t)$ can be defined 
using Fourier transforms
$$
{\cal R} (t)\xi = {\cal F}^{-1}({\cal F}{\bf r}(t)\,{\cal F}\xi) \,,
$$
for all $\xi$ such that ${\cal F}\xi\,$ has a representation as a function.
\vspace{1mm}

Now, we can characterize the stochastic convolution as follows.

\begin{theorem}\label{PizThDawson}
Assume that the function  ${{\cal F}}{\bf r} \,$ is continuous in both
variables and satisfies
condition (\ref{PizEA2}). Then the stochastic convolution
\index{${\cal R} * W_{\Gamma}(t)$}
$$
{\cal R} * W_{\Gamma}(t) = \int_{0}^{t} {\cal R} (t-\sigma)
dW_{\Gamma}(\sigma)\,, \quad t\geq 0\,,
$$
is a well--defined $S'(\mathbb{R}^d)$--valued stochastic process. For each
$t\geq 0$, ${\cal R} * W_{\Gamma}(t)\,$ is a Gaussian, stationary, generalized
random field with the spectral measure
\begin{equation}\label{PizThDaw1}
\mu_{t}(d\lambda) = \left(\int_{0}^{t} |{\cal F}{\bf
r}(\sigma)(\lambda)|^{2} d\sigma\right) \mu(d\lambda)\,, 
\end{equation} \index{$\mu_{t}(d\lambda)$}
and with the covariance 
\begin{equation}\label{PizThDaw2}
\Gamma_{t}\,=\int_{0}^{t}{\bf r}(\sigma) * \Gamma * {\bf
r}_{(s)}(\sigma) d\sigma\,. 
\end{equation}
\end{theorem}
\begin{proof}
Let $p$ be an arbitrary continuous scalar product on  $S(\mathbb{R}^d)\,,$
such that the embedding
$\; S_{q}'(\mathbb{R}^d)\,\subset S_{p}'(\mathbb{R}^d)\;$
is Hilbert-Schmidt; here $S_{p}'(\mathbb{R}^d)$ denotes a space of 
distributions on $S(\mathbb{R}^d)$ endowed with $p$. For more information on
a family of Hilbert spaces of distributions we refer to \cite{It83}.
 Note that for $\xi\,\in S'_{q}(\mathbb{R}^d)\,$\,,
$$
\xi = {\cal F}^{-1} (u \mu )\,~~~~{\rm with}~~~~u\in L^{2}_{(s)}(\mathbb{R}^d\,,
\mu)\,.
$$
By (\ref{PizEA2}), $\;{\cal F}({\bf r}(t)){\cal F}(\xi)\,$  is a measure
$$
{\cal F}({\bf r}(t))(\lambda) u(\lambda) \mu(d\lambda)\,,
$$
belonging again to $S'_{q}(\mathbb{R}^d)\,$. Moreover, for $t\in [0,T]\,,$
$$
||{\cal R} (t) ||_{L(S'_{q},S'_{q})}
\leq \sup_{t\in[0,T]}\sup_{\lambda \in \mathbb{R}^d}|{{\cal F}}{\bf
r}(t)(\lambda)|= M_{T}< +\infty\,.
$$
Since the embedding $S'_{q} \subset S_{p}'\,$ is Hilbert-Schmidt, the
stochastic integral, by
the very definition, is an $S_{p}'\,$--valued random variable. Denote
$$
Z_{t}={\cal R} * W_{\Gamma}(t)\,.
$$
\index{$Z_{t}$}
Then we may write
$$
\mathbb{E}\left(\langle Z_t,\varphi\rangle\langle Z_t,\psi\rangle\right)
=\mathbb{E}\left(\! \left\langle\int_0^t{\cal R} (t-
\sigma)dW_{\Gamma}(\sigma),
\varphi\right\rangle\left\langle\int_0^t{\cal R}(t-
u)dW_{\Gamma}(u),
\psi\right\rangle\! \right)
$$
$$
=\mathbb{E}\left(\int_0^t\langle {\bf r}(t-
\sigma)*\varphi, dW_{\Gamma}(\sigma)\rangle
 \int_0^t\langle {\bf r}(t-
u) * \psi, dW_{\Gamma}(u)
 \rangle\right)
$$
$$
= \int_0^t\langle \Gamma, ({\bf r}(\sigma)*\varphi)*({\bf
r}(\sigma)*\psi)_{(s)}\rangle d\sigma\,
$$
where $\varphi$, $\psi \in S(\mathbb{R}^d)$ .
This   implies the formula (\ref{PizThDaw2}) of the theorem, from which
(\ref{PizThDaw1}) easily folows.
\qed
\end{proof}

As an application define
\begin{equation}\label{PizEKhinchin}
v(\lambda ) = \frac12 \langle Q\lambda,\lambda \rangle -
\int_{R^{d}}(e^{i\langle \lambda,y\rangle} - 1)\nu (dy)
\end{equation}
In the formula (\ref{PizEKhinchin}),  $Q$ is a symmetric, 
non--negative definite matrix and $\nu$ is a  
symmetric  measure concentrated on $R^d\setminus \{0\}$ such that
\begin{equation} \label{PizE4.2.3}
\int_{|y|\le 1}|y|^2\nu (dy)<+\infty,\quad\
\int_{|y|>1}1\nu (dy)<+\infty\,.
\end{equation}
This is the {\it Levy-Khinchin exponent}\index{Levy-Khinchin exponent}
of an infinitely divisible symmetric law.
From Theorem \ref{PizThDawson} we have the following proposition.
\begin{proposition}\label{PizP4}
Assume that
$$
{\cal F}{\bf r}(t)(\lambda) = e^{-t v(\lambda)}\,, \quad t\geq 0\,
$$
where $v\,$ is the Levy-Khinchin exponent given by (\ref{PizEKhinchin}) 
and (\ref{PizE4.2.3}). Then
the conditions of Theorem \ref{PizThDawson} are satisfied.
\end{proposition}

For more information on   stochastic integral with
values in the
Schwartz space of tempered distributions $S'(\mathbb{R}^d)$ we refere to
It\^ o (\cite{It83}, \cite{It84}), Bojdecki with Jakubowski
(\cite{BJ89}--\cite{BJ99}), Bojdecki with 
Gorostiza
(\cite{BG86})\, and Peszat with Zabczyk (\cite{PZ97}--\cite{PZ00}).
\vskip5mm

\subsection{Stochastic Volterra equation}\label{PizSS4}

We finally pass to the linear,  stochastic, Volterra
equation in $S'(\mathbb{R}^d)$\,
\begin{equation}\label{PizEVolterra1}
X(t)=X_0 +\int_0^tb(t-\tau)AX(\tau)d\tau+ W_{\Gamma}(t),
\end{equation}
where $X_0\in S'(\mathbb{R}^d)$, $A$ is an operator given in the Fourier
transform form
\begin{equation}\label{PizFTr}
{\cal F}(A\xi)(\lambda) = -v(\lambda)\,{\cal F}(\xi)(\lambda)\,,
\end{equation}
\index{$v(\lambda)$}
$v$ is a locally
integrable
function and $W_\Gamma$ is an $S'(\mathbb{R}^d)$--valued space homogeneous
Wiener process.

Note that if $v(\lambda)=|\lambda|^2$, then $A=\Delta$ and if 
$v(\lambda)=|\lambda|^\alpha$, $\alpha\in(0,2)$, then 
\index{fractional Laplacian}
$A=-(-\Delta)^{\frac{\alpha}{2}}$ is the fractional Laplacian.\\

We shall assume the following\index{HYPOTHESIS (H)} {\tt ~HYPOTHESIS~(H)}:\\

{\it
For any $\gamma\ge 0$, the unique solution $s(\cdot,\gamma)$ to the equation
\begin{equation}\label{LiEq5}
{\bf s}(t)+\gamma\int_0^tb(t-\tau){\bf s}(\tau)d\tau=1,\quad t\ge 0
\end{equation}
\index{${\bf s}(t)$}
fullfils the following condition:
for any $T\ge 0~$, 
$$\sup_{t\in[0,T]}\sup_{\gamma\ge 0}|{\bf s}(t,\gamma)|<+\infty .$$}

\noindent{\bf Comment~} \label{LiC1}
 Let us note that under assumption, that the function $b$ is a locally
integrable
 function, the solution $s(\cdot,\gamma)$ of the equation (\ref{LiEq5}) is
 locally integrable function and measurable with respect to both variables 
 $\gamma\geq 0$ and $t\geq 0$.\\

For some special cases the function ${\bf s}(t;\gamma)$ may be found
explicitely. Namely,
we have (see e.g.\ \cite{Pr93}):
\begin{eqnarray}\label{PizE3.6a}
&&\hspace{-1.5cm} {\rm for\ } \ b(t) = 1,\qquad {\bf s}(t;\gamma)=e^{-\gamma
t},\quad
t\ge 0,\ \gamma \ge 0;\\
\label{PizE3.6b}
&& \nonumber\\
&&\hspace{-1.5cm} {\rm for\ }\  b(t)=t,\qquad {\bf
s}(t;\gamma)=\cos(\sqrt{\gamma}t),\quad t\ge 0,\
\gamma \ge 0;\\
&&\nonumber \\
\label{PizE3.6c}
&&\hspace{-1.5cm} {\rm for\ }\ b(t)=e^{-t},~~~
{\bf s}(t;\gamma){=}(1{+}\gamma)^{-1}\left[1{+}\gamma
e^{{-}(1{+}\gamma)t}\right]\,,~~
t\ge 0,\
\gamma \ge 0.\\
\nonumber
\end{eqnarray}

We introduce now the so called {\it resolvent} family ${\cal R}(\cdot)$\,
determined by the
operator $A\,$ and the function $v\,$. Namely,
$$
{\cal R}(t)\xi={\bf r}(t)*\xi,\quad \xi\in S'(\mathbb{R}^d),
$$
where,
$$
{\bf r}(t)= {\cal F}^{-1} {\bf s}(t, v(\cdot)),~~~~~~t\geq 0\,.
$$
\vskip2mm

As in the deterministic case the solution to the stochastic Volterra
equation (\ref{PizEVolterra1}) is of the form
\begin{equation} \label{PizE3n}
X(t) = {\cal R}(t)X_0\,+ \int_{0}^{t} {\cal R}(t-\tau) d
W_{\Gamma}(\tau)\,, \quad t\geq 0\,.
\end{equation}

By Theorem \ref{PizThDawson}, we can formulate the following result.

\begin{theorem}\label{PizTh2} (\cite[Theorem 1]{KZ00a})~
Let $W_{\Gamma}$ be a spatially homogeneous Wiener process   and ${\cal
R}(t)$, $t\ge 0$,  the
resolvent  for the
equation (\ref{PizEVolterra1}). If Hypothesis (H) holds then the stochastic
convolution 
$$
 {\cal R} * W_\Gamma(t) = 
 \int_0^t{\cal R}(t-\sigma)dW_{\Gamma}(\sigma)\,,~~~~~t\ge 0,
$$
is a well--defined $S'(\mathbb{R}^d)$--valued process. 
For each $t\geq 0\,$ the random
variable ${\cal R} * W_\Gamma(t)$ is
generalized, stationary random field on $\mathbb{R}^d\,$ with the spectral
measure 
\begin{equation}\label{PizEspectral}
\mu_t(d\lambda)=\left[\int_0^t\left({\bf s}(\sigma, v(\lambda)
 )\right)^2d\sigma\right]
\mu(d\lambda).
\end{equation}
\end{theorem}

By Propositions \ref{PizPth1} and \ref{PizPl1}, we can conclude the below
result.

\begin{theorem}\label{PizTh3} (\cite[Theorem 2]{KZ00a})~
Assume that the Hypothesis\, (H)\, holds. Then the process  $~{\cal R} *
W_\Gamma(t)~$ is func\-tion--valued for all $t\ge 0$ if and only if
$$
\int_{\mathbb{R}^d}\left(\int_0^t
\left({\bf s}(\sigma,v(\lambda))\right)^2 d\sigma \right)
\mu(d\lambda)<+\infty\,, \quad t \ge 0\,. 
$$
If for some $\varepsilon>0$ and all $t\ge 0$,
$$
\int_0^t\int_{\mathbb{R}^d}
\left(\ln(1+|\lambda|)\right)^{1+\varepsilon} ({\bf s}
(\sigma,v(\lambda)))^2d\sigma\mu(d\lambda)<+\infty\,, 
$$
then, for each $t\ge 0$\,,  ${\cal R} *
W_\Gamma(t)$   is a sample continuous  random field.
\end{theorem}

\subsection{Continuity in terms of $\Gamma$}\label{PizSS5}

In this subsection we provide sufficient conditions for continuity  of the
solutions in terms of the covariance kernel $\Gamma\,$ of the Wiener
process $W_{\Gamma}\,$ rather
than in terms of the spectral measure as we have done up to now. Analogical
conditions for
existence of function-valued solutions can be derived in a similar way.

\begin{theorem}\label{PizTh4}
Assume that $d\ge 2\,,$ that   $\Gamma$ is a non--negative measure and $
A=-(-\Delta)^{\alpha/2}$,\, $\alpha \in ]0,2]\,.$ If for some $\delta >0\,$,
$$
\int_{|\lambda|\le 1}\frac{1}{|\lambda|^{d-
\alpha +\delta}}\Gamma(d\lambda) < +\infty, \,\,
\int_{|\lambda|>1}\frac{1}{|\lambda|^{d+\alpha -\delta}}\Gamma(d\lambda)
<+\infty\,\,,
$$
then for the cases (\ref{PizE3.6a}), (\ref{PizE3.6b}) and (\ref{PizE3.6c}) the solution
to the stochastic Volterra equation (\ref{PizEVolterra1}) has continuous version.
\end{theorem}

The proof will be based on several lemmas. 
For any $\gamma \,\in ]0,2]\,$ denote by $p^{\gamma}_{t}\,$ the density of
\index{$p^{\gamma}_{t}$}
the $\gamma$--stable,
rotationally invariant, density on the $d$--dimensional space. Thus,
\begin{equation}\label{PizEstable}
e^{-t|\lambda |^{\gamma}} = {\cal F}p^{\gamma}_{t}(\lambda)\,\,.
\end{equation}

\begin{lemma}\label{PizLst1}
For arbitrary $t > 0\,$ and arbitrary $x \,\in {\mathbb{R}^d}\,,$
$$
p^{\gamma}_{t}(x) = t^{- d/{\gamma}} p^{\gamma}_{1}(x t^{- 1/{\gamma}})\,.
$$
\end{lemma}
\begin{proof}
From (\ref{PizEstable}) we have
$$
 I: \,\,= e^{-|t^{\frac{1}{\gamma}}\lambda|\gamma} = \int_{{\mathbb{R}}^d}
   e^{i\langle t^\frac{1}{\gamma}\lambda,x\rangle}\,p^\gamma_1(x)dx =  
 \int_{{\mathbb{R}}^d}
   e^{i\langle \lambda,t^\frac{1}{\gamma}x\rangle}\,p^\gamma_1(x)dx \,.
$$
Introducing a new variable $y = t^\frac{1}{\gamma}x$, one has 
$$
 I: \,\, = \int_{{\mathbb{R}}^d} e^{i\langle \lambda,y\rangle}\, p^\gamma_1
  \left(yt^{-\frac{1}{\gamma}}\right) dy
$$
and the result follows. 
\qed
\end{proof}

\begin{lemma}\label{PizLst2}
There exists a constant $c\,>0\,$ such that for all $\gamma \le 2\,,$
\index{$G^{\gamma}_{d}(x)$}
$$
G^{\gamma}_{d}(x) \stackrel{{\rm df}}= \int_{0}^{+ \infty}
e^{-t}p^{\gamma}_{t}(x) dt \leq {\frac
{c}{|x|^{d+ \gamma}}} , \quad x\,\in {\mathbb{R}^d}\,.
$$
\end{lemma}
\begin{proof}
It is well-known, see e.g.\ Gorostiza and Wakolbinger \cite[(p.~286)]{GW91},
that for some constant $c_1 > 0$:
\begin{equation}\label{PizE14}
 p^\gamma_1(x) \le \frac{c_1}{1 + |x|^{d+\gamma}}\,,~~~~x \in {\mathbb{R}}^d\,.
\end{equation}
From Lemma \ref{PizLst1} and the estimate (\ref{PizE14}) we obtain:
\begin{eqnarray*}
G^\gamma_d(x) &=& \int^{+\infty}_0 e^{-t}t^{-\frac{d}{\gamma}}
     p^\gamma_1\left(x\,t^{-\frac{1}{\gamma}}\right)\, dt   
  \le \int^{+\infty}_0 e^{-t} t^{-\frac{d}{\gamma}} \frac{c_1}{1 + \left|
	 xt^{-\frac{1}{\gamma}}\right|^{d+\gamma}}\, dt \\
 &\le& \int^{+\infty}_0 e^{-t} \,t^{-\frac{d}{\gamma}} 
	\frac{c_1 t^{\frac{d+\gamma}{\gamma}}} { t^{\frac{d+\gamma}{\gamma}} 
	+ |x|^{d+\gamma}}\, dt 
 \le \int^{+\infty}_0 e^{-t}\, t \frac{c_1}{ t^{\frac{d+\gamma}{\gamma}}
	  + |x|^{d+\gamma}}\, dt\\
 &\le& \frac{c_1}{|x|^{d+\gamma}} \, \int^{+\infty}_0 e^{-t}\,t\,dt\,.
\end{eqnarray*}
\qed
\end{proof}

\begin{lemma}\label{PizLst3}
If  $\gamma < d\,, \gamma \,\in ]0,2]$, \, then there exists a constant
$c\,>0\,$ such that,
$$
G^{\gamma}_{d}(x)\,\leq  {\frac {c}{|x|^{d- \gamma}}}\,\quad {\rm for}
 \quad |x| < 1\,.
$$
\end{lemma}
\begin{proof}
Since
$$
 G^\gamma_d(x) \le \int^{+\infty}_0 p^\gamma_t(x) \,dt\,,
$$
the result follows from the well-known formula for Riesz $\gamma$--potential,
see e.g.\ Landkof \cite{La75}.
\qed
\end{proof}

\noindent
{\bf Conclusion~}  There exists a constant $c > 0$ such that, if
$\gamma < d$, $\gamma \leq 2$, then:
$$
G^\gamma_d(x) \le \frac{c}{|x|^{d-\gamma}} \quad {\rm if} \quad
  |x| \le 1\,,
$$
and
$$
 G^\gamma_d(x) \le \frac{c}{|x|^{d+\gamma}} \quad {\rm if} \quad
  |x| \ge 1\,.
$$
\vskip2mm

\noindent
{\bf Proof of Theorem \ref{PizTh4}}
Now, we pass to the proof of the theorem and restrict to the case of
$b(t) = 1$, $t\ge 0$, that is, to stochastic heat equation.
(Proof for the next two cases may be obtained in a similar way.)
We have ${\bf s}(t;\gamma) = e^{-\gamma t}$, $t\ge 0$, $\gamma \ge 0$ and 
therefore
$$
 {\bf s}(t;v(\lambda)) = e^{-v(\lambda)t}\,,~~~~~~~~\lambda \in 
 {{\mathbb{R}}^d}\,,
  t\ge 0\,.
$$
By Theorem \ref{PizTh3}, if for some $\varepsilon > 0$ and all $t>0$,
\begin{equation}\label{Pe3.15}
\int_{{\mathbb{R}}^d} (\ln(1 +|\lambda|))^{1+\varepsilon} 
\left[\int^t_0 e^{-2v(\lambda)\sigma}d\sigma\right] \,\mu(d\lambda)  
< +\infty\,,
\end{equation}
then for all $t>0$, the solution of the stochastic equation (in the general
form), has a~continuous version.
Taking into account that $a$ is a non--negative continuous function, one can
replace (\ref{Pe3.15}) by
\begin{equation}\label{Pe3.16}
\int_{{\mathbb{R}}^d} (\ln(1 +|\lambda|))^{1+\varepsilon} 
\frac{1}{1+v(\lambda)} \,\mu(d\lambda) <+\infty\,.
\end{equation}
Since we have assumed that $A = -(-\Delta)^{\alpha/2}$, then $v(\lambda) =
|\lambda|^\alpha$, with $\alpha \in]0,2]$ and therefore (\ref{Pe3.16}) 
becomes
\begin{equation}\label{Pe3.17}
\int_{{\mathbb{R}}^d} (\ln(1 +|\lambda|))^{1+\varepsilon} 
\frac{1}{1+|\lambda|^\alpha} \,\mu(d\lambda)  < +\infty\,.
\end{equation}
However, the condition (\ref{Pe3.17}) holds for some $\varepsilon > 0$ if for some
$\delta > 0$
$$
\int_{{\mathbb{R}}^d} \frac{1}{1+|\lambda|^{\alpha-\delta}} \,\mu(d\lambda) <
+\infty\,. 
$$
In the same way as in the paper \cite{KZ00} by Karczewska and Zabczyk, for some
constant $c > 0$:                                  
$$
\int_{{\mathbb{R}}^d} \frac{1}{1+|\lambda|^\gamma} \,\mu(d\lambda) =
 c\int_{{\mathbb{R}}^d} G^\gamma_d(x)\Gamma(dx)\,,
$$
where $\gamma: = \alpha-\delta$. Taking into account Lemma \ref{PizLst2} 
and Lemma \ref{PizLst3}, the result follows.
\qed \medskip

\subsection{Some special cases}

In this subsection we illustrate the main results obtained considering 
 several special cases.

Let us recall that the linear stochastic Volterra equation 
(\ref{PizEVolterra1}) considered
in the chapter has the following form
$$
 X(t) = X_0 + \int^t_0 \,b(t-\tau)\, A\,X(\tau)d\tau + W_\Gamma(t)\,,
$$
where $X_0 \in S'({\mathbb{R}}^d)$, $A$ is an operator given in the Fourier 
transform form
$$
 {\cal F} (A\xi)(\lambda) = -v(\lambda)\,{\cal F}(\xi)(\lambda)\,, \quad
  \xi \in S'({\mathbb{R}}^d)\,,
$$
$v$ is a locally integrable function and $W_\Gamma$ is an 
$S'({\mathbb{R}}^d)$--valued space homogeneous Wiener process. This equation is
determined by three objects: the spatial correlation $\Gamma$
of the process $W_\Gamma$, the operator $A$ and the function $v$ or,
equivalently, by the spectral 
measure $\mu$, the function $a$ and the function ${\bf s}$,
respectively.

We apply our Theorems \ref{PizTh3} and \ref{PizTh4} to several 
special cases corresponding to particular choices of
functions $v$, $a$ and of the measure $\mu$. We will
assume, for instance, that $b(t) = 1$ or $b(t) = t$ or
$b(t) = e^{-t}$, $t\ge 0$, that $v(\lambda) = |\lambda|^\alpha$,
$\alpha\in ]0,2]$, $\lambda\in{\mathbb{R}}^d$ and that the 
measure $\mu$ is either finite or $\mu(d\lambda) =
\frac{1}{|\lambda|^\gamma}d\lambda$, $\gamma \in ]0,d[$. Note that if
$v(\lambda) = |\lambda|^2$, then $A = \Delta$ and if $v(\lambda) =
|\lambda|^\alpha$, $\alpha \in ]0,2[$, then $A = -(-\Delta)^{\alpha/2}$
is the fractional Laplacian. In all considered cases we assume that
Hypothesis (H), on the function $v$, holds.

\medskip
\noindent
{\bf Case 1}\, If (H) holds, the function $a$ is given by
(\ref{PizEKhinchin}) and (\ref{PizE4.2.3}) and the measure $\mu$ is finite then ${\cal R} * W_\Gamma$
is a function--valued process. To see this note that by (H) and Theorem
\ref{PizTh2}, the measure $\mu_t$ given by (\ref{PizEspectral}) is finite. 
So, the result follows from Theorem \ref{PizTh3}.

\medskip
\noindent
{\bf Case 2}\, If (H) holds, the function $a$ is given by 
(\ref{PizEKhinchin}) and (\ref{PizE4.2.3})
and $\mu$ is a measure such that for some $\varepsilon > 0$,
$$
 \int_{{\mathbb{R}}^d} (\ln(1 + |\lambda|)^{1+\varepsilon} \,
 \mu(d\lambda) < +\infty\,,
$$
then for arbitrary $t > 0$, ${\cal R} * W_\Gamma(t)$ is a continuous
random field. This follows immediately from Theorem \ref{PizTh3}.

\bigskip
\noindent
{\bf Case 3}\,\, Assume that $b(t) = 1$ or $b(t) = t$ or $b(t)
= e^{-t}$, $t\ge 0$, $A = \Delta$  (Laplace operator) and $\Gamma(x) =
\Gamma_\beta(x) = \frac{1}{|x|^\beta}$, $\beta\in [0,d[$. 
\index{$\Gamma_\beta(x)$}
Then function
${\bf s}$ is given by formulas (\ref{PizE3.6a}), 
(\ref{PizE3.6b}) and (\ref{PizE3.6c}), respectively.
Function $v(\lambda) = |\lambda|^2$, and the spectral measure
$\mu_\beta$ corresponding to $\Gamma_\beta$ is of the form
$\mu_\beta(d\lambda) = \frac{c_\beta}{|\lambda|^{d-\beta}}$, with
$c_\beta$ a positive constant. To simplify notation we assume that
$d\ge 2$. Then ${\cal R} * W_\Gamma$ is a function-valued process if an
only if $\beta \in ]0,2[$, see (\cite{KZ00}). Moreover, if $\beta \in
]0,2[$ then for each $t > 0$, ${\cal R} * W_\Gamma(t)$, is a continuous
random field. To prove this we use Theorem \ref{PizTh4} and show that for some
$\delta > 0$,
\begin{equation}\label{PizE18}
\int_{|x|<1} \frac{1}{|x|^{d-2+\delta}}\, \Gamma_\beta(x)dx < +\infty
\end{equation}
and
\begin{equation}\label{PizE19}
\int_{|x|\ge 1} \frac{1}{|x|^{d+2-\delta}}\, \Gamma_\beta(x)dx < +\infty\,.
\end{equation}
Condition (\ref{PizE19}) is always satisfied because (\ref{PizE19})
is equivalent to: $\beta > \delta - 2$.
Condition (\ref{PizE18}) may be replaced by the
following one:
$$
 \int_{|x|<1} \frac{1}{|x|^{d-2+\delta+\beta}}\,dx = 
  c\int^1_0 \frac{1}{r^{d-2+\delta+\beta}} \, r^{d-1}dr
 = c \int^1_0 \frac{1}{r^{\beta-1+\delta}} \, dr <  +\infty\,,
$$
equivalent to  $\beta < 2 -\delta$, which holds for sufficiently
small $\delta > 0$.

\bigskip
\noindent
{\bf Case 4}\,\, Assume that $b(t) = 1$ and the operator $A$ is
given by the formula
$$
{\cal F}(A\xi)(\lambda) = -v(\lambda)\,{\cal F}(\xi)\,,
$$
where
$$
 v(\lambda) =\langle Q\lambda,\lambda\rangle + \int_{{\mathbb{R}}^d}
 (1 -\cos\langle \lambda,x\rangle)\nu(dx)
$$
and $\nu$ is a symmetric measure such that
$$
 \int_{{\mathbb{R}}^d} (|x|^2 \wedge 1)\,\nu(dx) < +\infty\,.
$$
Then the equation (\ref{PizEVolterra1}) has a function--valued solution if and only if
$$
\int_{{\mathbb{R}}^d} \frac{1}{1+v(\lambda)} \,\mu(d\lambda) < +\infty\,.
$$
Additionally, if $X_0 = 0$ and
$$
\int_{{\mathbb{R}}^d} (\ln(1 +|\lambda|)^{1+\varepsilon})
  \frac{1}{1+v(\lambda)} \,\mu(d\lambda) < +\infty\,,
$$
then equation (\ref{PizEVolterra1}) has continuous version for each $t \ge 0$.\\
In this situation, ${\bf s}(\sigma, v(\lambda)) =
e^{-\sigma v(\lambda)}$. By Theorem \ref{PizTh3} the condition for
function--valued solution of the equation (\ref{PizEVolterra1}) becomes:
$$
\int_{{\mathbb{R}}^d} \left(\int^t_0({\bf
s}(\sigma,v(\lambda)))^2d\sigma\right) \mu(d\lambda) =
\int_{{\mathbb{R}}^d} \int^t_0 e^{-2\sigma v(\lambda)}
d\sigma\,\mu(d\lambda)  < +\infty\,,
$$
and it is equivalent to 
$$
\int_{{\mathbb{R}}^d} \int^t_0 \frac{1}{1+v(\lambda)}\,\mu(d\lambda) 
< +\infty\,.
$$

\section{Limit measure to stochastic Volterra equations} \label{SEEMTDsec:3}
\sectionmark{Limit measure to sVe}
This section is a natural continuation of the previous one.
Description of asymptotic properties of solutions to stochastic evolution equations
in finite dimensional spaces and Hilbert spaces is well-known and has been collected
in the monograph \cite{DZ96}. This problem has been studied for
generalized Langevin equations in conuclear spaces also by 
Bojdecki and Jakubowski
\cite{BJ99}. The question of existence of invariant and limit measures in the space of
distributions seems to be particularly interesting. Especially for stochastic
Volterra equations, because this class of equations is not well-investigated.

In the section we give necessary and sufficient conditions for the existence 
of a limit measure and describe all limit measures to the equation 
(\ref{PizE1.1}). Our results
are in a sense analogous to those formulated for the finite-dimensional and
Hilbert
space cases obtained for stochastic evolution equations, see 
\cite[Chapter 6]{DZ96}.

Let us recall the stochastic Volterra equation (\ref{PizE1.1}) in the simpler
form (\ref{PizEVolterra1}), that is, 
$$ 
X(t)= X_0 + \int_0^tb(t-\tau)AX(\tau)d\tau + W_{\Gamma}(t)\;.
$$ 
As previously,
we study this equation in the space $S'(\mathbb{R}^d)$,
where $X_0\in S'(\mathbb{R}^d)$, $A$ is an operator given in the Fourier
transform form (\ref{PizFTr}), i.e.,
$$ 
{\cal F}(A\xi)(\lambda) = -v(\lambda)\,{\cal F}(\xi)(\lambda)\,,
$$ 
where $v$ is a locally integrable
function and $W_\Gamma$ is an $S'(\mathbb{R}^d)$-valued space homogeneous
Wiener process. 




\subsection{The main results}\label{LiMain}

In this subsection we formulate results providing the existence of a limit
measure and the form of any limit measure for the stochastic Volterra equation
(\ref{PizEVolterra1}) 
with the operator $A$ given by (\ref{PizFTr}). 
In our considerations we assume that the  Hypothesis~(H) holds. 


Let us recall the definition of weak convergence of probability measures 
defined on the space $S'(\mathbb{R}^d)$ of tempered distributions.\\
\index{convergence!weak}
\index{probability measure}

\noindent
\begin{definition} \label{def4.12}
We say that 
a sequence $\{ \gamma_t\},~t\ge 0,$ of probability measures on 
$S'(\mathbb{R}^d)$ {\tt converges weakly}  
to probability measure $\gamma$ on $S'(\mathbb{R}^d)$ if for any function
$f\in C_b(S')$
\begin{equation}\label{LiEq13}
 \lim_{t\rightarrow +\infty} \int_{S'(\mathbb{R}^d)} \,f(x)\,\gamma_t(dx) = 
 \int_{S'(\mathbb{R}^d)} \,f(x)\,\gamma(dx) \;.
\end{equation}
\end{definition}
More general definition on weak convergence of probability measures defined on
topological spaces may be found, e.g.\ in \cite{Bi68} or \cite{KX95}.\vskip1mm

\index{$\nu_t$}
\index{$\mathcal{L}(\tilde{Z}(t))$}
\index{$\tilde{Z}(t)$}
By $\nu_t$ we denote the law 
$\mathcal{L}(\tilde{Z}(t))=\mathcal{N}(0,\Gamma_t)$ of the process 
\begin{equation}\label{LiMe1}
 \tilde{Z}(t):=\int_0^t \mathcal{R}(t-\sigma)\,dW_\Gamma(\sigma), 
 \quad t\ge 0. 
\end{equation}

\index{$\mu_\infty(d\lambda)$}
Let us define 
\begin{equation}\label{LiEq9}
\mu_\infty(d\lambda):=\left[\int_0^\infty\left({\bf s}
(\sigma, a(\lambda) )\right)^2d\sigma\right] \mu(d\lambda).
\end{equation}

Convergence of measures in the distribution sense is a special kind of weak
convergence of measures. This means that 
\begin{equation}\label{LiEq10}
 \int_{\mathbb{R}^d} \varphi(\lambda) d\mu_t(\lambda)
 \stackrel{t\rightarrow +\infty}{\longrightarrow}
 \int_{\mathbb{R}^d} \varphi(\lambda) d\mu_\infty(\lambda)
\end{equation}
for any test function $\varphi\in S({\mathbb{R}^d})$.\medskip

Now, we can formulate the following results.

\begin{lemma}\label{LiLl1} 
Let $\mu_t$ and $\mu_\infty$ be measures defined by (\ref{PizEspectral}) and 
(\ref{LiEq9}), respectively. 
If $\mu_\infty$ is a slowly increasing measure, then the measures 
$\mu_t \rightarrow \mu_\infty$, as $t \rightarrow +\infty$, 
in the distribution sense.
\end{lemma}
\begin{proof} 
First of all, let us notice that, by Theorem \ref{PizTh2}, 
$\mu_t$, $t\geq 0$, are spectral measures
of stationary generalized Gaussian random fields. Moreover, the measures
$\mu_t$, $t\geq 0$, are slowly increasing. Since the function ${\bf
s}(\tau,a(\lambda))$,
$\tau\geq 0$, $\lambda\in \mathbb{R}^d$ is bounded, then the integral
$g_t(\lambda)= \int_0^t ({\bf s}(\tau,a(\lambda)))^2\,d\tau $, for $t<+\infty$,
is bounded, as well. 
In the proof we shall use the
specific form of the measures $\mu_t$, $t\geq 0$, defined by 
(\ref{PizEspectral}).\medskip

We assume that the measure $\mu_\infty$ is slowly increasing, that is,
there exists $k>0$: 
$$\int_{\mathbb{R}^d} (1+|\lambda|^2)^{-k}\, 
 d\mu_\infty(\lambda) 
 = \int_{\mathbb{R}^d} (1+|\lambda|^2)^{-k} \left[
 \int_0^\infty ({\bf s}(t,a(\lambda)))^2\,d\tau\right] d\mu(\lambda) < +\infty
 \;. $$
Hence, the function $g_\infty(\lambda)=
 \int_0^\infty ({\bf s}(\tau,a(\lambda)))^2\,d\tau
\! <\! +\!\infty$ for $\mu$ - almost every~$\lambda$.\medskip

In our case, because of formulae (\ref{PizEspectral}) and (\ref{LiEq9}), we have to
prove the following convergence:
\begin{equation}\label{LiEq11}
 \lim_{t\rightarrow +\infty} \int_{\mathbb{R}^d} \varphi(\lambda)\, g_t(\lambda)
 \,d\mu(\lambda) = \int_{\mathbb{R}^d} \varphi(\lambda)\, g_\infty(\lambda)
 \,d\mu(\lambda) \;,
\end{equation}
where $\varphi\in S({\mathbb{R}^d})$, and 
$g_t$ and $g_\infty$ are as above.\medskip

In other words, the convergence (\ref{LiEq10}) of the measures $\mu_t, ~t\geq
0$, 
to the measure $\mu_\infty$ in the distribution sense, in our case is equivalent
to the weak convergence (\ref{LiEq11}) of functions $g_t, ~t\geq 0$, to the
function $g_\infty$. \medskip

Let us recall that the function ${\bf s}$ determining the measures $\mu_t, t\geq
0$,
and $\mu_\infty$, satisfies the Volterra equation (\ref{LiEq5}) 
(see  Hypothesis~(H)):
$$ {\bf s}(t)+\gamma\int_0^t b(t-\tau){\bf s}(\tau)d\tau = 1\;. $$
Additionally, 
by Lemma 2.1 from \cite{CD97}, $\lim_{t\rightarrow +\infty} {\bf s}(t)=0$.
\medskip

For any $\varphi\in  S({\mathbb{R}^d})$ we have the following estimations
\begin{eqnarray} \label{LiEq12}
&&  \left|  \int_{\mathbb{R}^d}  
 \varphi(\lambda)\, g_t(\lambda)\,d\mu(\lambda)   -  
 \int_{\mathbb{R}^d} \varphi(\lambda)\, g_\infty(\lambda)\,d\mu(\lambda) 
 \right| \leq \nonumber \\
   & \leq &  \int_{\mathbb{R}^d} |\varphi(\lambda)| \,
  |g_t(\lambda)-g_\infty(\lambda)| \,d\mu(\lambda) = \nonumber \\
   & = & \int_{\mathbb{R}^d} |\varphi(\lambda)|\,
 \left|\int_0^t ({\bf s}(\tau,a(\lambda)))^2\,d\tau - 
 \int_0^{+\infty} ({\bf s}(\tau,a(\lambda)))^2\,d\tau \right|
\,d\mu(\lambda)\leq \nonumber \\
  & \leq & \int_{\mathbb{R}^d} |\varphi(\lambda)|\,
 \left(\int_t^{+\infty} ({\bf s}(\tau,a(\lambda)))^2\,d\tau \right)\,
 d\mu(\lambda) \;. 
\end{eqnarray}

The right hand side of (\ref{LiEq12}) tends to zero because
$$h_{\infty}(\lambda)=\int_t^{+\infty} ({\bf s}(\tau,a(\lambda)))^2\,d\tau$$
tends to zero, as $t\rightarrow +\infty$.

Hence, we have proved the convergence (\ref{LiEq11}) which is equivalent to the 
convergence (\ref{LiEq10}) of the measures $\mu_t$, as $t\rightarrow +\infty$, 
to the measure  $\mu_\infty$ in the distribution sense.
\qed\end{proof}

\index{$\Gamma_t$}
\index{$\Gamma_\infty$}
\begin{lemma}\label{LiLl2} 
Let $\Gamma_t,\,\Gamma_\infty$ be covariance kernels of the stochastic
convolution (\ref{LiMe1}) for $t<+\infty$ and $t=+\infty$, respectively, and 
let 
$\mu_t,\;\mu_\infty$ be defined by (\ref{PizEspectral}) and (\ref{LiEq9}).  
Assume that $\mu_\infty$ is a slowly increasing measure on $\mathbb{R}^d$. 
Then $\Gamma_t\rightarrow\Gamma_\infty$, as $t\rightarrow +\infty$,
in the distribution sense if and only if the measures 
$\mu_t \rightarrow\mu_\infty$, for $t\rightarrow +\infty$,
in the distribution sense.
\end{lemma}
\begin{proof}
The sufficiency comes from the convergence of measures in the
distribution sense which, in fact, is a type of weak convergence of measures. 
Actually, the convergence of $\mu_t, \,t\geq 0$, to the measure 
$\mu_\infty$ in the distribution sense means that
$\langle \mu_t,\varphi \rangle 
\stackrel{t\rightarrow +\infty}{\longrightarrow} 
\langle \mu_\infty,\varphi \rangle$
for any $\varphi \in S({\mathbb{R}^d})$.
Particularly, because the Fourier transform acts from $S({\mathbb{R}^d})$ into
$S({\mathbb{R}^d})$, we have 
$\langle \mu_t,{\cal F}(\varphi) \rangle 
\stackrel{t\rightarrow +\infty}{\longrightarrow} 
\langle \mu_\infty,{\cal F}(\varphi) \rangle$
for any $\varphi \in S({\mathbb{R}^d})$.
This is equivalent to the convergence
$\langle {\cal F}^{-1}(\mu_t),\varphi \rangle 
\stackrel{t\rightarrow +\infty}{\longrightarrow} 
\langle {\cal F}^{-1}(\mu_\infty),\varphi) \rangle, 
~\varphi \in S({\mathbb{R}^d})\;.$

This means the convergence of the Fourier inverse transforms of considered 
measures $\mu_t$, as 
$t\rightarrow +\infty$, to the inverse transform of the 
measure $\mu_\infty$ in the distribution sense.

\noindent
Because the measures $\mu_t,\,t\geq 0$, and $\mu_\infty$ are positive, symmetric
and slowly increasing on $\mathbb{R}^d$, then their Fourier inverse transforms
define, by
Bochner-Schwartz theorem, covariance kernels $\Gamma_t={\cal F}^{-1}(\mu_t),
\; t\geq 0$, and $\Gamma_\infty={\cal F}^{-1}(\mu_\infty)$, respectively.
Hence, $\Gamma_t\rightarrow\Gamma_\infty$ as $t\rightarrow +\infty$,
in the distribution sense.

The necessity is the version of L\'evy-Cram\'er's theorem generalized for a 
sequence of slowly increasing measures $\{\mu_t\},\,t\geq 0$, and their Fourier
inverse 
transforms which are their characteristic functionals. \qed 
\end{proof}

Now, we are able to formulate the main results of the section.

\index{limit measure}
\index{$\nu_\infty$}
\begin{theorem}\label{LiTh1}
There exists the limit measure $\nu_\infty = {\cal N}(0,\Gamma_\infty)$, 
the weak limit of the measures $\nu_t = {\cal N}(0,\Gamma_t)$, as 
$t\rightarrow +\infty$, if and only if the measure $\mu_\infty$ defined by
(\ref{LiEq9}) is slowly increasing.
\end{theorem}
\begin{theorem}\label{LiTh2}
Assume that the measure $\mu_\infty$ defined by (\ref{LiEq9}) is slowly
increasing. Then any limit measure of the stochastic Volterra equation 
(\ref{PizEVolterra1}) is of the form 
\begin{equation}\label{LiEq14}
 m_\infty *  {\cal N}(0,\Gamma_\infty) \;,
\end{equation}
where $m_\infty$  is the limit measure for the deterministic version of the
equation (\ref{PizEVolterra1}) with condition (\ref{PizFTr}),
and ${\cal N}(0,\Gamma_\infty)$ is the limit measure 
of the measures $\nu_t$, as $t\rightarrow +\infty$.
\end{theorem}

We would like to emphasize that Theorems \ref{LiTh1} and \ref{LiTh2} have been
formulated
in the spirit analogous to well-known theorems giving invariant measures for
linear 
evolution equations, see e.g.\ \cite{DZ96} or \cite{BJ99}. Such results first
give 
conditions for the existence of invariant measure and next describe all
invariant
measures provided they exist. Our theorems extend, in some sense, Theorem 6.2.1
from 
\cite{DZ96}. Because we consider stochastic Volterra equations we can not study
invariant measures but limit measures.
 
\subsection{Proofs of theorems}
{\bf Proof of Theorem \ref{LiTh1}}

$(\Rightarrow)$
Let us notice that, by Theorem \ref{PizTh2}, the laws 
$\nu_t= {\cal N}(0,\Gamma_t)$, $t\geq 0$, are laws of Gaussian, stationary, 
generalized random fields with the spectral measures  $\mu_t$ 
and the covariances $\Gamma_t$.
The weak convergence (\ref{LiEq13})
is equivalent to the convergence of the characteristic
functionals corresponding to the measures $\nu_t$, $t\geq 0$ and 
$\nu_\infty$, respectively. Particularly
$$
\hat{\nu}_t(\varphi) \stackrel{t\rightarrow +\infty}{\longrightarrow}
\hat{\nu}_\infty(\varphi)\quad \mbox{for any} 
\quad \varphi\in S(\mathbb{R}^d)\,.
$$

We may use the specific form of the characteristic functionals of Gaussian
fields. Namely, we have 
$$
\hat{\nu}_t(\varphi) = \mathbb{E} \, e^{i\langle \tilde{Z}(t),\varphi\rangle} =
\exp \left(-\frac{1}{2}q_t(\varphi,\varphi)\right) = 
\exp \left(-\frac{1}{2}\langle \Gamma_t, \varphi*\varphi_{(s)}\rangle\right),
$$
where $t\geq 0$, $\varphi\in S(\mathbb{R}^d)$ and $\tilde{Z}(t)$ is the stochastic
convolution given by (\ref{LiMe1}).

Analogously
$$ \hat{\nu}_\infty(\varphi) = 
\exp \left(-\frac{1}{2}\langle \Gamma_\infty,
\varphi*\varphi_{(s)}\rangle\right), 
 \quad \varphi\in S(\mathbb{R}^d)\,.
$$
Hence, we have the following convergence
$$
\exp \left(-\frac{1}{2}\langle \Gamma_t, \varphi*\varphi_{(s)}\rangle\right)
\stackrel{t\rightarrow +\infty}{\longrightarrow}
\exp \left(-\frac{1}{2}\langle \Gamma_\infty,
\varphi*\varphi_{(s)}\rangle\right)
$$
for any $\varphi\in S(\mathbb{R}^d)$.

Because $\Gamma_t$, $t\geq 0$, are positive-definite
generalized functions then $\Gamma_\infty$ is a positive-definite generalized 
function, too. So, by Bochner-Schwartz theorem, there exists a slowly
increasing measure
$\mu_\infty$ such that $\Gamma_\infty={\cal F}^{-1}(\mu_\infty)$.
\medskip

$(\Leftarrow)$
Assume that the measure $\mu_\infty$, defined by the formula (\ref{LiEq9}) 
is slowly increasing.
Then, by Bochner-Schwartz theorem, there exists a positive-definite 
distribution $\Gamma_\infty$ on $S$ such that
$\Gamma_\infty = \mathcal{F}^{-1}(\mu_\infty)$ and 
$$ \langle \Gamma_\infty, \varphi \rangle = \int_{\mathbb{R}^d} \varphi(x) 
 d\mu_\infty (x) \;.
$$

Now, we have to show, that $\Gamma_\infty$ is the limit, in the distribution
sense, of the functionals $\Gamma_t$, $t\geq 0$.
In order to do this, by Lemma \ref{LiLl2}, we have to prove the convergence of
the
spectral measures $\mu_t\rightarrow\mu_\infty$, as $t\rightarrow +\infty$,
in the distribution sense.
But, by Lemma \ref{LiLl1}, the measures $\mu_t$, $t\geq 0$, defined by
(\ref{PizEspectral}), 
converge to the measure $\mu_\infty$ in the distribution sense. 
This fact implies, by Lemma \ref{LiLl2}, that
$\Gamma_t\rightarrow\Gamma_\infty$,
as  $t\rightarrow +\infty$, in the distribution sense.

Then, the following convergence 
$$
\exp \left(-\frac{1}{2}\langle \Gamma_t, \varphi*\varphi_{(s)}\rangle\right)
\stackrel{t\rightarrow +\infty}{\longrightarrow}
\exp \left(-\frac{1}{2}\langle \Gamma_\infty,
\varphi*\varphi_{(s)}\rangle\right)
$$
holds for any $\varphi\in S(\mathbb{R}^d)$. This means the convergence
of characteristic functionals of the measures $\nu_t=\mathcal{N} (0,\Gamma_t)$, 
$t\geq 0$, to the characteristic functional of the measure 
$\nu_\infty=\mathcal{N} (0,\Gamma_\infty)$.
Hence, there exists the weak limit $\nu_\infty$ 
of the sequence $\nu_t$, $t\geq 0$, and $\nu_\infty=\mathcal{N}
(0,\Gamma_\infty)$.
 \qed \\[2mm]

\noindent
{\bf Proof of Theorem \ref{LiTh2}}

Consider a limit measure for the stochastic Volterra equation
(\ref{PizEVolterra1})
with the condition (\ref{PizFTr}).
This means that we study a limit distribution of the solution given by 
(\ref{PizE3n}) 
to the considered equation~(\ref{PizEVolterra1}). 

Let us introduce the following notation for distributions, when 
$0\leq t <\infty$:\\
$\eta_t=\mathcal{L} (X(t))$ means the distribution of the solution $X(t)$;
$m_t=\mathcal{L} (\mathcal{R} (t)X_0)$ \linebreak
denotes the distribution of the part
$\mathcal{R}(t)X_0$ of the 
solution $X(t)$ and \linebreak
$\nu_t=\mathcal{L} (\tilde{Z}(t))= \mathcal{L} (\int_0^t
\mathcal{R}(t-\tau)dW_{\Gamma}(\tau))$ is,
as earlier, the distribution of the stochastic convolution
$\tilde{Z}(t)$, that is, $\nu_t =\mathcal{N} (0,\Gamma_t)$.

We assume that $\eta_\infty$ is any limit measure of the stochastic Volterra
equation (\ref{PizEVolterra1}) with the condition (\ref{PizFTr}). 
This means that distributions $\eta_t$ of the solution 
$X(t)$, as $t\rightarrow +\infty$, converge weakly to $\eta_\infty$.

We have to show the formula (\ref{LiEq14}), that is the distribution
$\eta_\infty$ 
has the form $\eta_\infty=m_\infty * \mathcal{N} (0,\Gamma_\infty)$.

The distribution of the solution (\ref{PizE3n}) can be written 
$$\mathcal{L} (X(t))=\mathcal{L} \left( \mathcal{R} (t) X_0+\int_0^t 
\mathcal{R} (t-\tau)dW_\Gamma(\tau) \right) $$
for any $0\leq t<+\infty$.

Because the initial value $X_0$ is independent of the process 
$W_\Gamma(t)$, we have 
$$\mathcal{L} ( X(t))= \mathcal{L} (\mathcal{R} (t) X_0) * \mathcal{L} (\tilde{Z}(t)) $$
or, using the above notation 
\index{$\eta_t$}
$$\eta_t=m_t * \nu_t, \quad \mbox{for any } 0\leq t< +\infty \;.$$
This formula can be rewritten in terms of characteristic functionals of the
above distributions:
\begin{equation}\label{LiEq15}
\hat{\eta}_t(\varphi) = \hat{m}_t(\varphi)\, \hat{\nu}_t(\varphi) \;, 
\end{equation}
\index{$\hat{\eta}_t$}
where $\varphi\in S(\mathbb{R}^d)$ and $0\leq t< +\infty$.

Then, letting in (\ref{LiEq15}) $t$ to tend to $+\infty$, we have 
\index{$\hat{\eta}_\infty$}
$$ \hat{\eta}_\infty (\varphi) = \mathcal{C} (\varphi)\,\hat{\nu}_\infty
(\varphi), 
\quad  \varphi\in S(\mathbb{R}^d) , $$
where $\hat{\eta}_\infty (\varphi)$ is the characteristic functional of the 
limit distribution $\eta_\infty$,  
$\mathcal{C} (\varphi)\!=\! \lim_{t\rightarrow\infty} \hat{m}_t(\varphi)$ and  
$\hat{\nu}_\infty(\varphi)$ 
is the characteristic functional of the limit
measure $\nu_\infty = \mathcal{N} (0,\Gamma_\infty)$; moreover 
$\hat{\nu}_\infty(\varphi)= \exp (-\frac{1}{2}\langle\Gamma_\infty,
\varphi * \varphi_{(s)}\rangle)$.

Now, we have to prove that $\mathcal{C} (\varphi)$ is the characteristic
functional
of the weak limit measure $m_\infty$ of the distributions 
$m_t=\mathcal{L} (\mathcal{R} (t)X_0)$.

In fact, 
$$ \mathcal{C} (\varphi) = \hat{\eta}_\infty(\varphi) \exp \left( \frac{1}{2}
\langle \Gamma_ \infty, \varphi *\varphi_{(s)}\rangle \right) \;,$$
where the right hand side of this formula, as the product of characteristic
functionals, satisfies conditions of the generalized Bochner's theorem 
(see e.g.\ \cite{It84}).
So, using the generalized Bochner's theorem once again, 
there exists a measure $m_\infty$
in $S'(\mathbb{R}^d)$, such that $\mathcal{C} (\varphi)=\hat{m}_\infty$, as
required.
Hence, we have obtained $\eta_\infty= m_\infty * \mathcal{N}
(0,\Gamma_\infty)$. 
 \qed

\subsection{Some special case}\label{LiSpec}
Stochastic Volterra equations have been considered by several authors, see e.g.\
\cite{CD96,CD97,CD00},\cite{CDP97} and \cite{KZ00a}, and are studied 
in connection with problems arising in viscoelasticity. Particularly, in 
\cite{CD97} the heat equation in materials with memory is treated. In that 
paper the authors consider an auxiliary equation of the form
\begin{equation}\label{LiEq16}
z(t) + \int_0^t [\mu\, c(t-\tau) + \beta(t-\tau)]\,z(\tau)\,d\tau =1,
\end{equation}
$t\geq 0$, where $\mu$ is a positive constant and $c, \beta$ are some functions
specified below.

Let us notice that if in the Volterra equation (\ref{LiEq5}) we take 
$b(\tau)=\frac{1}{\gamma}[\mu\, c(\tau) + \beta(\tau)]$, we arrive at the 
equation (\ref{LiEq16}). On the contrary, if we assume in the equation
(\ref{LiEq16}) 
that $\beta(\tau)=0$, $\mu=\gamma$ and $b(\tau)=c(\tau)$, we obtain the equation
(\ref{LiEq5}).

\medskip
Assume, as in \cite{CD97}, the following  {\tt HYPOTHESIS (H1)}: 
\index{HYPOTHESIS (H1)}
\begin{enumerate}
\item Function $\beta$ is nonnegative nonincreasing and integrable on 
  $\mathbb{R}^+$.
\item The constants $\mu,\; c_0$ are positive.
\item There exists a function $\delta \in L^1(\mathbb{R}^+)$ such that:
  $$ c(t) := c_0-\int_0^t \delta(\sigma) d\sigma  \quad  \mbox{and} \quad
   c_\infty := c_0-\int_0^{+\infty} |\delta(\sigma)| d\sigma >0\;. $$    
\end{enumerate}

\begin{proposition}\label{LiPr3} (\cite{CD97}, Lemma 2.1)
Let functions $\beta,\; \delta$ and $c$ be as in Hypothesis (H1). Then the
solution to (\ref{LiEq16}) satisfies:
\begin{enumerate}
\item $0\leq |z(t)| \leq 1, \hspace{2ex} t\geq 0$;
\item $\int_0^{+\infty} |z(t)| dt \leq (\mu\,c_\infty)^{-1} < +\infty$.
\end{enumerate}
\end{proposition}

In the next result we will use the above assumption and Proposition \ref{LiPr3} 
of Cl\'ement and Da Prato and follow the spirit of their argumantation.

\begin{proposition}\label{LiPr4} 
Assume that the stochastic Volterra equation (\ref{PizEVolterra1}) 
has the kernel function $v$ given in the form 
$$b(t) = c_0 - \int_0^t |\delta(\sigma)|d\sigma > 0, \hspace{2ex} c_0>0\;
\quad where \quad \delta\in L^1(\mathbb{R}^+)\;, $$
and the operator $A$ is given by (\ref{PizFTr}).
In this case the limit measure $\mu_\infty$ given by the formula (\ref{LiEq9})
is
a slowly increasing measure. 
\end{proposition}
\begin{proof}
This Proposition is the direct consequence of the definition 
(\ref{LiEq9}) of the
measure $\mu_\infty$ and Proposition \ref{LiPr3}. In fact, 
from Proposition \ref{LiPr3} we have 
\begin{equation} \label{LiEq17}
 \int_0^{+\infty} |{\bf s}(\tau,\gamma)|\, d\tau 
\leq (\gamma \,c_\infty)^{-1}\;,
\end{equation}
where $\gamma$ satisfies Hypothesis (H1), so $c_\infty$ is finite.
Hence, the right hand side of (\ref{LiEq17}) is finite for any finite $\gamma$. 
In our case, because $A$ satisfies (\ref{PizFTr}), $\gamma=a(\lambda)$.

From the definition (\ref{LiEq9}) of the measure $\mu_\infty$ we have:
\begin{equation}\label{LiEq18}
\int_{\mathbb{R}^d} (1+|\lambda|^2)^{-k}\, 
 d\mu_\infty(\lambda) 
 = \int_{\mathbb{R}^d} (1+|\lambda|^2)^{-k} \left[
 \int_0^\infty ({\bf s}(t,a(\lambda)))^2\,d\tau\right] d\mu(\lambda) 
\end{equation}
for $k>0$. 
Let us notice that, by Proposition \ref{LiPr3},
$0\leq |{\bf s}(t,a(\lambda))|\leq 1$ for $t\geq 0$. So, 
$({\bf s}(t,a(\lambda)))^2\leq |{\bf s}(t,a(\lambda))|$.
Therefore, because (\ref{LiEq17}) holds and the measure $\mu$ is slowly
increasing,
the right hand side of (\ref{LiEq18}) is finite. Hence, the measure 
$\mu_\infty$ is slowly increasing, too.
\qed\end{proof}


\section{Regularity of solutions to equations with infinite delay} 
\label{SEEMTDsec:4}
\sectionmark{Regularity of solutions to equations with infinite delay}
\subsection{Introduction and setting the problem}\label{KLPIntro}

In this section 
we consider the following integro-differential stochastic equation
\index{integro-differential stochastic equation}
with infinite delay
\begin{equation}{\label{eq1.1}}
X(t, \theta) = \int_{-\infty}^t b(t-s)[ \Delta X(s, \theta) +
\dot{W}_{\Gamma} (s,\theta)] ds, \quad t
\geq 0, \quad \theta \in T^d,
\end{equation}
where $ b \in L^1 ( \mathbb{R}_+)$, $\Delta$ is the Laplace
operator and $T^d$ is the $d$-dimensional torus. In (\ref{eq1.1}),
$W_{\Gamma}$ is a spatially homogeneous Wiener process with the
space covariance $\Gamma$  taking values in the space of
tempered distributions $\mathcal{S}'(T^d)$ and $\dot{W}_{\Gamma}$ denotes
its partial derivative with respect to the first argument in the sense of
distributions. 
Such equation arises,
in the deterministic case, in the study of heat flow in materials
of fading memory type (see \cite{CD88}, \cite{Nu71}).

In this section we address the following question: under what
conditions on the covariance $\Gamma$ the process $X$ takes values
in a Sobolev space $H^\alpha(T^d)$, particularly in $L^2(T^d)$?

We remark that the knowledge of the regularity of
solutions is important in the study of nonlinear stochastic
equations (see e.g. \cite{DF98} and \cite{MS99}).

We study a particular case of weak solutions under the basis of an
explicit representation of the solution to (\ref{eq1.1}) (cf.
Definition \ref{solution}).

Observe that equation (\ref{eq1.1}) can be viewed as the {\it
limiting equation } for the stochastic Volterra equation
\begin{equation}{\label{eq1.2}}
X(t, \theta) = \int_{0}^t b(t-s)[ \Delta X(s, \theta) +
\dot{W}_{\Gamma} (s,\theta)] ds, \quad t
\geq 0, \quad \theta \in T^d.
\end{equation}
 If $b$ is sufficiently regular, we get, by
differentiating (\ref{eq1.2}) with respect to $t$,

\begin{equation}{\label{eq1.3}}
\frac{\partial X}{\partial t}(t, \theta) =
 b(0)[\Delta X(t,\theta)+ \dot{W}_{\Gamma} 
 (t,\theta)] + \int_0^t b'(t-s)[ \Delta
X(s, \theta) + \dot{W}_{\Gamma} (s,\theta)]
ds,
\end{equation}
where $ t \geq 0$ and $ \theta \in T^d.$\\

Taking in (\ref{eq1.3}), $b(t) \equiv 1$ we obtain
\begin{equation}{\label{eq1.4}}
 \left \{\begin{array}{rcl} \displaystyle
 \frac{\partial X}{\partial t} (t, \theta)
 &=& \Delta X(t,\theta)+ \dot{W}_{\Gamma} (t,\theta), \quad t >0, 
 \theta \in T^d,
\\ \noalign{\smallskip} X(0,\theta) &= & 0, \quad \theta \in
T^d.
\end{array} \right. \end{equation}
Similarly, taking $b(t) = t$ and differentiating
(\ref{eq1.2}) twice with respect to $t$ we obtain

\begin{equation}{\label{eq1.5}}
 \left \{\begin{array}{rcl} \displaystyle
 \frac{\partial^2 X}{\partial t^2}(t, \theta)
 &=& \Delta X(t,\theta)+ \dot{W}_{\Gamma} (t,\theta), \quad t >0, 
 \theta \in T^d,
\\ \noalign{\smallskip} X(0,\theta) &= & 0, \quad \theta \in
T^d, \\ \noalign{\smallskip} \frac{\partial X}{\partial t}
(0,\theta) &= & 0, \quad \theta \in T^d.
\end{array} \right. \end{equation}
It has been shown in \cite[Theorem 5.1]{KZ00} (see also
\cite[Theorem 1]{KZ01}) that equations
 (\ref{eq1.4}) and (\ref{eq1.5}) on the $d$-dimensional
 torus $ T^d$ have
an $ H^{\alpha + 1}( T^d)$-valued solutions if and only if the
Fourier coefficients $(\gamma_n)$ of the space covariance $\Gamma$
of the process $W_\Gamma$ satisfy
\begin{equation}{\label{eq1.6}}
\sum_{n\in \mathbb{Z}^d} \gamma_n (1 + |n|^2)^{\alpha} < \infty.
\end{equation}

Observe that for both, stochastic heat (\ref{eq1.4}) and wave
(\ref{eq1.5}) equations, the conditions are exactly the same,
despite of the different nature of the equations. On the other
hand, the obtained characterization form a natural framework in
which nonlinear heat and wave equations can be studied.

In this section, we will prove that condition (\ref{eq1.6}) even
characterizes $ H^{\alpha + 1}( T^d)$-valued solutions for the
stochastic Volterra equation (\ref{eq1.1}), provided certain
conditions on the kernel $b$ are satisfied. This is a strong
contrast with the deterministic case, where regularity of
(\ref{eq1.1}) is dependent on the kernel $b$. The conditions that
we impose on $b$ are satisfied by a large class of functions.
Moreover, the important example $b(t) = e^{-t}$ is shown to
satisfy our assumptions.

We use, instead of resolvent families, a direct approach to the
equation (\ref{eq1.1}) finding an explicit expression for the
solution in terms of the kernel $b$. This approach reduces the
considered problem to questions in harmonic analysis and lead us
with a complete answer.


 Let $(\Omega,\mathcal{F},(\mathcal{F}_t)_{t\geq 0},P)$ be a complete
filtered probability space. By $T^d$ we denote the $d$-dimensional
torus which can be identified with the product $(-\pi,\pi]^d$. Let
$D(T^d)$ and $D'(T^d)$ denote, respectively, the space of test
functions on $T^d$ and the space of distributions. By
$\langle\xi,\phi\rangle$ we denote the value of a distribution
$\xi$ on a test function. We assume that $W_\Gamma$ is a
$D'(T^d)$-valued spatially homogeneous Wiener process 
with covariance $\Gamma$ which is a positive-definite distribution. 

As we have already written, 
any arbitrary spatially homogeneous Wiener process $W_\Gamma$ is uniquely
determined by its covariance $\Gamma$ according to the formula
\begin{equation}{\label{eq2.1}}
 \mathbb{E} \langle W_\Gamma(t,\theta),\phi\rangle \,
 \langle W_\Gamma(\tau,\theta),\psi\rangle = \mathrm{min}(t,\tau)\,
 \langle \Gamma, \phi\star\psi_{(s)} \rangle \,,
\end{equation}
where $\phi,\psi\in D(T^d)$ and $\psi_{(s)}(\eta)=\psi(-\eta)$,
for $\eta\in T^d$. Because $W_\Gamma$ is spatially homogeneous
process, the distribution $\Gamma=\Gamma(\theta-\eta)$ for
$\theta,\eta\in T^d$.

The space covariance $\Gamma$ , like distribution in $D'(T^d)$,
may be uniquely expanded (see e.g.\ \cite{GW99} or \cite{Sc65} )
into its Fourier series (with parameter $w=1$ because the period
is $2\pi$)
\begin{equation}{\label{eq2.2}}
\Gamma(\theta) = \sum_{n\in \mathbb{Z}^d} e^{i(n,\theta)}
\gamma_n, \quad \theta\in T^d,
\end{equation}
convergent in $D'(T^d)$. In (\ref{eq2.2}),
$(n,\theta)\!=\!\sum_{i=1}^d n_i\theta_i$ and $\mathbb{Z}^d$
denotes the product of integers.

The coefficients $\gamma_n$, in the Fourier series (\ref{eq2.2}),
satisfy:
\index{$\gamma_n$}
\begin{enumerate}
\item $\gamma_n=\gamma_{-n}$ for $n\in \mathbb{Z}^d$,
\item the sequence $(\gamma_n)$ is {\em slowly increasing}, that is
\index{slowly increasing}
\begin{equation}{\label{eq2.3}}
 \sum_{n\in \mathbb{Z}^d} \frac{\gamma_n}{1+|n|^r} < +\infty,
 \mbox{~~for~some~~} r>0.
\end{equation}
\end{enumerate}

Let us introduce, by induction, the following set of indexes.
Denote  \linebreak
$\mathbb{Z}_s^1:=\mathbb{N}$, the set of natural numbers and define
$\mathbb{Z}_s^{d+1}:=(\mathbb{Z}_s^1\times\mathbb{Z}^d)\cup
\{(0,n):n\in\mathbb{Z}_s^{d}\}$. Let us notice that
$\mathbb{Z}^d=\mathbb{Z}_s^d\cup(-\mathbb{Z}_s^d) \cup\{0\}$. For
instance, for $d=2$,
$\mathbb{Z}_s^2=\mathbb{N}\times\mathbb{Z}\cup\{(0,n):n\in\mathbb{Z}\}$.

Now, the spatially homogeneous Wiener process $W_\Gamma$ corresponding to the
covariance $\Gamma$ given by (\ref{eq2.2}) may be represented in
the form
\begin{eqnarray}{\label{eq2.4}}
 W_\Gamma(t,\theta) =
 \sqrt{\gamma_0}\beta_0(t) + \sum_{n\in \mathbb{Z}_s^d}
 \sqrt{2\gamma_n} \left[ \cos (n,\theta)\,\beta_n^1(t) \right.
  \!&\!+\!&\! \left.\sin (n,\theta)\,\beta_n^2(t) \right], \\ &&
 t\geq 0 \mbox{~and~} \theta\in T^d. \nonumber
\end{eqnarray}
In (\ref{eq2.4}) $\beta_0, \beta_n^1, \beta_n^2, n\in
\mathbb{Z}_s^d$, are independent real Brownian motions and
$\gamma_0, \gamma_n$ are coefficients of the series (\ref{eq2.2}).
The series (\ref{eq2.4}) is convergent in the sense of $D'(T^d)$.

Because any periodic distribution with positive period is a
tempered distribution (see, e.g\ \cite{GW99}), we may restrict our
considerations to the space $S'(T^d)$ of tempered distributions.
By $S(T^d)$ we denote the space of infinitely differentiable
rapidly decreasing functions on the torus $T^d$.

Let us denote by $H^\alpha=H^\alpha(T^d)$, $\alpha\in\mathbb{R}$,
the real Sobolev space of order $\alpha$ on the torus $T^d$. The
norms in such spaces may be expressed in terms of the Fourier
coefficients (see \cite{Ad75})
\begin{eqnarray*} 
||\xi||_{H^\alpha} & = & \left(\sum_{n\in\mathbb{Z}^d}(1+|n|^2)^\alpha
 |\xi_n|^2 \right)^\frac{1}{2}\\  & = & \left( |\xi_0|^2 +
 2 \sum_{n\in \mathbb{Z}_s^d} (1+|n|^2)^\alpha \left( (\xi_n^1)^2
 +(\xi_n^2)^2 \right) \right)^\frac{1}{2},
\end{eqnarray*}  
where $\xi_n=\xi_n^1+i\xi_n^2,~\xi_=\bar{\xi}_{-n},~
n\in\mathbb{Z}^d$.

There is another possibility to define the Sobolev spaces (see,
e.g.\ \cite{RS75}). We say that a distribution $\xi\in S'(T^d)$
belongs to $H^\alpha,\alpha\in\mathbb{R}$, if its Fourier
transform $\widehat{\xi}$ is a measurable function and
$$ \int_{T^d} (1+|\lambda|^2)^\alpha\,|\widehat{\xi}(|\lambda |)|^2
 d\lambda < +\infty.
$$

\subsection{Main results}\label{KLPmain}

If $ b \in L^1_{loc} ( \mathbb{R}_+) $ and $ \mu \in \mathbb{C}$,
we shall denote by $ r(t,\mu) $ the unique solution in $
L^1_{loc}( \mathbb{R}_+ ) $ to the linear Volterra equation
\begin{equation}{\label{eq3.1}}
r(t,\mu) = b(t) + \mu \int_0^t b(t-s)r(s,\mu) ds, \quad t \geq 0.
\end{equation}
In many cases the function $r(t,\mu)$ may be found explicitly. For
instance:
$$
\begin{array}{lcl}
 b(t) & \equiv 1, & \quad r(t,\mu) = e^{\mu t} \\
  b(t) & = t,& \quad r(t,\mu) = \displaystyle\frac{\sinh \sqrt{\mu}\,t}
{\sqrt{\mu}} \quad \mu \ne 0 \\ b(t) & = e^{-t},
 & \quad r(t,\mu) = e^{ (1+\mu)t} \\ b(t) & =
 te^{-t}, & \quad r(t,\mu) = e^{-t} \, \displaystyle \frac{\sinh
 \sqrt{\mu}\,t}{\sqrt{\mu}}, \quad \mu \neq 0.
\end{array}
$$
For more examples, see monograph \cite{Pr93} by Pr\"uss.

\noindent Let us denote by $ \hat f(k), k \in \mathbb{Z}$, the
$k$-th Fourier coefficient of an integrable function~$f$:
$$ \hat f(k) = \frac{1}{2\pi} \int_0^{2\pi} e^{ikt}f(t) dt. $$
Given $ b\in L^1 ( \mathbb{R}_+)$, we find that, for $F(t):=
\int_{-\infty}^t b(t-s) f(s) ds$, we have
\begin{equation}{\label{eq3.2}}
\hat F(k) = \tilde b(ik) \hat f(k), \quad k \in \mathbb{Z},
\end{equation}
where $ \tilde b(\lambda) := \displaystyle \int_0^{\infty}
e^{-\lambda t }\, b(t) dt $ denotes the Laplace transform of $b$.
\index{transform!Laplace}

In what follows we will assume that $\tilde b(ik)$ exists for all
$ k\in \mathbb{Z}$ and suppose that $ \lambda \to \tilde
b(\lambda) $ admits an analytical extension to a sector containing
the imaginary axis, and still denote this extension by $\tilde b$.
We introduce the following definition.

\begin{definition} \label{def3.2}
We say that a kernel\index{kernel!admissible} 
$b\in L^1 ( \mathbb{R_+})$ is {\tt admissible} for equation (\ref{eq1.1}) if
$$ \lim_{|n| \to \infty} |n|^2 \int_{0}^{\infty}
[r(s, -|n|^2)]^2 ds =: C_b $$ exists.
\end{definition}

\noindent{\bf Examples}
\begin{enumerate}
\item In the case $ b(t) = e^{-t}$ we obtain $$|n|^2 \int_{0}^{\infty}
[r(s, -|n|^2)]^2 ds = \frac{-|n|^2}{2(1-|n|^2)} $$
and hence $C_b = \frac{1}{2}.$\\

\item In the case $ b(t)= te^{-t}$ we obtain $$|n|^2 \int_{0}^{\infty}
[r(s, -|n|^2)]^2 ds = \frac{|n|^2}{4 + 4|n|^2}$$ and hence $ C_b =
\frac{1}{4}.$
\end{enumerate}

Denote by $W_n(t,\theta):=\cos (n,\theta)\,\beta_n^1(t)+\sin
(n,\theta)\,\beta_n^2(t),~ n\in\mathbb{Z}_s^d $, that is the
$n$-th element in the expansion (\ref{eq2.4}).

\begin{definition}{\label{solution}}
By a solution\index{solution!to stochastic Volterra equation}
$X(t,\theta)$ to the stochastic Volterra equation
(\ref{eq1.1}) we will understand the process of the form
\begin{equation}{\label{eq3.3}}
 X(t,\theta) = \sqrt{\gamma_0}\,\beta_0(t) + \sum_{n\in\mathbb{Z}_s^d}
 \sqrt{2\gamma_n} \int_{-\infty}^t r(t-s,-|n|^2)\,dW_n(s,\theta),
\end{equation}
where the function $r$ is as above, $t\geq 0$ and $\theta\in T^d$.
\end{definition}

The process $X$ given by (\ref{eq3.3}) is a particular form of the
weak solution to the equation (\ref{eq1.1}) (cf.\ \cite{KZ00})
and takes values in the space $S'(T^d)$.\\

The following is our main result.

\begin{theorem}{\label{th3.4}}
Assume $b\in L^1( \mathbb{R_+}) $ is admissible for (\ref{eq1.1}).
Then, the equation (\ref{eq1.1}) has an $ H^{\alpha + 1}(
T^d)$-valued solution if and only if the Fourier coefficients
$(\gamma_n)$ of the covariance $\Gamma$ satisfy
\begin{equation}{\label{eq3.4}}
\sum_{n\in \mathbb{Z}^d} \gamma_n (1 + |n|^2)^{\alpha} < \infty.
\end{equation}
In particular, equation ~(\ref{eq1.1}) has an $L^2(T^d)$-valued
solution if and only if
$$ \sum_{n\in \mathbb{Z}^d} \frac{\gamma_n}{1+|n|^2}< \infty.$$
\end{theorem}

\begin{proof} We shall use the representation (\ref{eq2.4}) for the
Wiener process $W_\Gamma(t,\theta)$ with respect to the basis: $1,
\cos(n,\theta), \sin(n,\theta)$, where $n\in\mathbb{Z}_s^d$ and
$\theta\in T^d$. Equation (\ref{eq1.1}) may be solved
coordinatewise as follows.

Assume that
\begin{equation}\label{eq3.5}
 X(t,\theta)=\sum_{n\in\mathbb{Z}_s^d} [\cos(n,\theta)\,X_n^1(t) +
\sin(n,\theta)\,X_n^2(t)] + X_0(t).
\end{equation}

Introducing (\ref{eq3.5}) into (\ref{eq1.1}), we obtain
\begin{eqnarray*}
\cos(n,\theta)X_n^1(t) &\! + \! &\sin(n,\theta)X_n^2(t) \\
 &\! = \! & -|n|^2 \int_{-\infty}^t \!\!
b(t-s)[\cos(n,\theta)X_n^1(s)  + \sin(n,\theta)X_n^2(s)] ds \\
 && + \sqrt{2\gamma_n} \int_{-\infty}^t \!\!
 b(t-s) [\cos(n,\theta)\,\beta_n^1(s)  +
\sin(n,\theta)\,\beta_n^2(s)] ds,
\end{eqnarray*}
or, equivalently
\begin{eqnarray*}
[\cos(n,\theta),\sin(n,\theta)] \left[\! \begin{array}{l}
X_n^1(t) \\ X_n^2(t) \end{array} \!\right] & \!=\!&
-|n|^2 \int_{-\infty}^t \!\!\!\!\!\! b(t-s)
[\cos(n,\theta),\sin(n,\theta)]
\left[\! \begin{array}{l} X_n^1(s) \\ X_n^2(s)\end{array}\!\right] ds \\
 && + \sqrt{2\gamma_n} \!\!\int_{-\infty}^t \!\!\!\!\!\!
b(t-s)[\cos(n,\theta),\sin(n,\theta)]
\left[\! \begin{array}{l} d\beta_n^1(s) \\
d\beta_n^2(s)\end{array}\!\right] .
\end{eqnarray*}
Denoting
$$ X_n(t):=
\left[ \begin{array}{l} X_n^1(t) \\ X_n^2(t)\end{array}\right]
\quad \mbox{and} \quad
 \beta_n(t):=
 \left[ \begin{array}{l} \beta_n^1(t) \\ \beta_n^2(t)\end{array}\right]
$$
we arrive at the equation
\begin{equation}{\label{eq3.6}}
 X_n(t) =-|n|^2 \int_{-\infty}^t b(t-s)X_n(s)ds + \sqrt{2\gamma_n}
\int_{-\infty}^t b(t-s) d \beta_n(s) \,.
\end{equation}
Taking Fourier transform in $t$, and making use of (\ref{eq3.1})
with $\mu = -|n|^2$ and (\ref{eq3.2}),
  we get the following solution to the equation (\ref{eq3.6}):
$$ X_n(t) = \int_{-\infty}^t r(t-s, -|n|^2) \sqrt{2 \gamma_n}
d \beta_n(s)= \int_0^{\infty} r(s, -|n|^2) \sqrt{2 \gamma_n} d
\beta_n(t-s).
$$
Hence, we deduce the following explicit formula for the solution
to the equation (\ref{eq1.1}):
\begin{equation}{\label{eq3.7}}
\begin{array}{rcl}
X(t,\theta) = \sqrt{\gamma_0} \beta_0(t) &+& \displaystyle \sum_{
n\in \mathbb{Z}_s^d } \sqrt{2 \gamma_n} \left[ \cos( n, \theta)
 \int_{0}^{\infty} r(s, -|n|^2) d\beta_n^1 (t-s) \right. \\ &+& \left.
\sin (n, \theta) \displaystyle \int_{0}^{\infty} r(s, -|n|^2) d
\beta_n^2 (t-s)\right].
\end{array}
\end{equation}
Since the series defining the process $X$ converges in $S'(T^d),
P$-almost surely, it follows from the definition of the space
$H^{\alpha}$ that $ X(t) \in H^{\alpha + 1}, P$-almost surely if
and only if
 \begin{eqnarray}\label{eq3.8}
 \sum_{n\in \mathbb{Z}^d} (1 + |n|^2)^{\alpha +1} \gamma_n &\! \! & \left[ 
 \left(  \int_{-\infty}^{t}\!\! r(t-s, -|n|^2) d \beta_n^1 (s)
  \right)^2 \right.\\
 &\! + \! & \left.\;\; \left( \int_{-\infty}^{t}\!\! r(t-s,-|n|^2) 
 d \beta_n^2 (s) \right)^2 \right]  < \infty. \nonumber
 \end{eqnarray}
Because the stochastic integrals in (\ref{eq3.8}) are independent
Gaussian random variables, we obtain that (\ref{eq3.8}) holds 
$P$-almost surely if and only if
 \begin{eqnarray}\label{eq3.9}
 \sum_{n\in \mathbb{Z}^d} (1 + |n|^2)^{\alpha +1} \gamma_n & &
 \mathbb{E} \left[ \left(\int_{-\infty}^{t} r(t-s, -|n|^2) 
 d \beta_n^1 (s)\right)^2 \right. \\
 &+& \left. \left(\int_{-\infty}^{t} r(t-s, -|n|^2) 
 d \beta_n^2 (s)\right)^2\right] < \infty. \nonumber
 \end{eqnarray}
Or equivalently, using properties of stochastic integrals, if and
only if
\begin{equation}{\label{eq3.10}}
\sum_{n\in \mathbb{Z}^d} (1 + |n|^2)^{\alpha +1} \gamma_n
\int_{0}^{\infty} [ r(s, -|n|^2)]^2 ds < \infty.
\end{equation}
Since $b$ is admissible for the equation (\ref{eq1.1}), we
conclude that (\ref{eq3.10}) holds if and only if
$$
\sum_{n\in \mathbb{Z}^d} (1 + |n|^2)^{\alpha +1}
\frac{\gamma_n}{|n|^2} < \infty,
$$
and the proof is completed.
 \qed\end{proof}

 Concerning uniqueness, we have the following result.

\begin{proposition}\label{pr4.21}
Assume $b\in L^1( \mathbb{R_+}) $ is admissible for (\ref{eq1.1})
and the following conditions hold:

(i)~ \quad $ \displaystyle \sum_{n\in \mathbb{Z}^d}
  \gamma_n (1 + |n|^2)^{\alpha} <\infty,$

 (ii) \quad $ \displaystyle \{1/\tilde b(ik) \}_
 {k\in \mathbb{Z}}
 \subset \mathbb{C} \setminus \{ -|n|^2 : n\in \mathbb{Z}^d \}$.\\

 Then, (\ref{eq1.1}) has a unique $ H^{\alpha + 1}(
T^d)$-valued solution.

\end{proposition}
\begin{proof}
Let $ X(t,\theta)$ be solution of $$ X(t,\theta) =
\int_{-\infty}^t b(t-s) \Delta X(s,\theta) ds.$$ Taking Fourier
transform in $\theta$ and denoting by $X_n(t)$ the $n$-th Fourier
coefficient of $X(t,\theta)$($t$ fixed), we obtain
$$ X_n(t) = -|n|^2 \int_{-\infty}^t b(t-s) X_n(s) ds $$ for all $
n \in \mathbb{Z}^d.$ Taking now Fourier transform in $t$, we get
that the Fourier coefficients of $X_n(t)$ ( $n$ fixed) satisfy
$$( 1 + |n|^2 \tilde b(ik) )\hat{X_n}(k)=0$$ for all $k \in
\mathbb{Z}$. According to (ii) we obtain that $ \hat{X_n}(k) = 0 $
for all $k \in \mathbb{Z}$ and all $ n \in \mathbb{Z}^d.$ Hence,
the assertion follows by uniqueness of the Fourier transform.
\qed\end{proof}

 The following corollaries are an immediate consequence of Theorem \ref{th3.4}. The arguments are
  the same as in \cite{KZ01}. We give here the
 proof for the sake of completeness.

 \begin{corollary}{\label{cor3.6}}
 Suppose $b\in L^1( \mathbb{R_+}) $ is admissible for (\ref{eq1.1}) and
 assume $\Gamma \in L^2 ( T^d). $ Then
 the integro-differential stochastic equation (\ref{eq1.1}) has a solution
 with values in $L^2( T^d)$ for $d=1,2,3.$

 \end{corollary}
 \begin{proof}
 We have to check equation (\ref{eq3.4}) with $\alpha = -1$. Note, that if
$\Gamma \in L^2 (T^d) $ then $\hat
 \Gamma = (\gamma_n) \in l^2 ( \mathbb{Z}^d).$ Consequently
$$
\sum_{n\in \mathbb{Z}^d} \frac{\gamma_n }{1 + |n|^2} \leq (
\sum_{n\in \mathbb{Z}^d} {\gamma_n }^2)^{1/2} ( \sum_{n\in
\mathbb{Z}^d} \frac{1}{(1 + |n|^2)^2})^{1/2}.
$$
But $\sum_{n\in \mathbb{Z}^d} {\gamma_n }^2 < \infty $ and
\begin{equation}{\label{eq3.11}}
\sum_{n\in \mathbb{Z}^d} \frac{1}{(1 + |n|^2)^q} < \infty \mbox{
if and only if } 2q > d.
\end{equation}
Hence, the result follows.
 \qed\end{proof}

 \begin{corollary}{\label{cor3.7}}
 Suppose $b\in L^1( \mathbb{R_+}) $ is admissible for (\ref{eq1.1}) and
 assume $\hat \Gamma \in l^p ( \mathbb{Z}^d) $ for $ 1 < p \leq 2.$ Then
 the integro-differential stochastic equation (\ref{eq1.1}) has a
 solution with values in $L^2( T^d)$ for all $ d <
 \frac{2p}{p-1}.$
\end{corollary}
\begin{proof} Note that
$$
\sum_{n\in \mathbb{Z}^d} \frac{\gamma_n }{1 + |n|^2} \leq (
\sum_{n\in \mathbb{Z}^d} {\gamma_n }^p)^{1/p} ( \sum_{n\in
\mathbb{Z}^d} \frac{1}{(1 + |n|^2)^q})^{1/q},
$$
where $ \frac{1}{p} + \frac{1}{q} = 1.$ Hence the result follows
from (\ref{eq3.11}) with $q= \frac{p}{p-1}.$
\qed\end{proof}

For $\alpha=-1$, the condition (\ref{eq3.4}) can be written as
follows.

\begin{theorem}{\label{th3.8}}
Let $b$ be admissible for the equation (\ref{eq1.1}). Assume that
the covariance $\Gamma$ is not only a positive definite
distribution but is also a non-negative measure. Then the equation
(\ref{eq1.1}) has $L^2(T^d)$-valued solution if and only if
\begin{equation}{\label{eq3.12}}
 (\Gamma,G_d) <+\infty,
\end{equation}
where
\begin{equation}{\label{eq3.13}}
 G_d(x) = \sum_{n\in\mathbb{Z}^d} \int_0^{+\infty}
 \frac{1}{\sqrt{(4\pi t)^d}}\, e^{-t}\, e^{-\frac{|x+2\pi n|^2}{t}}
 dt, \quad x\in T^d.
\end{equation}
\end{theorem}

The proof of Theorem \ref{th3.8} is the same that for Theorem 2,
part 2) in \cite{KZ01}, so we omit it. For more details
concerning the function $G_d$ we refer to \cite{KZ01} and
\cite{La75}.

Additionally, from properties of function $G_d$ defined by
(\ref{eq3.13}) and the condition (\ref{eq3.12}) we obtain the
following result (see \cite[Theorem 6.1]{KZ00}).\\

\begin{corollary}{\label{cor3.9}}
Assume that $\Gamma$ is a non-negative measure and $b$ is
admissible. Then equation
(\ref{eq1.1}) has function valued solutions:\\
~~i) ~~\quad for all $\Gamma$ if $d=1$; \\
~ii) ~~\quad for exactly those $\Gamma$ for which
 $\int_{|\theta|\leq 1} \ln |\theta|\,\Gamma(d\theta)<+\infty$
 if $d=2$; \\
iii) \quad for exactly those $\Gamma$ for which
$\int_{|\theta|\leq 1} \frac{1}{|\theta|^{d-2}}\,
\Gamma(d\theta)<+\infty$ if $d\geq 3$.
\end{corollary}

In what follows, we will see that formula (\ref{eq3.7}) also
provides h\"olderianity of
$X$ with respect to $t$. In order to do that, we need assumptions
very similar to those in \cite{CD96}.\\

\noindent{\tt HYPOTHESIS~(H2)~}
\index{HYPOTHESIS (H2)}
{\it Assume that there exist $\delta\in (0,1)$ and $C_\delta>0$
such that, for all $s \in (-\infty,t)$ we have:\\
(i) $~\quad \displaystyle \int_s^t [r(t-\tau,-|n|^2)]^2 d\tau \leq
 C_\delta |n|^{2(\delta-1)}\,|t-s|^\delta$ ; \\
 (ii) $~\quad \displaystyle \int_{-\infty}^s [ r( t-\tau,-|n|^2)
 -r(s-\tau,-|n|^2)]^2 d\tau \leq C_\delta |n|^{2(\delta-1)}
 \,|t-s|^\delta$.}\\

\begin{proposition}
 Assume that $\displaystyle \sum_{n\in\mathbb{Z}^d}\frac{\gamma_n}
 {1+|n|^2}<+\infty$. Under Hypothesis~(H2), the trajectories of
 the solution $X$ to the equation
 (\ref{eq1.1}) are almost surely $\eta$-H\"older continuous
 with respect to $t$, for every $\eta\in (0,\delta/2)$.
\end{proposition}
\begin{proof}
 From the expansion (\ref{eq3.7}) and properties of stochastic integral,
 we have
 \begin{eqnarray*}
  \mathbb{E} \!&\! &\!  ||X(t,\theta)-X(s,\theta)||^2_{L^2} =
 \mathbb{E}\left|\left| \sqrt{\gamma_0} (\beta_0(t)-\beta_0(s)) 
  \right.\right. \\
 & \!+\!& \sum_{n\in\mathbb{Z}_s^d} \sqrt{2\gamma_n} \left[
 \cos (n,\theta)\!\left(\!
 \int_{-\infty}^t \!\!\!\!\! r(t\!-\!\tau,\!-\!|n|^2) d\beta_n^1(\tau) \!
 -\!\!\int_{-\infty}^s \!\!\!\!\! r(s\!-\!\tau,\!-\!|n|^2)
  d\beta_n^1(\tau)\!\right) \right.\\
 & &\hspace{2ex} 
 +\sin (n,\theta)\left( \left.
 \int_{-\infty}^t \!\!\!\!\! r(t\!-\!\tau,\!-\!|n|^2) d\beta_n^2(\tau)
 -\int_{-\infty}^s \!\!\!\!\! r(s\!-\!\tau,\!-\!|n|^2) d\beta_n^2(\tau)
 \right) \right] ||^2_{L^2} \\
 &\! = \!&  (2\pi)^d \left(\gamma_0 |t-s| \right. \hspace{55ex} \\
  & & \!+ \!\sum_{n\in\mathbb{Z}_s^d}
 2\gamma_n\left[ \int_{-\infty}^s \!\!\!\!\! [r(t\!-\!\tau,\!-\!|n|^2)
 - r(s\!-\!\tau,\!-\!|n|^2)]^2d\tau \!+\!\! 
 \int_s^t \!\!\! r^2(t\!-\!\tau,\!-\!|n|^2)d\tau
 \right] ).
 \end{eqnarray*}

According to assumptions (i) and (ii) of the Hypothesis~(H2), we
have
$$ \mathbb{E}||X(t,\theta)-X(s,\theta)||^2_{L^2} \leq C_\delta
\sum_{n\in\mathbb{Z}_s^d} 2\gamma_n \,|n|^{2(\delta-1)}
 |t-s|^\delta . $$
Because $X$ is a Gaussian process, then for any $m\in N$, there
exists a constant $C_m>0$ that
$$\mathbb{E}||X(t,\theta)-X(s,\theta)||^{2m}_{L^2} \leq C_m
\left[ C_\delta \sum_{n\in\mathbb{Z}_s^d} 2\gamma_n
\,|n|^{2(\delta-1)} \right]^m |t-s|^{m\delta} . $$ Taking $m$ such
that $m\delta>1$ and using the Kolmogorov test, we see that the
solution $X(t,\theta)$ is $\eta$-H\"older continuous, with respect
to $t$, for $\eta=\delta/2-1/(2m)$. \qed\end{proof}
{\bf Example}
 Let us consider the particular case $b(t) = e^{-t}, ~t\geq 0$.
 Then, by previous considerations,
 $r(t,-|n|^2) = e^{(1-|n|^2)\,t}$. One can check that in this
 case the Hypothesis~(H2) is fulfilled.\\

\noindent
{\bf Remark}  We observe that the condition (i) in Hypothesis~(H2)
is the same as
$$
|n|^2 \int_0^{t} [r(s, -|n|^2)]^2 ds \leq C_{\delta} |n|^{2\delta}
|t|^{\delta}
$$
and hence it is nearly equivalent to say that the function $b$ is
admissible.


\backmatter

%
%
%
%
\printindex

\end{document}